\documentclass{amsart}
\usepackage{amssymb, amsbsy, amsthm, amsmath, amstext, amsopn, verbatim}
\usepackage[all]{xy}
\usepackage{amsfonts}
\usepackage{amscd}
\usepackage{mathrsfs}
\usepackage{verbatim}
\usepackage[usenames]{color}

\newtheorem{thm}{Theorem} [section]
\newtheorem{lemma}[thm]{Lemma}
\newtheorem{lem}[thm]{Lemma}
\newtheorem{sublemma}[thm]{Sub-Lemma}
\newtheorem{corollary}[thm]{Corollary}
\newtheorem{prop}[thm]{Proposition}
\newtheorem{notation}[thm]{Notation}

\newtheorem*{basic assumption}{Basic Assumption}

\newtheorem*{Kashiwara}{Kashiwara's Equivalence}

\theoremstyle{definition}

\newtheorem*{principal example}{Main Example}

\newtheorem{defn}[thm]{Definition}

\newtheorem{example}[thm]{Example}

\newtheorem{assumptions}[thm]{Assumption}
\newtheorem{terminology}[thm]{Terminology}

\theoremstyle{remark}

\newtheorem{remark}[thm]{Remark}
                         
\newtheorem{claim}[thm]{Claim}

\begin{document}

\numberwithin{equation}{section}

\newcommand{\hs}{\mbox{\hspace{.4em}}}
\newcommand{\bd}{\begin{displaymath}}
\newcommand{\ed}{\end{displaymath}}
\newcommand{\bcd}{\begin{CD}}
\newcommand{\ecd}{\end{CD}}

\newcommand{\proj}{\operatorname{Proj}}
\newcommand{\bproj}{\underline{\operatorname{Proj}}}
\newcommand{\spec}{\operatorname{Spec}}
\newcommand{\bspec}{\underline{\operatorname{Spec}}}
\newcommand{\pline}{{\mathbf P} ^1}
\newcommand{\pplane}{{\mathbf P}^2}
\newcommand{\coker}{{\operatorname{coker}}}
\newcommand{\ldb}{[[}
\newcommand{\rdb}{]]}

\newcommand{\Sym}{\operatorname{Sym}^{\bullet}}
\newcommand{\Symp}{\operatorname{Sym}}
\newcommand{\Pic}{\operatorname{Pic}}
\newcommand{\AAut}{\operatorname{Aut}}
\newcommand{\PAut}{\operatorname{PAut}}

\newcommand{\too}{\twoheadrightarrow}
\newcommand{\C}{{\mathbb C}}
\newcommand{\cA}{{\mathcal A}}
\newcommand{\cS}{{\mathcal S}}
\newcommand{\cV}{{\mathcal V}}
\newcommand{\cM}{{\mathcal M}}
\newcommand{\bA}{{\mathbf A}}
\newcommand{\cB}{{\mathcal B}}
\newcommand{\cC}{{\mathcal C}}
\newcommand{\cD}{{\mathcal D}}
\newcommand{\D}{{\mathcal D}}
\newcommand{\boldc}{{\mathbf C}}
\newcommand{\cE}{{\mathcal E}}
\newcommand{\cF}{{\mathcal F}}
\newcommand{\cG}{{\mathcal G}}
\newcommand{\G}{{\mathbf G}}
\newcommand{\fg}{{\mathfrak g}}
\newcommand{\bH}{{\mathbf H}}
\newcommand{\cH}{{\mathcal H}}
\newcommand{\cI}{{\mathcal I}}
\newcommand{\cJ}{{\mathcal J}}
\newcommand{\cK}{{\mathcal K}}
\newcommand{\cL}{{\mathcal L}}
\newcommand{\baL}{{\overline{\mathcal L}}}
\newcommand{\M}{{\mathcal M}}
\newcommand{\bM}{{\mathbf M}}
\newcommand{\bm}{{\mathbf m}}
\newcommand{\cN}{{\mathcal N}}
\newcommand{\ie}{\textit{i.e.}}
\newcommand{\theo}{\mathcal{O}}
\newcommand{\cP}{{\mathcal P}}
\newcommand{\cR}{{\mathcal R}}
\newcommand{\boldp}{{\mathbf P}}
\newcommand{\boldq}{{\mathbf Q}}
\newcommand{\bbL}{{\mathbf L}}
\newcommand{\cQ}{{\mathcal Q}}
\newcommand{\cO}{{\mathcal O}}
\newcommand{\Oo}{{\mathcal O}}
\newcommand{\OX}{{\Oo_X}}
\newcommand{\OY}{{\Oo_Y}}
\newcommand{\dd}{\D}
\newcommand{\Hamp}{\mathbb{H}^{\perp}}
\newcommand{\Ham}{\mathbb{H}}
\newcommand{\cX}{{\mathcal X}}
\newcommand{\cW}{{\mathcal W}}
\newcommand{\rad}{{\mathrm rad}}
\newcommand{\cZ}{{\mathcal Z}}
\newcommand{\qgr}{\operatorname{qgr}}
\newcommand{\gr}{\operatorname{gr}}
\newcommand{\coh}{\operatorname{coh}}
\newcommand{\End}{\operatorname{End}}
\newcommand{\Hom}{\operatorname{Hom}}
\newcommand{\uHom}{\underline{\operatorname{Hom}}}
\newcommand{\uHomY}{\uHom_{\OY}}
\newcommand{\uHomX}{\uHom_{\OX}}
\newcommand{\Ext}{\operatorname{Ext}}
\newcommand{\bExt}{\operatorname{\bf{Ext}}}
\newcommand{\Tor}{\operatorname{Tor}}

\newcommand{\inv}{^{-1}}
\newcommand{\airtilde}{\widetilde{\hspace{.5em}}}
\newcommand{\airhat}{\widehat{\hspace{.5em}}}
\newcommand{\nt}{^{\circ}}
\newcommand{\del}{\partial}

\newcommand{\supp}{\operatorname{supp}}
\newcommand{\GK}{\operatorname{GK-dim}}
\newcommand{\W}{W}
\newcommand{\id}{\operatorname{id}}
\newcommand{\res}{\operatorname{res}}
\newcommand{\lrar}{\leadsto}
\newcommand{\im}{\operatorname{Im}}
\newcommand{\HH}{\operatorname{H}}
\newcommand{\Coh}[1]{#1\text{-}{\mathsf{coh}}}
\newcommand{\QCoh}[1]{#1\text{-}{\mathsf{qcoh}}}
\newcommand{\PCoh}[1]{#1\text{-}{\mathsf{procoh}}}
\newcommand{\Good}[1]{#1\text{-}{\mathsf{good}}}
\newcommand{\QGood}[1]{#1\text{-}{\mathsf{Qgood}}}
\newcommand{\Hol}[1]{#1\text{-}{\mathsf{hol}}}
\newcommand{\Reghol}[1]{#1\text{-}{\mathsf{reghol}}}
\newcommand{\Bun}{\operatorname{Bun}}
\newcommand{\Hilb}{\operatorname{Hilb}}
\newcommand{\pa}{\partial}
\newcommand{\F}{\mathcal{F}}
\newcommand{\nthord}{^{(n)}}
\newcommand{\Aut}{\underline{\operatorname{Aut}}}
\newcommand{\Gr}{\operatorname{\bf Gr}}
\newcommand{\Fr}{\operatorname{Fr}}
\newcommand{\GL}{\operatorname{GL}}
\newcommand{\gl}{\mathfrak{gl}}
\newcommand{\SL}{\operatorname{SL}}
\newcommand{\ff}{\footnote}
\newcommand{\ot}{\otimes}
\newcommand{\Wx}{\mathcal W_{\mathfrak X}}
\newcommand{\gh}{\text{gr}_{\hbar}}
\newcommand{\ig}{\iota_g}
\def\Ext{\operatorname {Ext}}
\def\Hom{\operatorname {Hom}}
\def\Ind{\operatorname {Ind}}
\def\bbZ{{\mathbb Z}}

\newcommand{\nc}{\newcommand}
\newcommand{\on}{\operatorname}
\nc{\cont}{\on{cont}}
\nc{\rmod}{\on{mod}}
\nc{\Mtil}{\widetilde{M}}
\nc{\wb}{\overline}
\nc{\wt}{\widetilde}
\nc{\wh}{\widehat}
\nc{\sm}{\setminus}
\nc{\mc}{\mathcal}
\nc{\mbb}{\mathbb}
\nc{\Mbar}{\wb{M}}
\nc{\Nbar}{\wb{N}}
\nc{\Mhat}{\wh{M}}
\nc{\pihat}{\wh{\pi}}
\nc{\opp}{\mathrm{opp}}
\nc{\phitil}{\wt{\phi}}
\nc{\Qbar}{\wb{Q}}
\nc{\DYX}{\D_{Y\leftarrow X}}
\nc{\DXY}{\D_{X\to Y}}
\nc{\dR}{\stackrel{\bbL}{\underset{\D_X}{\ot}}}
\nc{\Winfi}{\cW_{1+\infty}}
\nc{\K}{{\mc K}}
\nc{\unit}{{\bf \on{unit}}}
\nc{\boxt}{\boxtimes}
\nc{\xarr}{\stackrel{\rightarrow}{x}}
\nc{\Cnatbar}{\overline{C}^{\natural}}
\nc{\oJac}{\overline{\on{Jac}}}
\nc{\gm}{{\mathbf G}_m}
\nc{\Loc}{\on{Loc}}
\nc{\Bm}{\operatorname{Bimod}}
\nc{\lie}{{\mathfrak g}}
\nc{\lb}{{\mathfrak b}}
\nc{\lien}{{\mathfrak n}}
\nc{\E}{\mathcal{E}}
\nc{\Cs}{\mathbb{G}_m}
\nc{\hol}{\mathrm{hol}}
\nc{\can}{\mathrm{can}}
\nc{\rat}{\mathrm{rat}}
%\nc{\wh}{\widehat}

\nc{\Gm}{{\mathbb G}_m}
\nc{\Gabar}{\wb{\G}_a}
\nc{\Gmbar}{\wb{\G}_m}
\nc{\PD}{{\mathbb P}_{\D}}
\nc{\Pbul}{P_{\bullet}}
\nc{\PDl}{{\mathbb P}_{\D(\lambda)}}
\nc{\PLoc}{\mathsf{MLoc}}
\nc{\Tors}{\on{Tors}}
\nc{\PS}{{\mathsf{PS}}}
\nc{\PB}{{\mathsf{MB}}}
\nc{\Pb}{{\underline{\operatorname{MBun}}}}
\nc{\Ht}{\mathsf{H}}
\nc{\bbH}{\mathbb H}
\nc{\gen}{^\circ}
\nc{\Jac}{\operatorname{Jac}}
\nc{\sP}{\mathsf{P}}
\nc{\otc}{^{\otimes c}}
\nc{\Det}{\mathsf{det}}
\nc{\PL}{\on{ML}}

\nc{\ml}{{\mathcal S}}
\nc{\Xc}{X_{\on{con}}}
\nc{\sgood}{\text{strongly good}}
\nc{\Xs}{X_{\on{strcon}}}
\nc{\resol}{\mathfrak{X}}
\nc{\map}{\mathsf{f}}
\nc{\tor}{\mathrm{tor}}
\nc{\base}{Z}
\nc{\bigvar}{\mathsf{W}}
\nc{\alg}{\mathsf{A}}
\nc{\T}{\mathsf{T}}
\nc{\qcoh}{\on{qcoh}}
\renewcommand{\o}{\otimes}
\nc{\mf}{\mathfrak}
\nc{\NN}{\mathsf{N}}
\nc{\h}{\hbar}
\nc{\ms}{\mathscr}
\nc{\good}{\mathrm{good}}
\newcommand{\LMod}[1]{#1\text{-}{\mathsf{Mod}}}
\newcommand{\Lmod}[1]{#1\text{-}{\mathsf{mod}}}
\nc{\mbf}{\mathbf}
\nc{\ad}{\mathrm{ad}}
\nc{\Rees}{\mathsf{Rees}}
\nc{\Supp}{\mathrm{Supp}}
\nc{\Z}{\mathbb{Z}}
\nc{\N}{\mathbb{N}}
\nc{\ann}{\mathrm{ann}}
\nc{\Blf}{B_{\mathrm{l.f.}}}
\nc{\bx}{\mathbf{x}}
\nc{\by}{\mathbf{y}}
\nc{\bz}{\mathbf{z}}
\nc{\bw}{\mathbf{w}}
\nc{\idot}{*}
\nc{\Der}{\mathrm{Der}}
\nc{\ds}{\dots}
\nc{\CatCs}{\ms{C}}
\renewcommand{\mod}{ \ \mathrm{mod} \ }
\nc{\sA}{\mc{A}} %This is the sheaf of DQ-modules
\nc{\A}{A} %This is the algebra of DQ-modules
\nc{\B}{B} %This is the quatization of C[S]
\nc{\sW}{\mathscr{W}} %This is the sheaf of W-algebras
\nc{\rh}{\mathrm{r.h.}}
\nc{\sT}{\mathsf{T}}
\nc{\SF}{E}
\nc{\limn}{\displaystyle\lim_{\substack{\longleftarrow\\n}}\,}
\nc{\lims}{\displaystyle\lim_{\substack{\longleftarrow\\s}}\,}
\nc{\Cone}{\mathrm{Cone}}
\nc{\Ker}{\mathrm{Ker}}
\nc{\cech}{\check{C}}
\nc{\loc}{\mathrm{loc}}
\nc{\Id}{\mathrm{Id}}

\title[Categorical Cell Decomposition for Quantum Symplectic Varieties]{Categorical Cell Decomposition of Quantized Symplectic Algebraic Varieties}
\author{Gwyn Bellamy}
\address{School of Mathematics and Statistics\\University of Glasgow\\Glasgow,  Scotland\\G12 8QW}
\email{gwyn.bellamy@glasgow.ac.uk}
\author{Christopher Dodd}
\address{Department of Mathematics\\University of Toronto\\Toronto, ON, Canada}
\email{cdodd@math.toronto.edu}
\author{Kevin McGerty}
\address{Mathematical Institute\\University of Oxford\\Oxford, England, UK}
\email{mcgerty@maths.ox.ac.uk}
\author{Thomas Nevins}
\address{Department of Mathematics\\University of Illinois at Urbana-Champaign\\Urbana, IL 61801 USA}
\email{nevins@illinois.edu}

\begin{abstract}
We prove a new symplectic analogue of Kashiwara's Equivalence from $\D$-module theory.  As a consequence,
we establish a structure theory for module categories over deformation quantizations that mirrors, at a higher categorical level, the Bia\l ynicki-Birula stratification of a variety with an action of the multiplicative group $\Gm$.  The resulting {\em categorical cell decomposition} 
provides an algebro-geometric parallel to the structure of Fukaya categories of Weinstein manifolds.
From it, we derive concrete consequences for invariants such as $K$-theory and Hochschild homology of module categories of interest in geometric representation theory.  
\end{abstract}

\maketitle

\newcommand{\Qcoh}{\mathsf{Qcoh}}
\newcommand{\WQcoh}{\Qcoh(\cW)}
\section{Introduction}
Since the 1970s, categories of (ordinary or twisted) $\D$-modules on algebraic varieties and stacks have become fundamental tools in geometric representation theory; see Beilinson and Bernstein \cite{BB}.  More recently, an emerging body of important work in geometric representation theory relies on sheaves over deformation quantizations of symplectic algebraic varieties more general than the cotangent bundles whose deformation quantizations give rise to $\D$-modules; see Bellamy and Kuwabara \cite{BKu}, Bezrukavnikov and Kaledin \cite{BK}, Bezrukavnikov and Losev \cite{BL},  Braden, Proudfoot and Webster \cite{BLPW2}, Dodd and Kremnitzer \cite{DKr}, Gordon and Losev \cite{GordonLosev}, Kaledin \cite{Kaledin},
Kashiwara and Rouquier \cite{KR} and McGerty and Nevins \cite{KTDerived}. A sophisticated theory of such quantizations now exists thanks to the efforts of many (see Bezrukavnikov and Kaledin \cite{BK2}, D’Agnolo and Kashiwara \cite{DKa}, D’Agnolo and Polesello \cite{DP}, D’Agnolo and Schapira  \cite{DS}, Kashiwara  \cite{Ks}, Kashiwara and Schapira  \cite{KS,KSDQ} and Nest and Tsygan  \cite{NT,NT2} among many others).

The present paper establishes a structure theory for deformation quantizations that mirrors, at a higher categorical level, the fundamental Bia\l ynicki-Birula stratification of a variety with an action of the multiplicative group $\Gm$ and the corresponding decomposition of its cohomology.  In the most prominent examples, the resulting {\em categorical cell decomposition} has many immediate and concrete consequences for invariants such as $K$-theory and Hochschild homology; it also makes possible the extension of powerful tools from $\D$-module theory, such as the Koszul duality relating $\D$-modules to dg modules over the de Rham complex, to a more general symplectic settings; see Bellamy, Dodd, McGerty and Nevins \cite{BDMN}.  The structures that we identify parallel those described for Fukaya--type categories in real symplectic geometry by Nadler \cite{Nadler}.  We derive these structures on module categories from a new symplectic analogue of Kashiwara's Equivalence for $\D$-modules.  

In Section \ref{intro:properties} we describe an enhancement of the Bia\l ynicki-Birula decomposition for symplectic varieties with a nice $\Gm$-action.  Section \ref{intro:categoricalcell} explains our categorical cell decomposition for sheaves on the quantizations of such varieties; in Section \ref{intro:kashiwara}, we lay out the symplectic Kashiwara Equivalence that underlies categorical cell decomposition.  Section \ref{sec:intro-cats} describes basic categorical consequences.  Section \ref{intro:applications} provides immediate applications of this structure theory for module categories of deformation quantizations.  Section \ref{intro:fukaya} explores parallels with Fukaya categories.  
\subsection{Symplectic Varieties with Elliptic $\Gm$-Action}\label{intro:properties}
We work throughout the paper over $\C$.  Let $\resol$ be a smooth, connected symplectic algebraic variety with symplectic form $\omega$.  
\begin{defn}\label{elliptic defn}
A $\Gm$-action on $\resol$ is said to be {\em elliptic} if the following hold.
\begin{enumerate}
\item $\Cs$ acts with positive weight on the symplectic form: $m_t^*\omega = t^{l} \omega$ for some $l >0$.  
\item For every $x\in \resol$, the limit $\displaystyle\lim_{t\rightarrow \infty} t \cdot x$ exists in $\resol$. 
\end{enumerate}
\end{defn}
\noindent
We remark that if we assume that $\omega$ is rescaled by $\Gm$ with some weight $l\in\mathbb{Z}$, then the existence of limits already implies that $l \geq 0$.  

Write $\resol^{\Gm} = \coprod Y_i$, a union of smooth connected components.  For each $i$, let 
\begin{displaymath}
\displaystyle C_i = \{p\in\resol : \lim_{t\rightarrow\infty} t\cdot p\in Y_i\};
\end{displaymath}
 these subsets are the Morse-theoretic attracting loci for the elliptic $\Gm$-action.  Note that $\resol = \coprod C_i$  by condition (2) of Definition \ref{elliptic defn}.  
 
 Recall that if $i:C\hookrightarrow\resol$ is a smooth coisotropic subvariety, a {\em coisotropic reduction} of $C$ consists of a smooth symplectic variety $(S,\omega_S)$ and a morphism $\pi:C\rightarrow S$  for which $\omega |_C = \pi^*\omega_S$.  
We establish a basic structural result for the decomposition $\resol = \coprod C_i$ in Section \ref{sec:symplecticvarietiessection}:
\begin{thm}[Theorem \ref{thm:maincoiso}]\label{thm:main3}
\mbox{}
\begin{enumerate}
\item Each $C_i$ is a smooth, coisotropic subvariety of $\resol$ and a $\Cs$-equivariant affine bundle over the fixed point set $Y_i$. 
\item There exist symplectic manifolds $(S_i,\omega_i)$ with elliptic $\Gm$-action and $\Cs$-equivariant coisotropic reductions $\pi_i : C_i \rightarrow S_i$.
\end{enumerate}
\end{thm}
Part (1) of the theorem is a symplectic refinement of the Bia\l ynicki-Birula stratification \cite{BBFix} arising from a $\Gm$-action.  Our proof of assertion (2) relies on formal local normal forms for symplectic varieties in the neighborhood of a coisotropic subvariety, which we develop in Section \ref{sec:symplecticvarietiessection}. 

We provide a refined description of the symplectic quotients $S_i$ and corresponding $\Gm$-equivariant affine fibrations $S_i\rightarrow Y_i$ of Theorem \ref{thm:main3}.  We need two definitions.  First, let $Y$ be a smooth connected variety. A \textit{symplectic fibration} over $Y$ is a tuple $(\SF,\eta, \{ - , - \})$, where $\eta : \SF \rightarrow Y$ is an affine bundle and $\{ - , - \}$ a $\mc{O}_Y$-linear Poisson bracket on $\SF$ such that the restriction of $\{ - , - \}$ to each fiber of $\eta$ is non-degenerate. The symplectic fibration is said to be {\em elliptic} if $\Cs$ acts on $\SF$ such that $\{ - , - \}$ is homogeneous of negative weight, $Y = \SF^{\Cs}$ and all weights of $\Cs$ on the fibers of $\eta$ are negative. 

Second, note that $T^* Y$ is naturally a group scheme over $Y$.   Suppose $p:B\rightarrow Y$ is a smooth variety over $Y$ equipped with a symplectic form $\omega_B$. Suppose that $B$ is equipped with an action $a: T^*Y\times_Y B\rightarrow B$ of the group scheme $T^*Y$ over $Y$. We say $B\rightarrow Y$ is {\em symplectically automorphic} if, for any $1$-form $\theta$ on $Y$, 
 we have 
 $a(\theta, \cdot)^*\omega_B  = \omega_B + p^*d\theta$.
In the special case that $B$ is a $T^*Y$-torsor, $B$ is thus a {\em twisted cotangent bundle} in the sense of \cite{BB}.

\begin{thm}[Theorem \ref{cor:sympafffactor}]\label{thm:mainsymplfib}
Keep the notation of Theorem \ref{thm:main3}(2).  Then for each $i$:
\begin{enumerate}
\item  The fibration $S_i\rightarrow Y_i$ comes equipped with a free $T^*Y_i$-action making $S_i$ symplectically automorphic over $Y_i$.  
\item The quotient $E_i := S_i / T^* Y_i$ inherits a Poisson structure, making 
$E_i\rightarrow Y_i$ into an elliptic symplectic fibration. 
\end{enumerate}
\end{thm}

Locally in the Zariski topology on $Y_i$, $S_i \simeq T^{*} Y_i \times_{Y_i} \SF_i$, as smooth varieties with $\Cs$-actions. 

\subsection{Categorical Cell Decomposition}\label{intro:categoricalcell}
We next turn to deformation quantizations.  Suppose that $\resol$ is a smooth symplectic variety with elliptic $\Gm$-action.  Let $\cA$ be a $\Gm$-equivariant deformation quantization of $\mathcal{O}_{\resol}$; this is a $\Gm$-equivariant sheaf of flat $\C[\![\hbar]\!]$-algebras, where $\Gm$ acts with weight $l$ on $\hbar$, for which $\cA/\hbar\cA$ is isomorphic, as a sheaf of $\Gm$-equivariant Poisson algebras, to $\mathcal{O}_{\resol}$; see Section \ref{sec:DQsheaf} for more details.  Let $\cW = \cA[\hbar\inv]$.  There is a natural analogue for $\cW$ of coherent sheaves on $\resol$, the category of $\Gm$-equivariant {\em good} $\cW$-modules, denoted by $\cW\operatorname{-}\mathsf{good}$; see Section \ref{sec:DQmodules} for details.  
\begin{defn}
The category of {\em quasicoherent $\cW$-modules} is $\WQcoh := \on{Ind}\big(\cW\operatorname{-}\mathsf{good}\big)$.
\end{defn}  

For each subcollection of the coisotropic attracting loci $C_i$ of Section \ref{intro:properties} whose union $C_K$, with $K \subset \{1, \ds, k \}$, is a closed subset of $\resol$, we let $\WQcoh_K$ denote the full subcategory of $\WQcoh$ whose objects are supported on $C_K$.  By Lemma \ref{lem:partial order}, the loci $C_i$ are naturally partially ordered. Refining to a total order, for each $i$ there are closed subsets $C_{K_{\geq i}} = \cup_{j\geq i} C_j$ and $C_{K_{>i}} = \cup_{j>i} C_j$.

\begin{thm}[Theorem \ref{thm:equivmain1} and Corollary \ref{cor:filter}]\label{thm:main2}
\mbox{}
\begin{enumerate}
\item 
The category $\WQcoh$ is filtered by localizing subcategories $\WQcoh_{K_{\geq i}}$.   
\item Each subquotient $\WQcoh_{K_{\geq i}}/\WQcoh_{K_{>i}}$ is equivalent to the category of quasicoherent modules over a deformation quantization of the symplectic quotient $S_i$.  
\end{enumerate}
\end{thm}
Mirroring the structure of $S_i$ in Theorem \ref{thm:mainsymplfib}, the category of $\Gm$-equivariant quasicoherent modules over a  deformation quantization of $S_i$ is equivalent to the category of modules for a specific type of algebra.  We describe this relationship explicitly in Sections \ref{sec:TDO} and \ref{sec:TDO2}, and in Theorem \ref{thm:main1b} below.  In particular, in the special case when
$Y_i$ is an isolated fixed point, we obtain:
\begin{corollary}
Suppose the fixed-point set $\resol^{\Gm}$ is finite.  Then each subquotient 
$$\WQcoh_{K_{\geq i}}/\WQcoh_{K_{>i}}$$ 
is equivalent to the category of modules over the Weyl algebra $\D(\mathbb{A}^{t_i})$ for $t_i = \frac{1}{2}\dim S_i$.  
\end{corollary}

The Weyl algebra $\D(\mathbb{A}^{t_i})$ is a quantization of the algebra of functions on an ``algebraic cell'' $\mathbb{A}^{2t_i}$.  Moreover, this category, although complicated, looks contractible from the point of view of certain fundamental invariants, for example algebraic $K$-theory and Hochschild/cyclic homology.  Thus, we view the category $\D(\mathbb{A}^{t_i})-\mathsf{mod}$ as a ``categorical algebraic cell,'' parallel to the way that the category $\on{Vect}_\C$ of finite-dimensional complex vector spaces is a categorical analogue of a topological cell.  
In particular, in the case when the fixed-point set $\resol^{\Gm}$ is finite, we interpret the filtration of $\WQcoh$ provided by Theorem \ref{thm:main2} as providing 
a {\em categorical cell decomposition} of  $\WQcoh$.  Building the category $\WQcoh$ from algebraic cells is thus a ``bulk analogue'' of the process of building a quasi-hereditary category from ``categorical topological cells,'' \ie copies of $\on{Vect}_\C$.  
In particular, to the extent that categories of the form $\WQcoh$ undergird many representation-theoretic settings of intense recent interest, our categorical cell decompositions are a basic structural feature of the ``big'' geometric categories that arise in representation theory.  

We expect Theorem \ref{thm:main2} to have many consequences for $\WQcoh$ and related algebraic categories from representation theory.  One such
application will appear in Bellamy, Dodd, McGerty and Nevins \cite{BDMN}: a symplectic analogue of the Koszul duality, sometimes called ``$\D-\Omega$-duality,'' between $\D$-modules on a smooth variety $X$ and dg modules over the de Rham complex $\Omega_X$ of $X$; see Kapranov \cite{Kap}.  More precisely, the Koszul duality of \cite{Kap} generalizes, to arbitrary coherent $\D$-modules, the Riemann-Hilbert correspondence between regular holonomic $\D$-modules and their associated de Rham complexes, which are constructible complexes on $X$.  Since $\Omega_X$ sheafifies over $X$, embedded as the zero section of $T^*X$, one can view this correspondence as a categorical means of sheafifying the category of $\D_X$-modules over $X$.  Such a sheafification is tautologous in the $\D$-module setting, but becomes less so in a general symplectic setting.  Namely, in \cite{BDMN}, starting from a {\em bionic symplectic variety}---a symplectic variety $\resol$ with {\em both} an elliptic $\Gm$-action and a commuting Hamiltonian $\Gm$-action defining a good {\em Lagrangian skeleton $\mathsf{\Lambda}$} of $\resol$---we will use Theorem \ref{thm:main2} to establish a Koszul duality between $\cW_{\resol}$-modules and dg modules over an analogue of
$\Omega_X$ that ``lives on'' $\mathsf{\Lambda}$.  As a result, the bounded derived category $D^b(\cW\operatorname{-}\mathsf{good})$ naturally sheafifies over $\mathsf{\Lambda}$.   

We explain in Section \ref{intro:kashiwara} the main technical result that makes Theorem \ref{thm:main2} possible, namely, an analogue of Kashiwara's Equivalence. 
In Section
\ref{intro:applications} we derive several applications to the category $\WQcoh$.  
 In Section \ref{intro:fukaya} we describe parallels to the structure of Fukaya categories in more detail, and indicate some future work in that direction.

\subsection{Analogue of Kashiwara's Equivalence for Deformation Quantizations}\label{intro:kashiwara}
Theorem \ref{thm:main2} is a consequence of a symplectic version of a fundamental phenomenon of $\D$-module theory, classically encoded in the following topological invariance property.
\begin{Kashiwara}   Suppose that $C\subset X$ is a smooth closed subset of a smooth variety $X$.  Then the category $\mathsf{Qcoh}(\D_X)_C$ of quasicoherent $\D_X$-modules supported set-theoretically on $C$ is equivalent to the category
$\mathsf{Qcoh}(\D_C)$ of quasicoherent $\D_C$-modules. 
\end{Kashiwara}

Assume now that $(\resol, \omega)$ is a smooth, connected symplectic variety with elliptic $\Gm$-action.  Let $Y$ be a connected component of the fixed point set of $\resol$ under $\Gm$-action and $C$ the set of points in $\resol$ limiting to $Y$ under $\Gm$. Assume that $C$ is closed in $\resol$.  Let $C\rightarrow S$ denote the symplectic quotient whose existence is assured by Theorem \ref{thm:main3}(2).  
 The subcategory of $\Good{\cW}$ consisting of objects supported on $C$ is $\Good{\cW}_C$.
 
 To the algebra $\cW$ and the coisotropic subset $C$ we associate a sheaf of algebras $\cW_S$ on the symplectic quotient $S$.  We define a natural {\em coisotropic reduction functor} $\Ham: \Good{\cW_\resol}_C \rightarrow \Good{\cW_S}$.   
  The following provides an analogue of Kashiwara's Equivalence for $\cW$-modules. 
\begin{thm}[Theorem \ref{thm:equivmain1} and Corollary \ref{cor:preshol}]\label{thm:main1}
\mbox{}
\begin{enumerate}
\item 
The functor $\Ham:\Good{\cW_\resol}_C \rightarrow \Good{\cW_S}$ is an exact equivalence of categories.  It induces an exact equivalence 
$\Ham: \mathsf{Qcoh}(\cW_\resol)_C \rightarrow \mathsf{Qcoh}(\cW_S)$.
\item The functor $\Ham$ preserves both holonomicity and regular holonomicity.  
\end{enumerate}
\end{thm}
We also analyze the structure of $\cW_S$ more fully.  More precisely, we prove that, on $Y$, there exists a sheaf of ``generalized twisted differential operators'' $\dd_S$, a filtered $\mc{O}_Y$-algebra, whose completed Rees algebra sheafifies over $S$ and gives exactly $\cW_S$.  We obtain:
\begin{thm}[Theorem \ref{thm:equivmain1} and Proposition \ref{prop:twistedequiv}]\label{thm:main1b}
\mbox{}
\begin{enumerate}
\item The functor $\Ham$ of coisotropic reduction, followed by taking $\Gm$-finite vectors, defines an equivalence 
$$
\Ham : \Good{\cW_\resol}_C \stackrel{\sim}{\longrightarrow} \mathsf{coh}(\dd_S).
$$
Passing to ind-categories defines an equivalence $\Ham: \WQcoh_C\rightarrow \mathsf{Qcoh}(\dd_S)$.
\end{enumerate}
\begin{enumerate}
\item[(2)] In particular, if $Y$ consists of a single isolated $\Gm$-fixed point, $\Ham$ defines an equivalence  $$\mathsf{Qcoh}(\cW_\resol)_C\cong\mathsf{Qcoh}\big(\D(\mathbb{A}^{t_i})\big)$$ for some $t_i$.  
\end{enumerate}
\end{thm}

We emphasize that the elliptic $\Cs$-action is both essential to relate $\cW$-modules to representation theory and an intrinsic part of the geometry behind Theorems \ref{thm:main1} and \ref{thm:main1b}: it is not simply a technical convenience to make the proofs work.  As an illustration, consider $\D$-modules on $\mathbb{A}^1$ with singular support in the zero section, \ie local systems.  The zero section is a conical coisotropic subvariety in $T^*\mathbb{A}^1$ with symplectic quotient a point.  However, the category of {\em all} algebraic local systems on $\mathbb{A}^1$ is {\em not} equivalent to the category $\on{Vect}_{\C}$ of finite-dimensional vector spaces, \ie coherent $\D$-modules on a point: such a statement is only true once one passes to the subcategory of local systems regular at infinity.  
One can define a good notion of regularity  in the deformation-quantization setting \cite{DP}; and then a regular $\cW$-module supported on a coisotropic subvariety $C$ will be in the essential image of the corresponding functor quasi-inverse to $\Ham$.  In particular, passing to regular objects yields a version of Theorem \ref{thm:main1}  as in \cite{DP}, but at a  cost too high for our intended applications: the subcategory obtained is no longer described in terms of support conditions, and correspondingly one loses control over what the subquotients look like.  

Theorem \ref{thm:main1}, on the other hand, imposes a natural geometric condition on the coisotropic subvariety $C$: it must arise from the (algebro-geometric) Morse theory of the $\Gm$-action.  With that condition satisfied, regularity can be replaced by the more natural geometric support condition, thus yielding a precise structural result on $\WQcoh$.  

\subsection{Abelian and Derived Categories}\label{sec:intro-cats}
Our main results and techniques also establish some basic properties of categories of $\cW$-modules that are analogues of familiar assertions for categories of coherent or quasicoherent sheaves. 

\subsubsection{Abelian Categories}
As one example, let $Z \subset \resol^{\Cs}$ be a closed, connected and smooth subvariety. Let $\displaystyle C = \{ x \in {\resol} \ | \ \lim_{t \rightarrow \infty} t \cdot x \in Z \}$ be the attracting locus for $Z$; it is a smooth, locally closed subvariety of $\resol$.  
Assume that $C$ is closed in $\resol$.  The complement to $C$ in ${\resol}$ is denoted by $U$ and we write $j  : U \hookrightarrow {\resol}$ for the embedding.

\begin{thm}[Theorem \ref{keyprop} and Corollary \ref{cor:qcohfullesssurj}]
\label{thm:ab-esssurj}
\mbox{}
\begin{enumerate}
\item
The inclusion functor $i_*: \mathsf{Qcoh}(\cW)_{C} \rightarrow \mathsf{Qcoh}(\cW)$ realizes $\mathsf{Qcoh}(\cW)_{C}$ as a localizing subcategory of $\mathsf{Qcoh}(\cW)$.  In particular, $i_*$ admits a right adjoint $ i^!$ such that the adjunction $\mathrm{id} \rightarrow i^! \circ i_*$ is an isomorphism. 
\item
The restriction functor $j^* : \mathsf{Qcoh}(\cW_{\resol})  \rightarrow \mathsf{Qcoh}(\cW_U)$ induces equivalences 
\[
\Good{\cW_{\resol}}/\Good{\cW_{\resol}}_{C} \simeq \Good{\cW_U} \hspace{2em} \text{and}\hspace{2em}
\mathsf{Qcoh}(\cW_{\resol}) / \mathsf{Qcoh}(\cW_{\resol})_{C} \simeq \mathsf{Qcoh}(\cW_U).
\]
In particular, $j^*$ admits a right adjoint $j_* : \mathsf{Qcoh}(\cW_U) \rightarrow \mathsf{Qcoh}(\cW_{\resol})$ with $j^* \circ j_*\simeq \on{id}$. 
\end{enumerate}
\end{thm}

As corollaries of Theorem \ref{thm:ab-esssurj}, we immediately get corresponding statements for both the bounded and unbounded derived categories.
Unfortunately, it is clear that neither $j^*$ nor $i_*$ can admit {\em left} adjoints in general (see however \cite[Theorem~5.20]{BLPW2} in a special case). However, we show in the forthcoming paper \cite{BDMN} that the inclusions $C \stackrel{i}{\hookrightarrow} \resol \stackrel{j}{\hookleftarrow} U$ determine a full recollement pattern on the derived category of {\em holonomic} $\cW$-modules.

\subsubsection{Derived Categories and Compact Generation}
The categorical cell decomposition of $\mathsf{Qcoh}(\cW)$ extends to the derived level: the (unbounded) derived category $D(\mathsf{Qcoh}(\cW))$ is filtered by the full localizing triangulated subcategories $D_{K_{\ge i}}(\mathsf{Qcoh}(\cW))$ of objects with cohomology supported on the closed subvarieties $C_{K_{\ge i}}$. The associated minimal subquotients are given by 
$$
D_{K_{\ge i}}(\mathsf{Qcoh}(\cW)) / D_{K_{> i}}(\mathsf{Qcoh}(\cW)) \simeq D(\mathsf{Qcoh}(\dd_{S_i})).
$$
Since $\dd_{S_i}$ is a sheaf of $\mc{O}_{Y_i}$-algebras of finite homological dimension, standard arguments show that the triangulated category $D(\mathsf{Qcoh}(\dd_{S_i}))$ enjoys strong generation properties. Namely, the category is compactly generated and the full subcategory of compact objects is precisely the category of perfect complexes. Our symplectic generalization of Kashiwara's equivalence allows one to inductively show that these properties lift to the categories $D_{K}(\mathsf{Qcoh}(\cW))$. In particular, if $D(\mathsf{Qcoh}(\cW))^c$ denotes the full subcategory of compact objects, then taking $K$ to be all of $\{ 1, 2, \dots, k \}$ we have:

\begin{thm}[Corollaries \ref{cor:compactgen} and \ref{cor:perfcompact}]
The derived category $D(\mathsf{Qcoh}(\cW))$ is compactly generated and there is an equality 
$$
D(\mathsf{Qcoh}(\cW))^c = \mathsf{perf}(\cW) = D^b(\Good{\cW})
$$
of full, triangulated subcategories of $D(\mathsf{Qcoh}(\cW))$. 
\end{thm}

An analogous compact generation result was shown by Petit \cite{P}, though the category of cohomologically complete deformation quantization modules considered in \textit{loc. cit.} is (in a precise sense) orthogonal to $\mathsf{Qcoh}(\cW)$.

\subsection{Applications to Invariants}\label{intro:applications}
Theorem \ref{thm:main2} yields immediate consequences for the structure of fundamental invariants associated to the category of sheaves over a deformation quantization of $\resol$.  For example:

\begin{corollary}[Section \ref{sec:Kgroup}]\label{cor:Ktheory}
Suppose $\resol^{\Gm}$ is finite of cardinality $k$.  Choose a refinement of the partial ordering of coisotropic attracting loci $C_i$ of $\resol$ to a total ordering.  Then the group $K_0\!\left(\mathsf{perf}(\cW) \right)$ comes equipped with a canonical $k$-step filtration  each subquotient of which is isomorphic to $\mathbb{Z}$; in particular, $K_0\!\left(\mathsf{perf}(\cW) \right)$ is free abelian of rank $k$. 
\end{corollary}

In fact, using the fundamental properties of holononic modules developed in the sequel \cite{BDMN}, one can show that there are natural isomorphisms
$$
\theta_n : K_n(\C) \o_{\Z} K_0 \left(\mathsf{perf}(\cW) \right) \stackrel{\sim}{\longrightarrow} K_n\!\left(\mathsf{perf}(\cW) \right) 
$$
for all $n \ge 0$. Similar results hold for the cyclic and Hochschild homology of the dg-enhancement $\mathsf{Perf}(\cW)$. 

\begin{corollary}[Section \ref{sec:homology}]\label{cor:homocyc}
Suppose $\resol^{\Gm}$ is finite of cardinality $k$.  
Let $H_{\idot}(\resol)$ denote the Borel-Moore homology of $\resol$, with coefficients in $\C$. There are isomorphisms of graded vector spaces
$$
HH_{\idot}(\mathsf{Perf}(\cW)) \simeq H_{\idot - \dim \resol}(\resol), \quad HC_{\idot}(\mathsf{Perf}(\cW)) \simeq H_{\idot - \dim \resol}(\resol) \o \C[\epsilon], 
$$
where $\epsilon$ is assumed to have degree two.
\end{corollary}

In most situations, such as those that appear in representation theory, it is also possible to calculate the Hochschild {\em co}homology of $\mathsf{Perf}(\cW)$. Namely, in Proposition \ref{prop:chomology} we show: 
\begin{corollary}[Section \ref{sec:homology}]
Suppose $\resol^{\Gm}$ is finite, of cardinality $k$. Then, 
$$
HH^{\idot}(\mathsf{Perf}(\cW)) = H^{\idot}(\resol, \C).
$$
\end{corollary}

Via derived localization, see \cite{KTDerived}, the above results allows one to easily calculate the additive invariants $HH_*$, $HC_*$ and $HH^*$ of many quantizations of singular (affine) symplectic varieties that occur naturally in representation theory.   See Section \ref{sec:HH} for a discussion and applications.
For example, let $\Gamma$ be a cyclic group and $\mathfrak{S}_n \wr \Gamma$ the wreath product group that acts as a symplectic reflection group on $\C^{2n}$. The corresponding symplectic reflection algebra at $t = 1$ and parameter $\mathbf{c}$ is denoted by $\mathsf{H}_{\mathbf{c}}(\mathfrak{S}_n \wr \Gamma)$. For the definition of the filtration $F$ in the corollary below, see example \ref{ex:RCA}. 

\begin{corollary}[Proposition \ref{prop:sph}]
Assume that $\mathbf{c}$ is spherical. Then,   
$$
HH^{\idot}(\mathsf{H}_{\mathbf{c}}(\mathfrak{S}_n \wr \Gamma)) = HH_{2n - \idot}(\mathsf{H}_{\mathbf{c}}(\mathfrak{S}_n \wr \Gamma)) = \gr^F_{\idot}(\mathsf{Z} \mathfrak{S}_n \wr \Gamma),
$$
as graded vector spaces. 
\end{corollary}

One can deduce similar results for finite $W$-algebras associated to nilpotent elements regular in a Levi, and quantizations of slices to Schubert varieties in affine Grassmannians. These examples are explained in more detail at the end of section \ref{sec:Kgroup}.  

\subsection{Relation to Fukaya Categories of Weinstein Manifolds}\label{intro:fukaya}
There are a close conceptual link and, conjecturally, a precise mathematical relationship between the categories $\WQcoh$ that we study and the  Fukaya categories of Weinstein manifolds in real symplectic geometry.  More precisely, a growing body of important work in real symplectic geometry (by, among others, Abouzaid, Kontsevich, Nadler, Seidel, Soibelman, Tamarkin, Tsygan and Zaslow) establishes fundamental links between structures of microlocal sheaf theory and Fukaya categories.  The exposition of Nadler \cite{Nadler} sets the Fukaya theory of Weinstein manifolds squarely in a Morse-theoretic context, by showing how to use integral transforms to realize brane categories as glued from the homotopically simpler categories of branes living on coisotropic cells.   Theorem \ref{thm:main2} provides an exact parallel to the structure described in \cite{Nadler}.  One difference worth noting is that our categories include objects with arbitrary coisotropic support, not just Lagrangian support as in standard Fukaya theory: we provide such gluing structure for an algebro-geometric ``bulk'' category of all coisotropic branes rather just than the ``thin'' category of Lagrangian branes.    

Expert opinion supports a direct relationship between the category $\WQcoh$ (or more precisely the holonomic subcategory) and the structure of Fukaya categories described in \cite{Nadler}.  Namely, many examples of hyperk\"ahler manifolds with $S^1$-action fit our paradigm and have affine hyperk\"ahler rotations possessing the requirements described in \cite{Nadler}.  In such cases, it is natural to try to prove that the category of \cite{Nadler} is equivalent to $\WQcoh$ for a particular choice of $\cW$  by first proving cell-by-cell equivalences; next, describing a classifying object for categories built from cells as in Theorem \ref{thm:main2} and \cite{Nadler}; and, finally and most difficult, isolating a collection of properties that distinguish the Fukaya category of \cite{Nadler} in the universal family and matching it to some $\WQcoh$.  We intend to return to this problem in future work.

\subsection{Relation to other work on Deformation Quantization}
In recent years there has been much interest in the study of quantizations of certain classes of symplectic algebraic varieties, going back at least as far as the work of Kashiwara and Rouquier \cite{KR} on the Hilbert scheme of points in the plane mentioned above. The class of varieties which has attracted the most interest is that of conical symplectic resolutions. These are  symplectic varieties $Y$ with a $\Cs$-action such that the affinization map $f\colon Y \to X$ is birational and the resulting $\Cs$-action on $X$ has a single attracting fixed point. Braden, Proudfoot and Webster \cite{BPWsequel} and Braden, Licata, Proudfoot and Webster\cite{BLPW2} give a systematic study of quantizations of these varieties, and study in detail a class of holonomic modules in the spirit of the classical theory of category $\mathcal O$ (see also the subsequent work of Losev \cite{LosevcatO}).

Clearly any such conic symplectic resolution is an elliptic symplectic variety, but the class of elliptic symplectic varieties is strictly larger. For example, if $\Sigma$ is a smooth complete curve, then $Y=\text{Hilb}^n(T^*\Sigma)$, the Hilbert scheme of points of the cotangent bundle of $\Sigma$ is naturally a symplectic variety (studied by Nakajima \cite{NakajimaBook}, for example) and it possess a natural elliptic $\Cs$-action induced by the scaling action on the fibers of $T^*\Sigma$. However, unless $\Sigma= \mathbb P^1$, the symplectic variety $Y$ is not a symplectic resolution. Moreover, in this paper we seek to investigate the structure of the full category of $\Cs$-equivariant modules, rather than focusing on particular classes of holonomic modules.  Our forthcoming work \cite{BDMN} was partly inspired by the question of which properties of the category of all (suitably equivariant) modules for a quantization of $\resol$ can be detected by  small subcategories such as the geometric incarnations of category $\mathcal O$ studied in \cite{BPWsequel}. That paper however, as in the work of \cite{BPWsequel} mentioned above, requires $Y$ to carry the action of a higher dimensional torus.

\subsection{Outline of the Paper} Section \ref{sec:symplecticvarietiessection} describes some of the basic geometric properties of symplectic manifolds equipped with an elliptic $\Cs$-action. The basic properties of modules over deformation-quantization algebras are recalled in Section \ref{sec:DQmod}. In Section \ref{sec:QCR} we describe a version of quantum coisotropic reduction for equivariant DQ-algebras and prove a version of Kashiwara's equivalence. This equivalence is used in Section \ref{sec:apps} to study the derived category $D(\Qcoh(\cW))$ and also calculate the additive invariants of $\Good{\cW}$. 

\subsection{Conventions}

Deformation-quantization algebras $A$, sheaves of DQ algebras $\cA$, $\cW$-algebras $W$ (global section case) and $\cW$ (sheaf case), and their modules and equivariant modules, are defined in the body of the paper.  For a (sheaf of) DQ-algebra(s) $A$ (respectively $\cA$) with $\Gm$-action, we always write
$\Lmod{(A,\Gm)}$ for the category of finitely generated equivariant modules (respectively, $\Coh{(\cA,\Gm)}$ for the category of coherent modules).  

For a $\cW$-algebra $\cW$, we write $\Good{\cW}$ for the category of good $\cW$-modules.  If $\sT$ is a torus acting on $\resol$ and $\cW$ is a $\sT$-equivariant $\cW$-algebra (where $\sT$ acts on $\hbar$ via a character), we write $\Good{(\cW,\sT)}$ for the category of good $\sT$-equivariant $\cW$-modules, or, if $\sT$ is clear from context, just $\Good{\cW}$.  We write $\Qcoh(\cW)$ and $\Qcoh(\cW,\sT)$, for the ind-categories of 
$\Good{\cW}$ and $\Good{(\cW,\sT)}$ respectively.  

\subsection*{Acknowledgments}
The authors are grateful to David Ben-Zvi, Andrea D'Agnolo, Kobi Kremnitzer, David Nadler, and Eric Vasserot for very helpful conversations. 

Bellamy was supported by the EPSRC grant EP-H028153. McGerty was supported by a Royal Society research fellowship.  McGerty was also supported by the EPSRC grant EP/1033343/1. Nevins was supported by NSF grants DMS-0757987 and DMS-1159468 and NSA grant H98230-12-1-0216, and by an All Souls Visiting Fellowship.  All four authors were supported by MSRI.

\section{The Geometry of Symplectic Varieties with Elliptic $\Gm$-Action}\label{sec:symplecticvarietiessection}
We assume throughout this section that $(\resol,\omega)$ is a smooth, connected symplectic, quasi-projective variety with elliptic $\Gm$-action. By \textit{symplectic manifold} we mean a smooth quasi-projective variety over $\C$ equipped with an algebraic symplectic form. In this section, we describe some basic geometric consequences of the $\Gm$-action.

\subsection{A Symplectic Bia{\l}ynicki-Birula Decomposition}
The connected components of the fixed point set of $\resol$ under the $\Gm$-action will be denoted by $Y_1 , \ds, Y_k$. Each $Y_i$ is a smooth, closed subvariety of $\resol$. Recall that  
$$
C_i = \{ x \in \resol \ | \ \lim_{t \rightarrow \infty} t \cdot x  \in Y_i \}.
$$
Then it follows that $\resol = \bigsqcup_{i = 1}^k C_i$. 

Let $C$ a smooth, connected, locally closed coisotropic subvariety of $\resol$. A \textit{coisotropic reduction} of $C$ is a smooth symplectic variety $(S,\omega')$ together with a smooth morphism $\pi : C \rightarrow S$ such that $\omega |_{C} = \pi^* \omega'$. A classical example of a coisotropic reduction is given by $\resol = T^* X$, $Y \subset X$ a smooth, closed subvariety, $C = (T^* X) |_Y$ and $\pi : (T^* X) |_Y \rightarrow T^* Y$ the natural map. 

\begin{thm}\label{thm:maincoiso}
Suppose $(\resol,\omega)$ is a symplectic manifold with elliptic $\Gm$-action.  Then:
\begin{enumerate}
\item Each $C_i$ is a smooth, coisotropic subvariety of $\resol$ and an $\Cs$-equivariant affine bundle over the fixed point set $Y_i$. 
\item There exist symplectic manifolds $(S_i,\omega_i)$ with elliptic $\Gm$-action and $\Cs$-equivariant coisotropic reductions $\pi_i : C_i \rightarrow S_i$.
\end{enumerate}
\end{thm}

The proof of the first statement of Theorem \ref{thm:maincoiso} is given in Section \ref{sec:coisoproof}. The proof of the second statement of Theorem \ref{thm:maincoiso} is given in Section \ref{sec:thmglobal} after some preparatory work.

\subsection{Proof of Theorem \ref{thm:maincoiso}(1)}\label{sec:coisoproof}  

The proof of Theorem \ref{thm:maincoiso}(1) is essentially a direct consequence of the Bia{\l}ynicki-Birula decomposition together with some elementary weight arguments. However, we provide details for completeness. The fact that each $C_i$ is an $\Cs$-equivariant affine bundle over $Y_i$ follows directly from \cite[Theorem 4.1]{BBFix}. 

With regard to $\Cs$-representations, the following conventions will be used throughout the paper. If $V$ is a graded vector space then $V_i$ denotes the subspace of degree $i$. Let $V$ be a pro-rational $\Cs$-module: that is, $V$ is the limit of its $\Gm$-equivariant rational quotients.  Then $V^{\mathrm{rat}} = \bigoplus_{i \in \Z} V_i$, the subspace of $\Cs$-finite vectors, is a rational $\Cs$-module. 

To show that each $C_i$ is coisotropic, we first show that $T_p C_i$ is a coisotropic subspace of $T_{p} \resol$ at each point $p \in Y_i$. Indeed we claim that 
\begin{equation}\label{eq:coisoweights}
(T_{p} C_i)^{\perp} = \mathrm{rad} \left( \omega |_{C_i} \right)_{p} = \bigoplus_{j < - l} (T_{p} C_i)_j.
\end{equation}
To see this, let $V = T_{p} \resol$ and $W = T_{p} C_i$. Then $V$ and $W$ have weight space decompositions $V = \bigoplus_j V_j$ and $W = \bigoplus_j W_j$ where $W_j = V_j$ for $j \le 0$ and $W_j = 0$ for $j > 0$. If $v \in V_a$ and $w \in V_b$ then
$$
t^{l} \omega(v,w) = (t \cdot \omega)(v,w) = \omega( t^{-1} \cdot v, t^{-1} \cdot w) = t^{- a -b} \omega(v,w).
$$
This implies that $\omega(v,w) = 0$ if $l \neq - a - b$ and $\omega$ restricts to a non-degenerate pairing $V_j \times V_{-l - j} \rightarrow \C$. Therefore, if $v \in V_j \cap W^{\perp} (= V_j)$ then the equality $\omega(v,w) = 0$ for all $w \in W$ implies that $V_{-l-j} \cap W = 0$ \ie $-l - j > 0$ and hence $j < - l$. Hence $V_j \subset W$. This implies that $W^{\perp} = W_{< -l}$ and (\ref{eq:coisoweights}) follows. Then, the following lemma completes the proof of Theorem \ref{thm:maincoiso}(1).

\begin{lemma}\label{lem:attract}
Let $C$ be an attracting set in $\resol$ and $Y \subset C$ the set of fixed points. Then $C$ is coisotropic if and only if $(T_p C)^{\perp} \subset T_p C$ for all $p \in Y$. 
\end{lemma}

\begin{proof}[Proof of Lemma \ref{lem:attract}]
Fix $p \in Y$. There exists a $\Cs$-stable affine open neighborhood $U$ of $p$ on which the tangent bundle equivariantly trivializes \ie $T \resol |_U \simeq U \times T_p \resol$. To see this, let $U_0$ be a $\Cs$-stable affine open neighborhood of $p$ and $\mf{m} \lhd \C[U_0]$ the maximal ideal defining $p$. Choose a homogeneous lift $x_1, \ds, x_{2n}$ of a basis of $\mf{m} / \mf{m}^2$ in $\C[U_0]$. Then there exists some affine open $\Cs$-stable subset $U \subset U_0$ such that $d x_1,\ds, d x_{2n}$ is a free basis of $\Omega^1_U$ as a $\C[U]$-module. Shrinking $U$ if necessary, we may assume that $U^{\Cs} = Y \cap U$. Under the corresponding identification $T_x \resol \stackrel{\sim}{\longrightarrow} T_p \resol$ of tangent spaces, $T_x C$ is mapped to $T_p C$ for all $x \in U \cap C$. Write $\omega = \sum_{i<j} f_{i,j} d x_i \wedge d x_j$, thought of as a family of skew-symmetric bilinear forms on the fixed vector space $T_p \resol$. 

If $\partial_i$ is dual to $d x_i$, then as shown above $T_p C$ is spanned by all $\partial_i$ of degree $\le 0$ and $(T_p C)^{\perp_p}$ is spanned by all $\partial_i$ of degree less than $- l$.
By definition $U \cap C$ is the set of all points in $U$ vanishing on all $f \in \C[U]$ of negative degree. Let $\partial_i \in (T_p C)^{\perp_p}$ and $\partial_j \in T_p C$. Then $\deg d x_i > l$ and $\deg x_j \ge 0$. Therefore $\deg f_{i,j} < 0$. This implies that $f_{i,j}(x) = 0$ for all $x \in U \cap C$ and hence $\omega_x(\partial_i, \partial_j) = 0$ so that $(T_p C)^{\perp_p} \subseteq (T_p C)^{\perp_x}$. Since $\dim(T_p C)^{\perp_p} = \dim(\resol) - \dim(T_p C) = \dim(T_p C)^{\perp_x}$ we must have $(T_p C)^{\perp_x} = (T_p C)^{\perp_p}$ and so the Lemma follows.
\end{proof}

Define a relation on the coisotropic attracting loci $C_i$ by  $C_i\geq C_j$ if $\overline{C_j}\cap C_i\neq\emptyset$. 
\begin{lemma}\label{lem:partial order}
The relation $C_i\geq C_j$ is anti-symmetric: in particular, it defines a partial order on the Bia\l ynicki-Birula strata.
\end{lemma}
\begin{proof}
Equivariantly compactify $\resol$ to a smooth projective $\Gm$-variety $\overline{\resol}$, see \cite[Theorem 5.1.25]{CG}.  Since we assume that every point in $\resol$ has a $\mathbb G_m$-limit in $\resol$, a point in $x \in \overline{\resol}$ has its $\mathbb G_m$-limit in the closed set $\partial\overline{\resol} = \overline{\resol}\smallsetminus \resol$ if and only if $x \in \partial\overline{\resol}$; in particular, $\partial\overline{\resol}$ is a union of BB strata of $\overline{\resol}$.  The relation on BB strata defined above on $\resol$ is a subset of the relation defined using the BB stratification of $\overline{\resol}$ (some stratum closures may intersect at the boundary in $\overline{\partial}$ but not in $\resol$ itself).  Since the BB stratification of $\overline{\resol}$ is {\em filterable} by \cite{BBFix2}, the conclusion follows.
\end{proof}

\subsection{Proof of Theorem  \ref{thm:maincoiso}(2): From Global to Local}\label{sec:thmglobal} We fix $C$ to be one of the coisotropic strata (one of the $C_i$) in $\resol$ and let $Y = C^{\Cs}$. Note that if $C = C_i$ then Lemma \ref{lem:partial order} shows that $C_{\geq i} = \bigcup_{j \geq i} C_j$ is open in $\resol$ and $C_i$ is closed in $C_{\geq i}$. Since an open union of coisotropic cells in $\resol$ inherits the structure of an elliptic symplectic variety, we may thus assume without loss of generality that $C$ is closed in $\resol$. We first describe a canonical global construction of a morphism $\pi\colon C \to S$. To show that this construction yields a coisotropic reduction is a local computation, which we carry out in the next section, giving a local normal form for the symplectic form on a formal neighborhood of $C$. 

The global construction can be described as follows:
Let $\rho : C \rightarrow Y$ be the projection map and $\mc{I}$ denote the sheaf of ideals in $\mc{O}_{\resol}$ defining $C$. The quotient $\mc{I} / \mc{I}^2$ is a locally free $\mc{O}_C$-module. By Lemma \ref{lem:reducedI} below, $\mc{I}$  is involutive. Therefore the Lie algebroid $\mc{L} := \rho_{\bullet} (\mc{I}  / \mc{I}^2)$ acts on $\rho_{\bullet} \mc{O}_C$ via Hamiltonian vector fields and we can consider the sheaf $\mf{H} := H^0(\mc{L}, \rho_\bullet \mc{O}_C)$ of sections of $\rho_{\bullet} \mc{O}_C$ that are invariant under these Hamiltonian vector fields. The fact that the Poisson bracket has weight $- l$ and $\mc{O}_Y$ is concentrated in degree zero implies that $\mc{L}$ is actually an $\mc{O}_Y$-Lie algebra and $\mf{H}$ a sheaf of $\mc{O}_Y$-algebras. 

The embedding of $\mf{H}$ into $\rho_{\bullet} \mc{O}_C $ defines a dominant map $\pi :  C \rightarrow S := \spec_Y \mf{H}$ of schemes over $Y$. The final claim of Theorem \ref{thm:maincoiso} is that the map $\pi$ is a coisotropic reduction and, in particular, $S$ is a symplectic manifold. Since both $C$ and $S$ are affine over $Y$ and the statement of the claim is local on $S$, it suffices to assume that we are in the local situation of section \ref{sec:localalg} below. Then the claim is a consequence of Theorem \ref{thm:canonicalform}; see the end of section \ref{sec:localalg}.  

\subsection{Proof of Theorem  \ref{thm:maincoiso}(2): Affine Local Case}\label{sec:localalg} In this section we prove that the construction described in section \ref{sec:thmglobal} does indeed give a coisotropic reduction. The idea is to show that, locally, there is a different construction of $\pi$ using the $\Cs$-action on $\resol$. This construction clearly gives a coisotropic reduction. Unfortunately, this construction doesn't obviously lift to a global construction. Therefore the main thrust of this section is to show that this second construction agrees (locally) with the construction given in section \ref{sec:thmglobal}. 

It follows from \cite[Theorem 2.5]{BBFix} that, for each point $y \in Y$, there is some affine open neighborhood of $y$ in $Y$ such that the affine bundle $\rho : C \rightarrow Y$ trivializes equivariantly. Replacing $Y$ by such an affine open subset, we assume that $\rho$ is equivariantly trivial. Moreover, we may assume that $\resol$ is also affine (still equipped with a $\Cs$-action). Set $R = \C[\resol]$, a regular affine $\C$-algebra with non-degenerate Poisson bracket $\{ - , - \}$ and $\Cs$-action such that $\{ - , - \}$ has weight $- l$. Let $I$ be the ideal in $R$ defining $C$.  

Let $\mf{l} = I / I^2$. The Poisson bracket on $R$ makes $\mf{l}$ into a Lie algebra which acts on $R/I$. Let 
$$
H := H^0(\mf{l}, R/I) = (R / I)^{\{ I , -  \}}
$$
denote the ``coisotropic reduction" of $R$ with respect to $I$. The bracket $\{ - , - \}$ descends to a bracket on $H$ and $R/I$ is a Poisson module for $H$. We set $S = \spec H$ and let $\pi$ be the dominant map coming from the inclusion $H \hookrightarrow R/I$. This is the local version of the construction described in section \ref{sec:thmglobal}. 

\begin{lem}\label{lem:reducedI}
The ideal $I$ equals the ideal of $R$ generated by all homogeneous elements of negative degree and is involutive \ie $\{ I , I \} \subset I$. 
\end{lem}

\begin{proof}
Let $J$ be the ideal of $R$ generated by all homogeneous elements of negative degree. The fact that $\{ - , - \}$ has degree $- l$ implies that $J$ is involutive. Moreover, the set of zeros of $I$ and $J$ clearly coincide. Therefore the lemma is really asserting that $J$ is a radical ideal. If $\resol$ were not smooth then this need not be true. Let $D = \spec R/J$. Since the non-reduced locus is a closed, $\Cs$-stable subscheme of $D$, it suffices to show that the local ring $\mc{O}_{D,y}$ is a domain for all $\Cs$-fixed closed points of $D$. But in this case, if $\mf{m}$ is the maximal ideal in $R$ defining $y$, then $T^*_y D$ is precisely the subspace of $\mf{m} / \mf{m}^2$ of non-negative weights and $D$ is locally cut out by homogeneous lifts of the elements of $\mf{m} / \mf{m}^2$ of negative degree. Since $\resol$ is smooth at $y$, these elements form a regular sequence and hence $\mc{O}_{D,y}$ is reduced. 
\end{proof}

Now we give our alternative, local construction of $\pi$ which we will use to verify the morphism $\pi$ is a coisotropic reduction. For the remainder of this section, let $S$ denote the affine variety such that $\C[S] \subset \C[C]$ is the subalgebra generated by all homogeneous elements of degree at most $l$ and let $\pi$ denote the dominant morphism $C \rightarrow S$. 

\begin{lem}\label{lem:iscoisotropic}\mbox{} 
The variety $S$ is symplectic and $\pi$ is a coisotropic reduction of $C$. 
\end{lem}

\begin{proof}
Fix some $y \in Y$. First we establish that $S$ is a smooth variety of dimension $\dim \ (T_y C)_{\ge - l }$. By definition, there is an equivariant trivialization $\phi : C \stackrel{\sim}{\longrightarrow} Y \times V \times Z$, where $V \subset T_y C$ is the sum of all weight spaces with weight $- l \le i < 0$ and $Z$ the sum of all weight spaces of weight $< - l$. Since $\phi$ is equivariant and $S$ defined in terms of weights, $\phi^*(\C[S]) = \C[Y \times V]$ and $\pi' \circ \phi^{-1}$ corresponds to the projection map onto $Y \times V$. Thus, $S$ is smooth of the stated dimension and $\pi$ is a smooth morphism. 

Since $C$ is only coisotropic, the bracket on $\resol$ does not restrict to a bracket on $R/I$. However, as explained above, it does induce a bracket on $H$. Since $\C[S]$ is generated by homogeneous elements of degree at most $l$, and $I$ is generated by elements of negative degree, the algebra $\C[S]$ is contained in $H$. Again, weight considerations imply that it is Poisson closed. Thus, it inherits a bracket from $\resol$ making $\pi$ a Poisson morphism. 

Finally we need to show that the Poisson structure on $\C[S]$ is non-degenerate. Since $\C[S]$ is positively graded and smooth, it suffices to check the induced pairing 
$$
\{ - , - \}_y : T_y^* S \times T_y^* S \rightarrow \C
$$
is non-degenerate. As noted in the proof of Lemma \ref{lem:reducedI}, the space $T^*_y C = (T^*_y \resol) / (T_y^* \resol)_{< 0}$. Since $\C[S]$ is generated by all elements of degree at most $l$ in $\C[C]$, $(T^*_y \resol)_{\le l} / (T_y^* \resol)_{< 0}$ is contained in $T^*_y S$. But these two spaces have the same dimension. Therefore they are equal. Since $\{ - , - \}$ has degree $- l$ and is non-degenerate on $T^*_y \resol$, the induced pairing 
$$
\{ - , - \}_y : (T^*_y \resol)_{\le l} / (T_y^* \resol)_{< 0} \times (T^*_y \resol)_{\le l} / (T_y^* \resol)_{< 0} \longrightarrow \C
$$
is non-degenerate. 
\end{proof}

As shown in the proof of Lemma \ref{lem:iscoisotropic}, we have Poisson subalgebras $\C[S] \subset H$ of $R/I$ and the Poisson structure on $\C[S]$ is non-degenerate. Therefore, to show that $H$ is also a regular affine algebra with non-degenerate Poisson bracket it suffices to show that $\C[S] = H$. In order to do this, we shall need to investigate more closely the local structure of the symplectic form. In particular, we shall need a Darboux-Weinstein type theorem to describe the behavior of the form near $C \subset \resol$. Since we are working in the Zariski topology, we shall take the formal completion along $C$. 

Shrinking $Y$ further if necessary, we may assume that the normal bundle $\NN_{\resol/C}$ to $C$ in $\resol$ is $\Cs$-equivariantly trivial. Thus, we have a $\Cs$-equivariant trivialization $\mathrm{Tot}(\NN_{\resol / C}) \simeq C \times Z^{*}$. This implies that $I/I^2$ is a free $R/I$-module.

In these circumstances,  \cite[paragraphe III.1.1.10 and th\'eor\`em~III.1.2.3]{Illusie} imply that we can choose ($\Gm$-equivariantly) an identification of the formal neighborhood $\mf{C}$ of $C$ in $\resol$ with the formal neighborhood $\widehat{\NN}_{\resol/C}$ of the zero section in $\NN_{\resol / C}$.   
By Lemma \ref{lem:iscoisotropic}, $C \simeq S \times Z$, where $Z \simeq \mathbb{A}^m$. It follows easily that there is a $\Cs$-equivariant isomorphism $\mf{C} \simeq S \times \widehat{T^{*} Z}$ where $\widehat{T^{*} Z}$ denotes the completion of $T^{*} Z$ along the zero section; the corresponding ideal in $\C[\mf{C}] = \widehat{R}$ is therefore the completion of $I$, denoted $\widehat{I}$. Then $\pi$ extends to a projection map $S \times \widehat{T^{*} Z} \rightarrow S$. The inclusion $S \hookrightarrow \mf{C}$ is denoted by $\iota$. 

For an arbitrary morphism $f : X \rightarrow X'$, there is a map $\Omega_f : f^* \Omega^{\bullet}_{X'} \rightarrow \Omega^{\bullet}_X$. Assume $X$ and $X'$ are affine. If $\nu$ is a closed $k$-form on $X$ and $\overline{\nu}$ its image in $f^* \Omega_{X'}^k$, then $\Omega_f(\overline{\nu})$ is closed in $\Omega^k_X$. From $\resol$, the space $\mf{C}$ inherits a symplectic form $\omega$ of weight $l$. Set $\omega_{S} = \Omega_{\pi} ( \Omega_{\iota} (\overline{\omega}))$, a closed 2-form on $\mf{C}$. Our local normal form result states:

\begin{thm}\label{thm:canonicalform}
Under the identification $\mf{C} \simeq S \times \widehat{T^* Z}$ there is a $\Cs$-equivariant automorphism $\phi$ of $\mf{C}$, with $\phi ( \widehat{I} ) = \widehat{I}$ such that
\begin{equation}\label{eq:canform}
\phi^* \omega = \omega_{\can} := \omega_{S} + \sum_{i = 1}^m d z_i \wedge d w_i, 
\end{equation}
with respect to some homogeneous bases $\mbf{z} = z_1, \dots,z_m$ and $\mbf{w} = w_1,\dots, w_m$ of $Z^{*} \subset \C[Z]$ and $Z \subset \C[\![Z^*]\!]$ respectively.
\end{thm}

The proof of Theorem \ref{thm:canonicalform} will be given in section \ref{sec:proof}.

\begin{remark}
We have $\deg z_i > l$, $\deg w_i < 0$ and $\deg z_i + \deg w_i = l$ for all $i$.  
\end{remark}

Theorem \ref{thm:canonicalform} implies:

\begin{corollary}\label{lem:comcompleteiso}
We have $\NN_{\resol/C} \simeq (\pi^* Z^*) \o \chi^{l}$, where $\chi = \id : \Cs \rightarrow \Cs$ is the fundamental character of $\Cs$. 
\end{corollary}

We can now complete the proof of Theorem \ref{thm:maincoiso}. Recall from section \ref{sec:localalg} that our goal is to show that $\C[S] = H$ as subalgebras of $R / I$. 

The ring $\wh{R}$ can be identified with functions on $\mf{C}$. Under this identification, the ideal $\wh{I} = \wh{R} I$ is the ideal of functions vanishing on the zero section $C$ and $ \wh{R} / \wh{I} = R / I$. Since $I$ is involutive and $\wh{I} = \wh{R} I$, the ideal $\wh{I}$ is involutive. By Theorem \ref{thm:canonicalform}, there is an automorphism $\phi^*$ of $\wh{R}$ such that $\phi^* ( \{ f, g \}) = \{ \phi^*(f) , \phi^*(g) \}_{\mathrm{can}}$, where $\{ - , - \}_{\mathrm{can}}$ is the canonical Poisson bracket coming from the symplectic two-form (\ref{eq:canform}). By construction, $\phi^*(\wh{I}) = \wh{I}$. Therefore, $(\wh{R}/ \wh{I})^{ \{ \wh{I}, - \} } = (\wh{R}/ \wh{I})^{ \{ \wh{I}, - \}_{\mathrm{can}} }$. Since 
$$
\wh{R} = (R / I)[\![w_1, \ds, w_m]\!] = \left( \C[S] \o \C[z_1, \ds, z_m] \right) [\![w_1, \ds, w_m]\!],
$$
and $\wh{I}$ is generated by $w_1, \ds, w_m$, the embedding $\C[S] \hookrightarrow \wh{R}$ induces an isomorphism $\C[S] \simeq (\wh{R}/ \wh{I})^{ \{ \wh{I}, - \} }$. Finally, the equality $\wh{I} = \wh{R} I$ implies that $(\wh{R}/ \wh{I})^{ \{ \wh{I}, - \} } = (R/ I)^{ \{ I, - \} } = H$. 

The following observation will be useful later. 

\begin{lem}\label{lem:TPoisson}
The projection $\rho : C \rightarrow Y$ factors through $\pi : C \rightarrow S$.  
\end{lem}

\begin{proof}
The algebra $R / I$ is $\mathbb{N}$-graded such that the quotient by the ideal generated by all elements of strictly positive degree equals $\C[Y]$: the proof of this last claim is identical to the proof of Lemma \ref{lem:reducedI}, using the fact that $C$ is smooth. On the other hand we can also identify $\C[Y]$ with the degree zero part of $R/I$. Therefore, it suffices to show that all sections of $R/I$ of degree zero lie in $T$. To see this, notice that the Poisson bracket on $R$ has degree $-l$. Therefore, given $f \in (R/I)_0$ and $g = \sum_i g_i h_i \in I$, where $\deg h_i < 0$ for all $i$, the element $\{ f, g \} = \sum_i g_i \{ f, h_i \} + h_i \{ f, g_i \}$ lies in $I$ because $\deg  \{ f, h_i \} < 0$. 
\end{proof}  

\subsection{The Proof of Theorem \ref{thm:canonicalform}}\label{sec:proof}

The crucial tool in the proof of Theorem \ref{thm:canonicalform} is an algebraic version of the Darboux-Weinstein theorem which is due to Knop, \cite[Theorem 5.1]{KnopWeyl}. In our setup, Knop's Theorem can be stated as:

\begin{thm} \label{thm:Knopthm}
Let $\{ - , - \}_{\omega}$ and $\{ - , - \}_{\omega_{\can}}$ denote the Poisson brackets on $\widehat{R}$ associated to the symplectic forms $\omega$ and $\omega_{\can}$ respectively. Denote their difference by $\{- , - \}_{-}$. If $\{\widehat{R},\widehat{R}\}_{-} \subset \widehat{I}$, then there exists an automorphism $\phi$ of $\mf{C}$ as described in Theorem \ref{thm:canonicalform}. 
\end{thm}

Knop's proof of Theorem \ref{thm:Knopthm} is based on the proof by Guillemin and Sternberg of the equivariant Darboux-Weinstein Theorem, \cite{GuilleminSternberg}, except that Knop works in the formal algebraic setting.

In our case, it is not necessarily true that $\{\widehat{R},\widehat{R}\}_{-} \subset \widehat{I}$. Instead, we construct an equivariant automorphism $\psi$ of $\widehat{R}$, such that $\psi (\widehat{I}) = \widehat{I}$ and the difference of $\{ - , - \}_{\psi^* (\omega)}$ and $\{ - , - \}_{\omega_{\can}}$ has the desired properties. In fact we prove the following:

\begin{prop} \label{prop:basis2}
There exist homogeneous elements $w_1,\dots,w_m$ and $z_1,\dots,z_m$ in $\widehat{R}$, and a graded subalgebra $T$ of $\widehat{R}$ such that:

\begin{enumerate}
\item $T \cap \widehat{I} =0$ and the map $\widehat{R} \to R/I$ induces a graded algebra isomorphism $T \to \C[S]$.  
\item The elements $w_1,\dots,w_m$ generate $\widehat{I}$.  
\item There is a $\Cs$-equivariant isomorphism $R/I \simeq \C[S][z_1,\dots, z_m]$ 
\item With respect to the Poisson structure $\{ - , - \}=\{ - , - \}_{\omega}$, we have $\{w_i,z_j\}= \delta_{ij} \mod \widehat{I}$,  $\{z_i,T\}, \{w_i,T\} \subset \widehat{I}$ and $\{z_i,z_j\} \in \widehat{I}$ for all $i$ and $j$. 
\end{enumerate} 
\end{prop}

Since the choice of such elements $w_1,\ldots,w_m,z_1,\ldots,z_m$ clearly yield a $\Cs$-equivariant automorphism of $\widehat{R}$ (which fixes $\widehat{I}$ and $S$), one sees directly that this proposition implies the existence of $\psi$ as described above. 

So we turn to the proof of Proposition \ref{prop:basis2}. As in the proof of Lemma \ref{lem:TPoisson}, we make the identification $\C[Y] = (R/I)_0$. We have $S = Y \times V$ and $C = Y \times V \times Z$, where $V$ and $Z$ are as in Lemma \ref{lem:iscoisotropic}. Let $V_{l}^*$ be the $l$-weight subspace of $V^*$.

\begin{lemma}\label{lem:DerFree}
The Poisson bracket on $\C[S]$ defines an isomorphism of $\C[Y]$-modules
$$
\C[Y] \o V_{l}^* \stackrel{\sim}{\longrightarrow} \Der (\C[Y]). 
$$
\end{lemma}

\begin{proof}
That the map is well-defined and $\C[Y]$-linear follows from degree considerations. Lemma \ref{lem:iscoisotropic} implies that the Poisson bracket on $\C[S]$ is non-degenerate. This non-degeneracy implies that the above map is surjective. Since it is a surjective map between two projective $\C[Y]$-modules of the same rank, it is an isomorphism. 
\end{proof}

Choose an arbitrary point $y \in Y$.  Our strategy will be to complete at $y$ and use the Darboux theorem for a symplectic formal disc to control the behavior of the Poisson bracket in some neighborhood of $y$. For the convenience of the reader we recall the statement of the Darboux theorem in the presence of an elliptic $\Cs$-action: 

\begin{thm}\label{thm:formaldarboux}
Let $\widehat{R}_y$ denote the completion of $R$ at $y$. Then, in $\widehat{R}_{y}$ there is a regular sequence $\{u_{1},\dots,u_{n},v_{1},\dots,v_{n}\}$ of homogenous elements such that $\{u_{i},v_{j}\}=\delta_{ij}$ and
$\{u_{i},u_{j}\}=0=\{v_{i},v_{j}\}$ for all $i$ and $j$. 
\end{thm}

The proof of the equivariant Darboux theorem is a slight modification of standard arguments and we shall omit it. Since $\C[Y] \subset H$, $R/I$ is a Poisson $\C[Y]$-module, where $\C[Y]$ is equipped with the trivial Poisson bracket. Let $K \subset R/I$ denote the set of all elements $k$ such that $\{ \C[Y], k \} =0$; it is a graded subalgebra of $R/I$.  

\begin{lem}\label{lem:KerY}
Multiplication defines an isomorphism $\C[V_{-l}] \o K \stackrel{\sim}{\rightarrow} R/ I$. 
\end{lem}

\begin{proof}
By Krull's intersection theorem, we may consider $R/I$ as a subalgebra of $\widehat{R}_{y}/I\widehat{R}_{y}$, and hence of $(\widehat{R}_{y}/I\widehat{R}_{y})^{\rat}$ too. The ring $\widehat{R}_{y}/I\widehat{R}_{y}$ is a Poisson module over the ring $\widehat{\C}[Y]_y$, where $\widehat{\C}[Y]_y$ is the completion of $\C[Y]$ at the point $y$. If $\widehat{K} \subset \widehat{R}_{y}/I\widehat{R}_{y}$ and $K^{\rat} \subset (\widehat{R}_{y}/I\widehat{R}_{y})^{\rat}$ are defined analogously to $K$, then it suffices to show that $\C[V_{-l}] \o K^{\rat} \stackrel{\sim}{\rightarrow} (\widehat{R}_{y}/I\widehat{R}_{y})^{\rat}$. 

Reordering if necessary, we may suppose that $\{v_{1},\dots,v_{m}\}$ are the elements of Theorem \ref{thm:formaldarboux} that have negative degree. They generate the ideal $I \widehat{R}_y$. There exist $\alpha_1, \ds, \alpha_r, \beta_1, \ds, \beta_s$ such that $\{ \alpha_1, \ds, \alpha_r, \beta_1, \ds, \beta_s \} \subset \{u_1, \ds, u_n, v_{m+1}, \ds, v_n \}$ and $\C[\![ \alpha_1, \ds, \alpha_r]\!] = \C[\![V_{-l}]\!], \C[\![\beta_1, \ds, \beta_s]\!] = \C[\![V_{>- l} \times Z]\!]$. In this case, it follows from Theorem \ref{thm:formaldarboux} that 
$$
\widehat{R}_{y}/I\widehat{R}_{y} = \widehat{\C}[Y]_y \widehat{\o} \C[\![\alpha_1, \ds, \alpha_r, \beta_1, \ds, \beta_s]\!] \quad \textrm{with} \quad \widehat{K} = \widehat{\C}[Y]_y \widehat{\o} \C[\![\beta_1, \ds, \beta_s]\!]. 
$$
Hence, since $\deg \alpha_i$ and $\deg \beta_j > 0$ for all $i$ and $j$,  
\begin{equation}\label{eq:ratdecomp}
(\widehat{R}_{y}/I\widehat{R}_{y})^{\rat} = \C[V_{-l}] \o \widehat{\C}[Y]_y \o \C[\beta_1, \ds, \beta_s]. 
\end{equation}
Lemma \ref{lem:DerFree} implies that we may fix a basis $x_1, \ds, x_r$ of $V_{l}^{*} \subset \C[V_{-l}]$ and regular sequence $y_1, \ds, y_r$ in $\C[Y]$ such that $d y_1, \ds, d y_r$ are a basis of $\Omega^1_Y$ and $\{x_i,y_j\}=\delta_{ij}$ for all $i$ and $j$. This, together with the identification (\ref{eq:ratdecomp}), implies that $K^{\rat} = \widehat{\C}[Y]_y \o \C[\beta_1, \ds, \beta_s]$ and $\C[V_{-l}] \o K^{\rat} \stackrel{\sim}{\rightarrow} (\widehat{R}_{y}/I\widehat{R}_{y})^{\rat}$. 
\end{proof}

Now we begin constructing elements that satisfy the conditions of Proposition \ref{prop:basis2}. Since the statement of Proposition \ref{prop:basis2} is local, it suffices to replace $\resol$ by some sufficiently small affine neighborhood of $y$ where the statement holds. We prove :

\begin{lemma}
\label{lem:factordecomposition}
There exists a $\Cs$-equivariant identification $C \simeq S \times Z$ such that $\{\C[Y], \C[Z] \} = 0$.
\end{lemma}

\begin{proof}
Let $\C[Z] = \C[z_1, \ds, z_m]$. We show that the elements $z_i$ can be modified so that the lemma holds. If $x_1, \ds, x_r$ is the basis of $V^*_{l}$ as in the proof of Lemma \ref{lem:KerY}, then that lemma implies that we may uniquely decompose: 
$$
z_i = \sum_{I \in \mathbb{N}^{r}} x^{I} \cdot p_{I}^{(i)} 
$$
for $p_I^{(i)} \in K$. Since the $x_i$ have degree $l >0$ it follows that $\deg p_I^{(i)} < \deg z_i$ for all $I \neq 0$. Using this fact, it is straight-forward to show by induction on degree that 
$$
\C[Y \times V_{>- l}][z_1, \ds, z_m] = \C[Y \times V_{>- l}][p^{(1)}_0, \ds, p_0^{(m)}], 
$$
which implies the lemma. Indeed if $z_{1}, \ds, z_{k}$ have the minimal possible degree then $\deg p_I^{(i)} < \deg z_i$ for $I \neq 0$ implies that $p_I^{(i)} \in \C[Y \times V_{>- l}]$ and hence, 
$$
\C[Y \times V_{>-l}][z_{1}, \ds, z_{k}] = \C[Y \times V_{>- l}][p^{(1)}_0, \ds, p_0^{(k)}].
$$
The inductive step is entirely analogous.
\end{proof}

Recall that the Lie algebra $\mf{l} = I/I^2$ acts on $R/I$ as well. Next, we consider the action of $\mf{l}$ on $\C[Z]$. Recall that we have assumed that $Y$ is small enough so that we have a $\Cs$-equivariant trivialization $\mathrm{Tot}(\NN_{\resol /C}) \cong C\times Z^*$. It follows that $\mf{l}$ is a free $R/I$-module of rank equal to $\dim(Z)$. We now consider the action of $\mf{l}$ on $\C[Z]$ viewed as a subalgebra of $R/I$ via the isomorphism constructed in Lemma \ref{lem:factordecomposition}:

\begin{lemma}\label{lem:homopairing}
There exists a homogeneous subspace $W \subset \mf{l}$ such that $\mf{l} = R/I \o W$ as an $R/I$-module and the action of $\mf{l}$ on $R/I$ restricts to a non-degenerate, equivariant pairing $W \times Z^{*} \rightarrow \C$. 
\end{lemma}

\begin{proof}
Let $\{z_1,\ldots,z_m\}$ be the generators of $\C[Z]\subset R/I$ as in the proof of Lemma \ref{lem:factordecomposition}, viewed as elements of $Z^*$. We must show that there is a homogeneous $R/I$-basis of $\mf{l}$ dual to the $z_i$ with respect to the pairing induced by the Poisson bracket. Fix $\{w_1,\ldots,w_m\}$ some homogeneous $R/I$-basis. We will modify this basis by increasing induction on degree in order to obtain the required dual basis. 

For each integer $-m<0$ the Poisson bracket induces a $\C[Y]$-linear pairing between the $\mf{l}_{-m}$, the $(-m)$-th graded piece of $\mf{l}$, and $\C[Z]_{l+m}$. Let $Z^*_{l+m}$ be $\mathbb C$-span of the $z_i$s of degree $l+m$ and let $N_{l+m}$ be the $\C[Y]$-module it generates. Similarly let $\mf{n}_{-m}\subseteq \mf{l}_{-m}$ be the $\C[Y]$-module generated by the $w_j$ of degree $-m$. Since the pairing induced by the Poisson bracket on the tangent space $T_y \resol$ at a closed point $y \in Y$ is homogeneous and nondegenerate, it follows that the $\C[Y]$-pairing between $N_{l+m}$ and $\mf{n}_{-m}$ is also nondegenerate. It follows that we may modify the $w_j$s in $\mf{n}_{-m}$ to be dual to the $z_j \in N_{l+m}$, that is, so that $\{w_j,z_k\} = \delta_{j,k}$. 

Thus we suppose by induction that for all $d<-m$ the $w_j$ of degree $-d$ are dual to the $z_k$ of degree $l+d$ and the pairings $\{w_j,z_k\}=0$ for $w_j$ and $z_k$ in degrees less than $-m$ and greater than $l+m$, respectively, which are not paired. Note that for sufficiently large $m$ this holds vacuously. In degree $-m$, modify the $w_i$s to be dual to the $z_k$ of degree $l+m$ as above. To complete the inductive step we must ensure the new $w_i$s have trivial Poisson bracket with the $z_k$s of higher degree, modify the $w_i$ according to the homogeneous substitution:
\[
w_i \mapsto w_i - \sum_{\deg(w_j)<-m} \{w_i,z_j\}w_j
\]
Since the $\mf{l}$-module structure on $R/I$ is left $R/I$-linear, it follows from the inductive hypothesis that for these new $w_i$s we have $\{w_i,z_k\} = 0$ whenever $\deg(z_k)>l+m$. Moreover, 
these new $w_i$ are still linearly independent, as they remain dual to the $z_k$ of degree $l+m$, since $\{w_j,z_k\}=0$ for all $w_j$ of degree less than $-m$ by consideration of degree.
\end{proof}

With these two Lemmas in hand we may now give:

\begin{proof}[Proof of Proposition \ref{prop:basis2}]
If we denote by the same letter a homogeneous lift of the space $W$ of Lemma \ref{lem:homopairing} to $\widehat{I}$, then $\widehat{R} = (R/I)[\![W]\!]$. In this way we regard $W$ and $Z^*$ as subspaces of $\widehat{R}$. Now by Lemma \ref{lem:homopairing} we may choose bases $w_1,\dots,w_m$ and $z_1,\dots,z_m$ of $W$ and $Z^*$ respectively so that conditions (2), (3), and the first condition of (4) in Proposition  \ref{prop:basis2} hold. It remains to ensure the conditions that $\{ z_i , z_j \}$ belong to $\widehat{I}$ and to find a suitable subalgebra $T$. For this we repeatedly use the substitution strategy of Lemma \ref{lem:homopairing}: Inductively define
\[
z_i' = z_i - \sum_{j=1}^{i-1}\{z_i,z'_j\}w_j, \quad \textrm{for all } \ i .
\]
One checks directly that these elements Poisson commute with each other modulo $\widehat{I}$, and since $z_i' - z_i \in \widehat{I}$, we have that  $\{ z_i',\C[Y] \} = 0 \mod \widehat{I}$ and $\{z_i',w_j\} = \delta_{ij} \mod \widehat{I}$. Finally, if $V$ is as in Lemma \ref{lem:iscoisotropic}, and we pick a basis $x_1, \ds, x_t$ of $V^*$ so that $\C[S] = \C[Y] [x_1, \ds, x_t]$, and set  
$$
x'_i = x_i - \sum_{j=1}^{m}\{ z_j',x_i\} w_j, \quad \textrm{for all } \ i
$$
and let $T$ be the algebra generated by $\C[Y]$ and the $x'_i$. Then it follows easily that $\{z_i',T\}=0 \mod \widehat{I}$ and $\{w_i,T\}=0 \mod \widehat{I}$ (the latter from elementary degree considerations) and thus the proposition is proved. 
\end{proof}

\subsection{Special Case of Theorem \ref{thm:maincoiso}: Isolated Fixed Points} Assume now that each fixed-point component $Y_i$ of the elliptic $\Gm$-action is a single point $\{ p_i \}$. In this case each $S_i$ is isomorphic to $\mathbb{A}^{2 t_i}$ and there exists globally a splitting $S_i \hookrightarrow C_i \stackrel{\pi_i}{\twoheadrightarrow} S_i$.  Let $(\mathbb A^{2n},\omega_{can})$ be the $2n$-dimensional affine space equipped with the constant symplectic form. We now check that in this case the symplectic reduction is isomorphic to a linear symplectic form. 

\begin{prop}\label{prop:formonV}
Let $(S,\omega_S)$ be an affine symplectic manifold, isomorphic to $\mathbb{A}^{2n}$, equipped with an elliptic $\Cs$-action with unique fixed point $o \in S$. Then there exists an isomorphism $\phi\colon S \to \mathbb A^{2n}$ and homogeneous algebraically independent generators $\bz = z_1, \ds, z_n$ and $\bw = w_1, \ds, w_n$ of $\C[S]$ such that 
$$
\phi^* \omega = d z_1 \wedge d w_1 + \dots + d z_n \wedge d w_n. 
$$
\end{prop}

\begin{proof}
Take
 the completion at $o \in S$.  Then by the formal Darboux theorem we may choose homogeneous generators $\{u_1,\ldots,u_n,v_1,\ldots,v_n\}$ for which the Poisson bracket associated to $\omega$ has standard form. We may view $\C[S]\subset \widehat{\C[S]}$ via the obvious map, and so we are reduced to showing that the generators $u_i$ and $v_j$ are actually elements of $\C[S]$. But the $\Gm$-action naturally extends to $\widehat{\C[S]}$, and since $o$ is the unique fixed point which is therefore the limit of every point in $S$, it follows that we may describe $\C[S]\subset \widehat{\C[S]}$ as $(\widehat{\C[S]})^{\text{rat}}$ the elements for which the $\Gm$-action is locally finite. Clearly any homogeneous element is locally finite and so we are done.
\end{proof}

\begin{corollary}\label{cor:finitefixed}
Suppose $\resol^{\Cs}$ is finite, of cardinality $k$. Then, for all $i$, there is a coisotropic reduction $\pi_i : C_i \rightarrow (\mathbb{A}^{2 t_i}, \omega_{\can})$. 
\end{corollary}

\subsection{Symplectic Fibrations}\label{sec:Sympaffbun}
In this subsection, we consider those symplectic manifolds equipped with an elliptic $\Cs$-action for which the set of $\Cs$-fixed points is a connected variety. At one extreme, we have cotangent bundles of a smooth variety, with $\Cs$ acting by rescaling the fibers; at the other extreme one has symplectic fibrations, which are affine bundles such that each fiber is a copy of affine symplectic space. We show that generally one gets a mix of these two extremes.  

Let $Y$ be a smooth connected variety. 
\begin{defn}
\mbox{}
\begin{itemize}
\item A \textit{symplectic fibration} over $Y$ is a tuple $(\SF,\eta,\{ - , - \} )$, where $\eta : \SF \rightarrow Y$ is an affine bundle and $\{ - , - \}$ an $\mc{O}_Y$-linear Poisson bracket on $\eta_{\bullet} \mc{O}_E$ such that the restriction of $\{ - , - \}$ to each fiber of $\eta$ is non-degenerate.  
\item The symplectic fibration is said to be {\em elliptic} if $\Cs$ acts on $\SF$ such that $\{ - , - \}$ is homogeneous of negative weight, $Y = \SF^{\Cs}$ and all weights of $\Cs$ on the fibers of $\eta$ are positive. 
\end{itemize}
\end{defn}
\noindent
If $(\SF,\eta,\omega)$ is a elliptic symplectic fibration then $(\eta^{-1}(y), \{ - , - \} |_{\eta^{-1}(y)})$ is a symplectic manifold for each $y \in Y$.

Since $T^* Y$ is a vector bundle, it is naturally an abelian group scheme over the base $Y$. 
\begin{defn}
Suppose $p:B\rightarrow Y$ is a smooth variety over $Y$ equipped with a symplectic form $\omega_B$.  Suppose that $B$ is equipped with an action $a: T^*Y\times_Y B\rightarrow B$ of the group scheme $T^*Y$ over $Y$.  We say $B\rightarrow Y$ is {\em symplectically automorphic} if, for any $1$-form $\theta$ on $Y$, 
 we have 
 \bd
 a(\theta, \cdot)^*\omega_B  = \omega_B + p^*d\theta.
 \ed
 \end{defn}
 As remarked in the introduction, if $B$ is a $T^*Y$-torsor, this reduces to the notion of {\em twisted cotangent bundle} as in \cite{BB}.

For the remainder of this section we assume that $\resol$ is equipped with an elliptic $\Cs$-action such that the fixed point set $Y$ of $\resol$ is connected and hence every point of $\resol$ has a limit in $Y$. By \cite[Theorem 4.1]{BBFix}, there is a smooth map $\rho : \resol \rightarrow Y$, making $\resol$ an affine bundle over $Y$. We should like to show that such symplectic varieties are symplectically automorphic varieties built from cotangent bundles and elliptic symplectic fibrations.

\begin{thm}\label{cor:sympafffactor}
Let $\resol$ be a smooth symplectic variety equipped with an elliptic $\Cs$-action such that the fixed point locus $Y = \resol^{\Cs}$ is connected. Then
\begin{enumerate}
\item The group scheme $T^* Y$ acts freely on $\resol$, making $\resol\rightarrow Y$ symplectically automorphic. 
\item The cotangent bundle $T^* Y$ embeds $T^* Y$-equivariantly in $\resol$. Moreover, the restriction of the symplectic form $\omega$ to $T^* Y \subset \resol$ equals the standard $2$-form on $T^* Y$.
\item The quotient $E := \resol / T^* Y$ inherits a Poisson structure making the projection $E\rightarrow Y$ an elliptic symplectic fibration. 
\end{enumerate} 
\end{thm}

The proof of Theorem \ref{cor:sympafffactor} is similar to the proof of the local normal form Theorem \ref{thm:canonicalform}; however, we must also use the powerful Artin approximation theorem of \cite{Artin}. For brevity, we write $\mc{O}_{\resol}$ for the algebra $\rho_{\bullet} \mc{O}_{\resol}$. Let $\mc{K}$ be the ideal in $\mc{O}_{\resol}$ generated by $(\mc{O}_{\resol})_{0 < i < l}$. 

\begin{lem}\label{lem:symtheta}
The symplectic form on $\resol$ restricts to a symplectic form on $\spec_Y \mc{O}_{\resol} / \mc{K}$.  Moreover, $\spec_Y \mc{O}_{\resol} / \mc{K} \simeq T^* Y$ as smooth symplectic varieties. 
\end{lem}

\begin{proof}
The $\mc{O}_Y$-submodule $(\mc{O}_{\resol})_{\le l}$ of $\mc{O}_{\resol}$ is preserved by the Poisson bracket. If $f \in (\mc{O}_{\resol})_{\le l}$ and $g \in \mc{O}_Y$, then $\{ f, g \} \in \mc{O}_Y$. Thus, $\{ f, - \}$ defines a derivation of $\mc{O}_Y$. If $f \in (\mc{O}_{\resol})_{< l}$ then this derivation is zero. Hence, the Poisson bracket defines a $\mc{O}_Y$-linear map $\phi : \mc{E} \rightarrow \Theta_Y$, where $\mc{E} = (\mc{O}_{\resol})_{\le l} / (\mc{O}_{\resol})_{< l}$. This is a map of Lie algebroids:
\begin{align*}
[\phi(f),\phi(g)](h) & = \phi(f)(\phi(g)(h)) - \phi(g)(\phi(f)(h)) \\
 & = \{ f, \{ g, h \} \} - \{ g,  \{ f,h \} \} \\
 & = - \{ h, \{ f,g \} \} -  \{ g,  \{ h,f \} \} - \{ g,  \{ f,h \} \} \\
 & = \phi(\{ f,g \} )(h). 
\end{align*}
Since the Poisson bracket on $\resol$ is non-degenerate, $\phi$ is surjective. Locally trivializing $\resol \simeq Y \times V$ shows that $\mc{E}$ is locally free of rank $\dim V_{-l} = \dim Y$ and $\mc{O}_{\resol} / \mc{K} = \Sym \mc{E}$. Therefore $\phi$ is an isomorphism.  
\end{proof}

\begin{lem}\label{lem:comodule}
The algebra $\mc{O}_{\resol}$ is a comodule for $\Sym \Theta_Y$.
\end{lem}

\begin{proof}
Recall that the comultiplication on the Hopf algebra  $\Sym \Theta_Y$ is defined by $\Delta(v) = v \o 1 + 1 \o v$ for $v \in \Theta_Y$. The algebra $\mc{O}_{\resol}$ is generated by $(\mc{O}_{\resol})_{ \le l}$. We define $\Delta_{\resol} : \mc{O}_{\resol} \rightarrow \mc{O}_{\resol} \o_{\mc{O}_Y} \Sym \Theta_Y$ by $\Delta_{\resol}(f) = f \o 1 + 1 \o \overline{f}$ for $f \in (\mc{O}_{\resol})_{ \le l}$ and extending to $\mc{O}_{\resol}$ using the fact that $\Delta_{\resol}$ should be an algebra homomorphism. Here $\overline{f}$ denotes the image of $f$ in $ \mc{O}_{\resol} / \mc{K} \simeq \Sym \Theta_Y$. Since $\mc{O}_{\resol}$ is not freely generated by $(\mc{O}_{\resol})_{ \le l}$, we need to show that this is well-defined. 

Choose a local set $x_1, \ds, x_k, z_1, \ds, z_r$ of homogeneous algebraically independent generators of $\mc{O}_{\resol}$ over $\mc{O}_Y$ such that $0 < \deg(x_i) < l$ and $\deg(z_j) = l$ for all $i,j$. The images $\overline{z}_i$ in $\Sym \Theta_Y$ of the $z_i$ form a basis of $\Theta_Y$. Then we define $\Delta_{\resol}'$ by $ \Delta_{\resol}'(f) = f \o 1 + 1 \o \overline{f}$ for $f \in \{ x_1, \ds, x_k, z_1, \ds, z_r \}$. This clearly defines a local comodule structure on $\mc{O}_{\resol}$. We just need to show that it equals $\Delta_{\resol}$ \ie it is independent of the choice of local generators. Take $f \in (\mc{O}_{\resol})_{\le l}$, a homogeneous element. There exist $a_i \in \mc{O}_Y$ and some $g$ such that $f = \sum_{i = 1}^r a_i z_i + g(x_1, \ds, x_k)$. Then,
\begin{align*}
\Delta_{\resol}' (f) & = \sum_{i = 1}^r a_i \Delta_{\resol}'(z_i) + g(\Delta_{\resol}'(x_1), \ds, \Delta_{\resol}'(x_k)) \\
 & = \sum_{i = 1}^r (a_i z_i \o 1 + 1 \o a_i \overline{z}_i) + g(x_1 \o 1, \ds, x_k \o 1) \\
 & = \left( \sum_{i = 1}^r a_i z_i  \right) \o 1  + 1 \o \overline{\left( \sum_{i = 1}^r a_i z_i  \right)} + g(x_1, \ds, x_k) \o 1 \\
 & = f \o 1 + 1 \o \overline{f} = \Delta_{\resol}(f) . 
\end{align*}
Finally, to check that $\Delta_{\resol}$ is compatible with the comultiplication on $\Sym \Theta_Y$, it suffices to do so locally, as above, where it is clear.
\end{proof}

We denote by $\mc{F}$ the graded $\mc{O}_Y$-subalgebra of $\mc{O}_{\resol}$ generated by $(\mc{O}_{\resol})_{< l}$. Set $\SF:= \spec_Y \mc{F}$ and let $\resol / T^* Y$ denote the spectrum of $\mc{O}_{\resol}^{\Delta_{\resol}}$. From the definition of $\Delta_{\resol}$ given in the proof of Lemma \ref{lem:comodule}, the algebra $\mc{F}$ is contained in $\mc{O}_{\resol}^{\Delta_{\resol}}$. Thus, we have a dominant map $\resol / T^* Y \rightarrow \SF$. Moreover, the local description of $\Delta_{\resol}$ given  in the proof of Lemma \ref{lem:comodule} shows that $\resol / T^* Y$ is the space of $T^* Y$-orbits in $\resol$ and the map $\resol / T^* Y \rightarrow \SF$ is an isomorphism. Thus, parts (1) and (2) of Theorem \ref{cor:sympafffactor}, except for the ``symplectically automorphic'' assertion of (1), are a consequence of Lemmata \ref{lem:symtheta} and \ref{lem:comodule}. For part (3) we need to show that the Poisson bracket on $\mc{O}_{\resol}$, restricted to $\mc{F}$, makes $\SF$ into a symplectic fibration. For this and the ``symplectically automorphic'' assertion, we need a local normal form of the Poisson bracket on $\mc{O}_{\resol}$.     

\begin{prop}\label{thm:hardfactor}
Locally, in the \'etale topology on $Y$, there is a $\Cs$-equivariant isomorphism of Poisson algebras 
$$
\mc{O}_{\resol} \simeq  \mc{O}_{T^* Y} \o \C[\mathbb{A}^{2n}],
$$
where the inherited Poisson bracket on $\C[\mathbb{A}^{2n}]$ is the standard one. 
\end{prop}

\begin{proof}
Let $r = \dim Y$. Since the statement is local we assume that $\resol \simeq Y \times V$ is affine, where $V$ is a $\Cs$-module with strictly negative weights. The algebra $R := \C[\resol]$ is isomorphic to $\C[Y] \otimes\C[V]$. If $x_1, \ds, x_{2n}$ is a homogeneous basis of $V^*_{< l} \subset \C[V]$ and $x_{2n+1}, \ds, x_{2n + r}$ a basis of $V^*_{l} \subset \C[V]$, then $\C[\resol] = \C[Y] \o \C[x_1, \ds, x_{2n +r}]$. Concretely, we wish to show that there exists an equivariant \'etale morphism $p : U \rightarrow Y$ such that $\Gamma(U, p^* \C[\resol])$ is isomorphic to $\C[U] [x'_{1},\dots,x'_{2n+r}]$, where the $\{x'_{i}\}$ are graded elements, with $\deg(x_i) = \deg(x_i')$, and the Poisson bracket satisfies $\{x'_{i},x'_{j}\}=\delta_{i+j,2n+1}$. Assuming we have done this, weight considerations imply that 
$$
\Gamma(U, p^* \C[\resol]) \simeq \C[U][x'_{2n+1},\dots,x'_{2n+r}] \o \C[\mathbb{A}^{2n}]
$$
as Poisson algebras and Lemma \ref{lem:DerFree} implies that $\C[U][x'_{2n+1},\dots,x'_{2n+r}] \simeq \C[T^* U]$ as Poisson algebras.  
\begin{remark}
Since this identification is compatible with the action of $T^*Y$, the ``symplectically automorphic'' assertion of part (1) of Theorem \ref{cor:sympafffactor} follows immediately from the same assertion for $T^*Y$ itself.
\end{remark}

Choose $y\in Y$, and consider the algebra $\widehat{R}_y^{\rat} \simeq \widehat{\C}[Y]_{y} \otimes \C[V]$ of $\Cs$-locally finite sections of the completion of $R$ at $y$. Let $\mf{m}$ denote the maximal ideal of $(y,0)$ in $\widehat{R}_y^{\rat}$. The formal Darboux theorem, Theorem \ref{thm:formaldarboux}, implies that there exist homogeneous elements $u_1, \ds, u_{2n+r}$ in $\widehat{R}_y^{\rat}$ such that $\{ u_i, u_j \} = \delta_{i + j, 2n +1}$ and their image in $\mf{m} / \mf{m}^2 = T^*_y Y \times V^*$ is a basis of $V^*$. With respect to our chosen basis, ${\displaystyle u_{i} = \sum_{j} g_{ij} x_{j}}$  for some $g_{ij}\in\widehat{\C}[Y]_{y}$. This implies the relations
\begin{equation}\label{eq:syseq}
\delta_{i+j,2n +1}=\{u_{i},u_{j}\}=\sum_{k,l} g_{ik}g_{jl}\{x_{k},x_{l}\},
\end{equation}
 where the last equality follows from the fact that $\{x_{k},g_{ij}\} = 0$ for all $i,j,k$ for reasons of degree.

Since these relations are taking place in the finite, free $\C[Y]$-module $\C[Y] \o \C[V]_{\le l}$ we may consider (\ref{eq:syseq}) as defining a system of polynomial equations in the variables $g_{ij}$. Then the formal Darboux theorem says that these equations have solutions in the completion of $Y$ at $y$. The Artin approximation theorem \cite[Corollary 2.1]{Artin} now assures us of the existence of an \'etale neighborhood of $y$ (\ie an \'etale map $p : U \to Y$ with a chosen
closed point $x$ living above $y$) in which the equations have a
solution $\{ g_{ij}' \}$, whose difference with the solution in $\widehat{\C}[Y]_{y}=\widehat{\C}[U]_{x}$
lies in $\mf{n}$, the maximal ideal of this local ring. 

Now, we can base change the affine map $\resol \to Y$
to get an affine bundle $\resol_U \to U$; furthermore, the induced
map $\resol_U \to \resol$ is \'etale since $U\to Y$
is. Therefore, the pullback of the Poisson bracket induces a Poisson
bracket on $p^* \C[\resol]$, which is homogeneous of weight $-l$
by construction. Thus the given solutions of these equations yield
elements 
\[
x_{i}'=\sum g'_{ij} x_{j}\in p^* \C[\resol] 
\]
 which satisfy the Poisson relations as in the conclusion of the proposition.
Furthermore, the choice of $x'_{i}$ implies that the determinant
of the $\C[U]$-linear transformation $x_{i}\mapsto x'_{i}$ is nonzero
at $p^{-1}(y)$ (because this is true for $\widehat{\C}[Y]_{y}$-linear transformation $x_i \mapsto u_i$). Thus, restricting to smaller neighborhoods if
necessary, we may assume that the map $x_{i}\mapsto x_{i}'$ is invertible;
this implies that the $x_{i}'$ are algebra generators of $p^* \C[\resol]$ over $\C[U]$, which proves the proposition 
\end{proof}

Proposition \ref{thm:hardfactor} implies that $\SF$ is a symplectic fibration. This proves Theorem \ref{cor:sympafffactor}. 

\begin{remark}
It follows from \cite[Theorem 2.2]{BBFix} that the equivariant closed embedding $T^* Y \hookrightarrow \resol$ of Corollary \ref{cor:sympafffactor} is unique. We also note that Theorem \ref{cor:sympafffactor} implies that if $\dim \resol = 2 \dim Y$ then $\resol \simeq T^* Y$. Moreover, the proof of Theorem \ref{cor:sympafffactor} shows that, locally in the Zariski topology,  
$$
\resol \simeq T^{*} Y \times_Y \SF
$$
as smooth varieties with $\Cs$-actions, though not as Poisson varieties. 
\end{remark}

\subsection{Reduction of coisotropic subvarieties} We can now give a generalization of Theorem \ref{thm:maincoiso}. Fix a coisotropic stratum $\rho : C \rightarrow Y$ and let $\pi : C \rightarrow S$ be the coisotropic reduction of $C$ given by Theorem \ref{thm:maincoiso}. Let $Y' \subset Y$ be a smooth, closed subvariety and set $C' = \rho^{-1}(Y')$. The proof of Theorem \ref{thm:maincoiso}(1) shows that $C' \subset C$ is coisotropic. Let $\mc{I}$ be the sheaf of ideals in $\mc{O}_S$ vanishing on $C'$. Since it is generated by elements in degree zero, it is an involutive ideal. Let $\rho' : S \rightarrow Y$ be the projection map. We perform coisotropic reduction as before and set 
\begin{equation}\label{eq:coiso2}
S' := \spec_{Y'} (\rho_{\bullet}' \mc{O}_S / \mc{I})^{ \{ \mc{I}, - \} },
\end{equation}
a Poisson variety. Just as in the proof of Theorem \ref{thm:maincoiso}, to show that $S'$ is a coisotropic reduction of $C'$, it suffices to do so locally on $Y$. Instead of considering the formal neighborhood of $Y$ in $S$, we pass to an etal\'e local neighborhood. Thus, by Proposition \ref{thm:hardfactor}, we may assume that $S \simeq T^* Y \times \mathbb{A}^{2n}$. Then $C' = (T^* Y)|_{Y'} \times \mathbb{A}^{2n}$ and $S'$ becomes the classical coisotropic reduction $T^* Y' \times \mathbb{A}^{2n}$. Thus, we have shown:

\begin{corollary}\label{cor:bigcoiso}
The space $C'$ is coisotropic in $\resol$ and there is a coisotropic reduction $\pi': C' \rightarrow S'$. 
\end{corollary}

In section \ref{sec:TDO} the above coisotropic reduction will be quantized. 

\section{Deformation-Quantization Modules}\label{sec:DQmod}

In this section, we recall the basic properties of DQ-algebras and their modules. We also prove an extension result, Theorem \ref{keyprop}, which will play a key role in section \ref{sec:apps}. 

\subsection{DQ-algebras: Affine Setting} We begin by recalling the definition of a deformation-quantization algebra. Let $(R, \{ - , - \})$ be a regular Poisson $\C$-algebra.

\begin{defn}
 A \textit{deformation-quantization} of $R$ is an $\h$-flat and $\h$-adically complete $\C[\![\h]\!]$-algebra $A$ equipped with an isomorphism of Poisson algebras $A/ \h A \cong R$. Here $A / \h A$ is equipped with a Poisson bracket via 
$$
\left\{ \overline{a}, \overline{b} \right\} = \left( \frac{1}{\h} [a,b] \right) \mod \h A,
$$
for arbitrary lifts $a,b$ of $\overline{a}, \overline{b}$ in $A$.  The algebra $A$ is a {\em deformation-quantization algebra} or {\em DQ-algebra} if it is a deformation-quantization of some regular Poisson algebra $R$. An {\em isomorphism of DQ-algebras} that are quantizations of the same Poisson algebra $R$ is a $\C[\![\h]\!]$-algebra isomorphism such that the induced map on $R$ is the identity.
\end{defn}

\begin{lem}\label{lem:hcomp1}
Let $A$ be a deformation-quantization algebra. 
\begin{enumerate}
\item $A$ is a (left and right) Noetherian domain of finite global dimension.
\item The Rees ring $\mathsf{Rees}_{\h} A  = \bigoplus_{n\geq 0} \hbar^nA$ is Noetherian. 
\item If $M$ is a finitely generated $A$-module, then $M$ is $\h$-complete. 
\end{enumerate}
\end{lem}

\begin{proof}
Part (1) follows from the fact that the associated graded of $A$ with respect to the $\h$-adic filtration is $R[\h]$, which is a regular domain. Part (2) is shown in Lemma 2.4.2 of \cite{FiniteDimRepsLosev}. Part (3) is a consequence of the Artin-Rees lemma.  
\end{proof}  

We also have the following well-known complete version of Nakayama's lemma. 

\begin{lem}\label{lem:NAK}
Let $M$ be a complete $\C[\![\h]\!]$-module. If $\h M = M$ then $M = 0$. 
\end{lem}

\subsection{Sheaves of DQ-Algebras}\label{sec:DQsheaf}

Let $(\resol, \{ - , - \})$ be a smooth Poisson variety. A sheaf of $\C[\![\h]\!]$-modules $\sA$ on $\resol$ is said to be $\h$-flat if each stalk $\sA_p$ is a flat $\C[\![\h]\!]$-module. For each positive integer $n$, let $\displaystyle\sA_n = \sA / \h^n \sA$. The $\h$-adic completion of $\sA$ is $\widehat{\sA} = \limn \sA_n$ and $\sA$ is said to be $\h$-adically complete if the canonical morphism $\sA \rightarrow \widehat{\sA}$ is an isomorphism. 

\begin{defn}
A sheaf of $\C[\![\h]\!]$-algebras $\sA$ on $X$ is said to be a \textit{deformation-quantization} algebra if it is $\h$-flat and $\h$-adically complete, equipped with an isomorphism of Poisson algebras $\sA_0 \cong \mc{O}_X$. 
\end{defn} 

If, moreover, the algebra $\sA$ is equipped with a $\Cs$-action that acts on $\hbar \in \sA$ with weight $l$ and the Poisson bracket on $\resol$ has degree $-l$, then we replace $\sA$ by $\sA[\h^{1/l}]$ and $\h$ by $\h^{1/l}$ so that, without loss of generality $t \cdot \h = t \h$ and the Poisson bracket on $\mc{O}_{\resol}$ coming from $\sA$ is defined by 
$$
\{ \overline{a}, \overline{b} \} := \frac{1}{\h^{l}} [ a, b] \mod \h \sA.
$$ 
\begin{remark}
For a symplectic variety $\resol$ with $\Gm$-action and deformation quantization $\sA$, we always assume $\sA$ is equivariant in the above sense.
\end{remark}

In the algebraic setting, the existence and classification of sheaves of deformation-quantization algebras is well understood. See \cite{BK, LosevIso} for the equivariant setting. Assume that $\resol$ is affine and let $R = \C[\resol]$; let $A$ be a deformation-quantization of $R$. For any multiplicatively closed subset $S$ of $R$, there is an associated microlocalization $Q^{\mu}_{S}(A)$ of $A$; the algebra $Q^{\mu}_{S}(A)$ is, by definition, a deformation-quantization of $R_S$. Using Gabriel filters, one can extend the notion of microlocalization to define a presheaf $\mc{O}^{\mu}_A$ of algebras on $\resol$ such that $\Gamma(D(f),\mc{O}^{\mu}_A) = Q^{\mu}_f(A)$ for all $f \in R$. By \cite[Theorem 4.2]{VS}, the presheaf $\mc{O}^{\mu}_A$ is a sheaf and the following proposition holds.

\begin{prop}
Assume that $\resol$ is affine. Then microlocalization defines an equivalence between the category of DQ-algebras quantizing $R$ and sheaves of DQ-algebras on $\resol$. 
\end{prop}

\subsection{Sheaves of DQ-Modules}\label{sec:DQmodules}
In this section, we define those $\sA$-modules that will be studied in sections \ref{sec:QCR} and \ref{sec:apps}. First, 
 
\begin{lemma}
If $H^1 (\resol, \mc{O}_X) = 0$, then 
$\Gamma(\resol, \sA_n)  = \Gamma(\resol, \sA) / \h^n \Gamma(\resol, \sA)$. 
\end{lemma}
\begin{proof}
Consider the short exact sequence 
$
0 \rightarrow \sA \stackrel{\cdot \h^n}{\longrightarrow} \sA \longrightarrow \sA_n \rightarrow 0.
$
Taking derived global sections, it suffices to show that $H^1(\resol, \sA) = 0$. 
Induction on $n$ using the exact sequence $0 \rightarrow \hbar\sA_{n-1} \rightarrow \sA_{n} \rightarrow \mc{O}_{\resol} \rightarrow 0$ shows that 
$H^1 (\resol, \mc{O}_X) = 0$ implies $H^1(\resol, \sA_n) = 0$. Then \cite[Proposition 13.3.1]{EGAIII} implies that $H^1(\resol, \sA) = 0$.
\end{proof}

Thus, write $A= \Gamma(\resol, \sA)$ and $A_n = \Gamma(\resol, \sA_n)$.  If $M$ is an $A_n$-module then, thinking of $A_n$ and $M$ as constant sheaves on $\resol$, we define $M^{\Delta} = \sA_n \o_{A_n} M$. 

\begin{defn}
A $\sA_n$-module $\ms{M}_n$ is \textit{quasi-coherent} if there exists an affine open covering $\{ U_i \}$ of $\resol$ such that $\ms{M}_n |_{U_i} \simeq \Gamma(U_i, \ms{M}_n)^{\Delta}$. If, moreover, $\Gamma(U_i, \sA_n)$ is a finitely generated $\Gamma(U_i, \ms{M}_n)$-module for all $i$ then $\ms{M}_n$ is said to be \textit{coherent}. 
\end{defn}

As in the commutative case, we have:

\begin{prop}\label{prop:naffinity}
An $\sA_n$-module $\ms{M}_n$ is quasi-coherent if and only if 
$$
\ms{M}_n |_U \simeq \Gamma(U, \ms{M}_n)^{\Delta}
$$
for all affine open subsets $U$ of $\resol$. 
\end{prop}

Assume that $\resol$ is affine and let $A = \Gamma(\resol, \sA)$. Let $M$ be an $A$-module. We define $M^{\Delta} = \limn \left( \sA_n \o_{A_n} M / \h^n M \right)$. We remark that if $\{ \mc{M}_n \}$ is an inverse system of sheaves, then $\limn \mc{M}_n$ is defined to be the sheaf $U \mapsto \limn \Gamma(U,\mc{M}_n)$; there is no need to sheafify. In particular, $\Gamma(\resol,\limn \mc{M}_n) = \limn \Gamma(\resol,\mc{M}_n)$.

We may now define the two classes of modules that play a role in this paper:

\begin{defn}
Let $\ms{M}$ be an $\sA$-module. 
\begin{enumerate}
\item $\ms{M}$ is \textit{coherent} if it is $\hbar$-complete and each $\ms{M}_n$ is a coherent $\sA_n$-module. 
\item $\ms{M}$ is \textit{quasi-coherent} if it the union of its coherent $\sA$-submodules.
\end{enumerate}
\end{defn}

The category of all coherent, resp. quasi-coherent $\sA$-modules is denoted $\Coh{\sA}$, resp. $\QCoh{\sA}$. The proof of the following is based on \cite[Theorem 5.5]{ArdakovWadsley}. 

\begin{prop}\label{prop:coheretaffinity}
Assume that $\resol$ is affine. Then $\Gamma(\resol, -)$ defines an exact equivalence between $\Coh{\sA}$ and the category $\Lmod{A}$ of finitely generated $A$-modules, where $A = \Gamma(\resol, \sA)$. A quasi-inverse is given by $M \mapsto M^{\Delta}$. Moreover, $M^{\Delta} = \sA \o_A M$.
\end{prop}

\begin{proof}
As for any localization theorem, the proof has three parts. First, we show that $\Gamma$ is exact on $\Coh{\sA}$. Then, we show that $\Gamma(\resol, \ms{M})$ is a finitely generated $A$-module, for all $\ms{M} \in \Coh{\sA}$. Finally, we show that $\ms{M}$ is generated by its global sections. 

Let $M_n = \Gamma(\resol, \ms{M}_n)$ and $M=\Gamma(\resol, \ms{M})$. Since the $\sA_n$-module $\h^{n-1} \ms{M} / \h^n \ms{M}$ is a submodule of the coherent $\sA_n$-module $\ms{M}_n$, it is coherent, and hence Proposition \ref{prop:naffinity} implies that the cohomology groups $H^i (\resol, \h^{n-1} \ms{M} / \h^n \ms{M})$ are zero for all $i \neq 0$. Therefore, we have surjective maps $M_n \rightarrow M_{n-1} \rightarrow \cdots$. Therefore, the inverse system $\{ M_n \}_n$ satisfies the Mittag-Leffler condition and hence \cite[Proposition 13.3.1]{EGAIII}, together with the fact that $\ms{M}$ is assumed to be complete, implies that $H^i(\resol, \ms{M}) = 0$ for all $i \neq0$. 

Since $\ms{M}_n / \h^n \ms{M}_n \cong \ms{M}_{n-1}$, the fact that $H^1(\resol, \ms{M}_n) = 0$ implies $M_n / \h^n M_n = M_{n-1}$. Therefore, by \cite[Lemma 3.2.2]{Berthelot}, the fact that each $M_n$ is a finitely generated $A_n$-module implies that $M$ is a finitely generated $A$-module. 

The fact that $\Gamma$ is exact implies that $M_n = M / \h^n M$. Therefore, by Proposition \ref{prop:naffinity}, 
$$
\ms{M} \simeq \limn \ms{M}_n = \limn  \sA_n \o_{A_n} M_n = \limn  \sA_n \o_{A_n} M / \h^n M = M^{\Delta}. 
$$
Finally, to show that this is the same as $\sA \o_A M$, let $A^n \rightarrow A^m \rightarrow M \rightarrow 0$ be a finite presentation of $M$. Since we have a natural map $\sA \o_A M  \rightarrow  \sA_n \o_{A_n} M / \h^n M$ for all $n$, there is a canonical morphism $\sA \o_A M  \rightarrow M^{\Delta}$. Then, we get the usual commutative diagram 
$$
\xymatrix{
\sA \o_A A^n \ar[r] \ar[d] & \sA \o A^m \ar[r] \ar[d] & \sA \o_A M \ar[r] \ar[d] & 0 \\
(A^n)^{\Delta} \ar[r] & (A^m)^{\Delta} \ar[r] & M^{\Delta} \ar[r] & 0 
}
$$
so the result follows from the Five Lemma and the fact that $(A^m)^{\Delta} = (\sA)^m$, which in turn is a consequence of the fact that microlocalization is an additive functor.  
\end{proof}
\begin{remark}
It is clear that, under the identification of the proposition, if $U\subset \resol$ is an inclusion of affine opens, the restriction functor $\ms{M}\mapsto \ms{M}|_U$ is identified with $M\mapsto H^0(U,\cA) \otimes_A M$.
\end{remark}

\begin{corollary}\label{cor:coherentaffine}
Let $\ms{M} \in \Coh{\sA}$ and $U \subset \resol$ an affine open set. Then, $H^i(U,\ms{M}) = 0$ for all $i \neq 0$ and $\ms{M} |_U \simeq \Gamma(U,\ms{M})^{\Delta}$, where $M$ is a finitely generated $\Gamma(U,\sA)$-module. 
\end{corollary}

\begin{lem}\label{lem:fgfp}
Let $\ms{M}$ be an $\sA$-module. The following are equivalent:
\begin{enumerate}
\item $\ms{M}$ is coherent. 
\item $\ms{M}$ is locally finitely presented. 
\item $\ms{M}$ is locally finitely generated. 
\end{enumerate}
\end{lem}

\begin{proof}
1) implies 2). Let $U$ be an affine open subset of $\resol$ and $M = \Gamma(U,\ms{M})$. Then, $M$ is finitely generated. Since $A$ is Noetherian, it is actually finitely presented and hence there is sequence $A^n \rightarrow A^m \rightarrow M \rightarrow 0$. The functor $\Delta$ is an equivalence on $\Lmod{A}$, hence we have $\sA^n |_U \rightarrow \sA^m |_U \rightarrow \ms{M} \rightarrow 0$. 

2) implies 3) is clear. 

3) implies 1) We have $\phi : \sA^m |_U \twoheadrightarrow \ms{M} |_U$ and hence $\sA^m_n |_U \twoheadrightarrow \ms{M}_n |_U$. Thus, each $\ms{M}_n |_U$ is coherent. The module $\ker \phi$ is a submodule of the coherent $\sA$-module $ \sA^m |_U$. Therefore the Artin-Rees lemma implies that the filtrations $\{ \h^n \ker \phi \}$ and $\{ (\ker \phi) \cap (\h^n \sA^m) \}$ are comparable. Hence $\displaystyle{\lim_{\longleftarrow}}^{(1)} (\ker \phi) / [ (\ker \phi) \cap \h^n (\sA^m) ] = 0$. This implies that $\ms{M}$ is complete. 
\end{proof}

\begin{prop}\label{prop:gg}
Suppose $\resol$ is affine and that $j: U\hookrightarrow\resol$ is an open subset.  Suppose that $\ms{M}$ is a coherent $\sA_U$-module.  Then $\ms{M}$ is globally generated.
\end{prop}
\begin{proof}
First, let $p\in U$ be any point.  Write $\ms{M}_m = \ms{M}/\hbar^m\ms{M}$, and let $\ms{M}_0[p]$ denote the fiber of $\ms{M}_0$ at $p$.  We have that $j_*\ms{M}_m$ is a quasi-coherent $\sA$-module; thus, by Proposition \ref{prop:coheretaffinity}, $j_*\ms{M}_m$ is the union of its globally generated subsheaves, hence is itself globally generated.  Thus $\Gamma(\ms{M}_m)\rightarrow \ms{M}_0[p]$ is surjective.  Taking (inverse) limits and applying Theorem 4.5 of \cite{HartshorneAlgDeRham}, we get that $\Gamma(\ms{M})\rightarrow \ms{M}_0[p]$ is surjective.  It follows (by a standard argument) from Nakayama's Lemma that $\Gamma(\ms{M})\otimes\sA\rightarrow \ms{M}_0$ is surjective.  Writing $\operatorname{ev}: \Gamma(\ms{M})\otimes\sA\rightarrow \ms{M}$ for the evaluation map, we get that $\ms{M} = \operatorname{Im}(\operatorname{ev}) + \hbar\ms{M}$, \ie that 
$\ms{M}/\operatorname{Im}(\operatorname{ev}) = \hbar\cdot\big(\ms{M}/\operatorname{Im}(\operatorname{ev})\big)$, and thus by Lemma 
\ref{lem:NAK} that $\operatorname{ev}$ is surjective.
\end{proof}
We remark that, in the proof of Proposition \ref{prop:gg}, we use only quasi-coherence of $j_*\ms{M}_m$ (and not of the naive sheaf-theoretic direct image $j_{\bullet}\ms{M}$, which we expect is not quasi-coherent in general).
\begin{corollary}\label{cor:esssurj}
Let $\resol$ be an  affine variety and $U\subset\resol$ an open subset with complement $C = \resol\smallsetminus U$.  Let $\Coh{\sA}_C$ denote the subcategory of sheaves supported on $C$ (that is, the kernel of the restriction-to-$U$ functor).  Suppose that the induced functor
$\Coh{\sA}/\Coh{\sA}_C\rightarrow \Coh{\sA_U}$ is full.  Then $\Coh{\sA}\rightarrow \Coh{\sA_U}$ is essentially surjective.
\end{corollary}
\begin{proof}
By Proposition \ref{prop:gg}, objects of $\Coh{\sA}$ are globally generated; hence, given $\ms{M}\in\Coh{\sA}$, we may produce a presentation 
$\sA_U^I \xrightarrow{\phi} \sA_U^J \rightarrow \ms{M}\rightarrow 0$ with $I, J$ finite index sets.  It suffices to prove that there are objects $F_1, F_0\in \Coh{\sA}$, a morphism $\widetilde{\phi}: F_1\rightarrow F_0$ and isomorphisms $F_1|_U\cong \sA_U^I, F_0|_U\cong  \sA_U^J$ that identify $\widetilde{\phi}|_U$ with $\phi$: then $\ms{M} \cong \on{coker}(\widetilde{\phi})|_U$.  But since 
$\sA_U^I, \sA_U^J$ are in the essential image of the functor $\Coh{\sA}/\Coh{\sA}_C\rightarrow \Coh{\sA_U}$,
this follows immediately from the fullness hypothesis.
\end{proof}

Recall that a $\C[\![\h]\!]$-module is flat if and only if it is torsion free. The following is a consequence of \cite[Theorem 5.6]{Yakcomplete}. 

\begin{lem}\label{lem:completeA}
Let $\ms{M}$ be a $\h$-adically complete and $\h$-flat $\sA$-module. Let $U \subset \resol$ be an affine open set. Then, $\Gamma(U,\ms{M}) \simeq \Gamma(U,\ms{M}_0) \widehat{\o} \C[\![\h]\!]$ as $\C[\![\h]\!]$-modules \ie $\Gamma(U,\ms{M})$ is $\h$-adically free. 
\end{lem}

Based on Lemma \ref{lem:completeA}, a $\sA$-module $\ms{M}$ is said to be $\h$-adically free if it is $\h$-adically complete and $\h$-flat. At the other extreme, an $\sA$-module $\ms{M}$ is said to be $\h$-torsion if, for each $p \in \resol$, there exists an affine open neighborhood $U$ of $p$ and $N \gg 0$ such that $\h^N \cdot \Gamma(U,\ms{M}) = 0$. Since $\resol$ is assumed to be quasi-compact, this is equivalent to requiring that $\h^N \cdot \ms{M} = 0$ for $N \gg 0$. We define $\Coh{\sA}_{\tor}$ be the full subcategory of $\Coh{\sA}$ consisting of all $\h$-torsion sheaves.

\subsection{Equivariant Algebras and Modules}
\begin{terminology}\label{term:torus}
Let $\sT$ be a torus, \ie $\sT$ is isomorphic to $\Cs^n$ for some $n$. We recall that a representation $M$ of $\sT$ is \textit{pro-rational} if it is the inverse limit of rational $\sT$-modules. Fix a character $\chi$ of $\sT$ and let $\sT$ act on $\C[\![\h]\!]$ by $t \cdot \h = \chi(t) \h$. 
\end{terminology}

Let $(\resol,\omega)$ be an affine symplectic variety with $\Gm$-action.  Assume that $m_t^*\omega = \chi(t)\omega$ for all $t\in\sT$.  
\begin{defn}
A complete $\C[\![\hbar]\!]$-algebra $A$ is said to be {\em $\sT$-equivariant} if $A$ is a pro-rational $\sT$-module such that $g\cdot(ab) = (g\cdot a)(g\cdot b)$, and $g\cdot \h = \chi(g) \h$.  The algebra $A$ is a $\sT$-equivariant deformation quantization of $\C[\resol]$ if $A$ comes equipped with a $\sT$-equivariant isomorphism $A/\hbar A\cong \C[\resol]$.

A finitely generated $A$-module is said to be {\em $\sT$-equivariant} (or just equivariant) if it is a pro-rational $\sT$-module such that the multiplication map $A \o M \rightarrow M$ is equivariant.
\end{defn}

The category of all finitely generated, equivariant $A$-modules is denoted $\Lmod{(A,\sT)}$ and the corresponding ind-category is $\LMod{(A,\sT)}$. The morphisms in these categories are equivariant.
 
\begin{prop}
$\Lmod{(A,\sT)}$ and $\LMod{(A,\sT)}$ are abelian categories.  
 \end{prop}

\begin{lem}\label{lem:Csfinitedim}
Let $M \in \Lmod{(A,\sT)}$. Then, there exists a finite dimensional $\sT$-submodule $V$ of $M$ such that $M = A \cdot V$. 
\end{lem}

\begin{proof}
Nakayama’s lemma implies that if $V$ is any subspace of $M$ whose image in $M / \h M$ generates $M / \h M$, then $V$ generates $M$. As noted in \cite[Section 5.2.1]{GordonLosev}, each $M_n$ is a rational $\sT$-module and $M$ is the inverse of these $\sT$-modules. Since $\sT$ is reductive, we may fix $\sT$-equivariant splittings $M_n = K_n \oplus M_{n-1}$ such that $K_n$ is the kernel of $M_n \twoheadrightarrow M_{n-1}$. This implies that $M = \prod_{n} K_n$ as a $\sT$-module. Hence, if we choose a finite-dimensional $\sT$-submodule $V'$ of $K_0 = M / \h M$ that generates $M/ \h M$ as an $A_0$-module, then we can choose a $\sT$-module lift $V$ of $V'$ in $M$. 
\end{proof}

\subsection{Equivariant DQ-Algebras and Modules}
We maintain Terminology \ref{term:torus}.

Let $(\resol,\omega)$ be any smooth symplectic variety with $\sT$-action: assume $m_t^*\omega = \chi(t)\omega$ for all $t\in\sT$.  
\begin{defn}
A deformation quantization $\sA$ of $\resol$ is said to be {\em $\sT$-equivariant} if it is equipped with the structure of a $\sT$-equivariant sheaf of algebras, with $\sT$ acting on $\C[\![\hbar]\!]$ as in Terminology \ref{term:torus}, so that the $\sT$-action on each $\sA_n$ is rational.  

A coherent $\sA$-module $\ms{M}$ is $\sT$-equivariant if it comes equipped with a $\sT$-equivariant structure making each $\ms{M}_n$ a $\sT$-rational 
$\sA_n$-module.  The category of $\sT$-equivariant coherent $\sA$-modules is $\Coh{(\sA,\sT)}$.
\begin{prop}\label{prop:equiv-coheretaffinity}
Assume that $\resol$ is affine. Then $\Gamma(\resol, -)$ defines an exact equivalence between $\Coh{(\sA,\sT)}$ and the category $\Lmod{(A,\sT)}$ of finitely generated $\sT$-equivariant $A$-modules, where $A = \Gamma(\resol, \sA)$. A quasi-inverse is given by $M \mapsto M^{\Delta}$. Moreover, $M^{\Delta} = \sA \o_A M$.
\end{prop}
\begin{proof}
This is immediate from Proposition \ref{prop:coheretaffinity}.
\end{proof}
\end{defn}

\subsection{Support} 

Let $\ms{M}$ be an $\sA$-module. Then $\Supp\ \ms{M}$ denotes the sheaf-theoretic support of $\ms{M}$ \ie it is the set of all points $x \in \resol$ such that $\ms{M}_x \neq 0$. 

\begin{lem}\label{lem:suppsupp}
Let $\ms{M}$ be a coherent $\sA$-module. Then, $\Supp \ \ms{M} = \Supp \ \ms{M} / \h \ms{M}$. In particular, it is a closed subvariety of $\resol$. 
\end{lem}

\begin{proof}
Since both notions of support are local, we may assume that $\resol$ is affine and set $M = \Gamma(\resol,\ms{M})$, $A = \Gamma(\resol, \sA)$ and $R = A / \h A$. 

\begin{claim}
Let $f \in R$. Then $\Gamma(D(f), \ms{M}) = Q_f^{\mu}(M)$. 
\end{claim}

\begin{proof}
As noted in section \ref{sec:DQsheaf}, it follows from \cite[Theorem 4.2]{VS} that the claim is true when $\ms{M} = \sA$. Since $M$ is finitely generated we may, by Lemma \ref{lem:fgfp}, fix a finite presentation $A^n \rightarrow A^m \rightarrow M \rightarrow 0$ of $M$. Then, the claim follows from the fact that $Q^{\mu}_f( - )$ is exact on finitely generated $A$-modules and the five lemma applied to the diagram 
$$
\xymatrix{
Q_f^{\mu}(A^n) \ar[r] \ar[d]_{\wr} & Q_f^{\mu}(A^m) \ar[r] \ar[d]_{\wr} & \ar[d] \ar[r] Q_f^{\mu}(M) & 0 \\
\Gamma(D(f),\sA^n) \ar[r] & \Gamma(D(f),\sA^m) \ar[r] & \Gamma(D(f),\ms{M}) \ar[r] & 0 
}
$$
This completes the proof of the claim. 
\end{proof} 

For each $f \in R$, the short exact sequence $0 \rightarrow \h M \rightarrow M \rightarrow M / \h M \rightarrow 0$ gives 
$$
0 \rightarrow Q^{\mu}_f(\h M) \rightarrow Q_f^{\mu}(M) \rightarrow Q^{\mu}_f(M / \h M) \rightarrow 0.
$$
Therefore, $Q^{\mu}_f(M / \h M) \neq 0$ implies that $Q_f^{\mu}(M) \neq 0$. On the other hand, if $Q_f^{\mu}(M) \neq 0$ then 
$$
Q^{\mu}_f(\h M) = Q^{\mu}_f(A) \o_A \h M = \h Q^{\mu}_f(A) \o_A M = \h Q^{\mu}_f(M). 
$$
Since $Q^{\mu}_f(M)$ is a finitely generated $Q^{\mu}_f(A)$-module and $Q^{\mu}_f(A)$ is $\h$-adically complete, Nakyama's lemma implies that  $\h Q_f^{\mu}(M)$ is a proper submodule of $Q_f^{\mu}(M)$. Thus, $\Gamma(D(f),\ms{M}) \neq 0$ if and only if $\Gamma(D(f), \ms{M}/ \h \ms{M}) \neq 0$. The lemma is a direct consequence. 
\end{proof}

\subsection{$\cW$-Algebras and Good Modules}
Let $\sA$ be a $\sT$-equivariant DQ-algebra on $\resol$. Then, $\cW := \sA[\h^{-1}]$ is a sheaf of $\C(\!(\h)\!)$-algebras on $\resol$, it is the $\cW$-algebra associated to $\sA$. Base change defines a functor $\Coh{(\sA,\sT)} \rightarrow \Lmod{(\cW,\sT)}$, $\ms{M} \mapsto \ms{M}[\h^{-1}]$. Let $U\subset \resol$ be a $\sT$-stable open subset and $\ms{M}$ be a $\cW_U$-module. A \textit{lattice} for $\ms{M}$ is a coherent $\sA_U$-submodule $\ms{M}'$ of $\ms{M}$ such that $\ms{M}'[\h^{-1}] = \ms{M}$; it is a {$\sT$-lattice} if it is a $\sT$-equivariant coherent module. The category of all $\sT$-equivariant $\cW_U$-modules that admit a (global) $\sT$-lattice is denoted $\Good{(\cW_U, \sT)}$, and we refer to a module in this category as a \textit{good} ($\sT$-equivariant) $\cW_U$-module. Recall that $\Coh{(\sA_U,\sT)}_{\tor}$ denotes the full subcategory of $\Coh{(\sA_U,\sT)}$ consisting of all $\h$-torsion sheaves. 

\begin{prop}\label{prop:quot-cat}
The category $\Coh{(\sA_U,\sT)}_{\tor}$ is a Serre subcategory of $\Coh{(\sA_U,\sT)}$ and we have an equivalence of \textit{abelian} categories 
$$
\Coh{(\sA_U,\sT)} / \Coh{(\sA_U,\sT)}_{\tor} \simeq \Good{(\cW_U,\sT)}. 
$$
\end{prop}

\subsection{Restriction and Quotient Categories}\label{sec:fullj} Let ${\resol}$ be a smooth symplectic manifold with $\Cs$-action of positive weight. Let $Z \subset \resol^{\Cs}$ be a closed, connected and smooth subvariety. Let $\displaystyle C = \{ x \in {\resol} \ | \ \lim_{t \rightarrow \infty} t \cdot x \in Z \}$ be the attracting locus for $Z$; it is a smooth, locally closed subvariety of $\resol$.  

Assume that $C$ is closed in $\resol$.  The complement to $C$ in ${\resol}$ is denoted $U$ and we write $j  : U \hookrightarrow {\resol}$ for the embedding.  In this section we prove the following.

\begin{thm}\label{keyprop}
Suppose that $C\subset \resol$ is closed and let $U = \resol\smallsetminus C$.  The functor $j^*$ induces an equivalence 
\begin{equation}\label{restriction-equiv}
\Good{\cW_{\resol}} / \Good{\cW_{\resol}}_C \stackrel{\sim}{\longrightarrow} \Good{\cW_U}.
\end{equation} 
\end{thm}

The remainder of this section is devoted to the proof of Theorem \ref{keyprop}. First we note an immediate corollary. The  Ind-category of $\Good{(\cW,\sT)}$ is denoted $\Qcoh(\cW,\sT)$, or $\Qcoh(\cW)$ if $\sT$ is understood from context; we call it (abusively) the {\em quasicoherent} category.  We say a quasicoherent object has support in a closed subset $K\subset\resol$ if one can write $\displaystyle \ms{M} = \lim_{\longrightarrow} \ms{M}_i$ where each $\ms{M}_i$ is good and has support in $K$.  We write $\Qcoh(\cW)_K$ for the full subcategory whose objects have support in $K$.   

Recall (for example, from \cite{KS}) that a full subcategory of a Grothendieck category is called \textit{localizing} if it is closed under subobjects, quotients, extensions and small inductive limits.

\begin{corollary}\label{cor:qcohfullesssurj}
Let $C\subset \resol$ be a closed subset as above.  The functor 
$$
j^*:  \Qcoh(\cW_{\resol},\sT)\longrightarrow \Qcoh(\cW_{U},\sT)
$$
is essentially surjective and induces an equivalence
\bd
 \Qcoh(\cW_{\resol},\sT)/\Qcoh(\cW_{\resol},\sT)_{C} \simeq \Qcoh(\cW_{U}).
 \ed
Moreover, $j^*$ admits a right adjoint.
\end{corollary}
\begin{defn}
We write
\begin{displaymath}
 j_*: \Qcoh(\cW_{U})\longrightarrow  \Qcoh(\cW_{\resol})
 \end{displaymath}
 to denote the right adjoint of $j^*$, taking care to note that it {\em need not} be identified with the sheaf-theoretic direct image.
 \end{defn}
\begin{proof}[Proof of Corollary \ref{cor:qcohfullesssurj}]
Essential surjectivity and equivalence are immediate from the theorem.  The existence of a right adjoint follows since the kernel of $j^*$, \ie the subcategory of modules supported on $U$, is a localizing subcategory.
\end{proof}

We begin the proof of Theorem \ref{keyprop}.
The main part of the proof will show that the faithful functor $\Good{\cW_{\resol}} / \Good{\cW_{\resol}}_C \longrightarrow \Good{\cW_U}$ is full.  

We begin with a lemma that will allow us to reduce to affine statements.
\begin{lemma}\label{lem:limiting-affine}
Suppose that $Z\subseteq \resol^{\Gm}$ is connected and closed, and that 
$$
\displaystyle C = \{ x \in {\resol} \ | \ \lim_{t \rightarrow \infty} t \cdot x \in Z \}
$$
is closed in $\resol$.  If $U\subseteq \resol$ is an affine open subset of $\resol$ for which $Z\cap U \neq\emptyset$, then
\bd
C\cap U  = \big\{ x \in U \ \big| \ \lim_{t \rightarrow \infty} t \cdot x \in Z\cap U \big\}.
\ed
In particular, $\displaystyle \lim_{t \rightarrow \infty} t \cdot x$ exists in $U$ for every $x\in C\cap U$.
\end{lemma}
\begin{proof}
Note that $Z^\circ = Z\cap U$ is open in $Z$.  Let
\bd
C^\circ =  \big\{ x \in U \ \big| \ \lim_{t \rightarrow \infty} t \cdot x \in Z^\circ \big\}.
\ed
Then $C^\circ$ is the preimage, under the projection morphism $C\rightarrow Z$, of the dense open set $Z^\circ$; hence $C^\circ$ is dense in $C$.  Suppose that $f\in \mathbb{C}[U]$ is a $\Gm$-semi-invariant, say $f(t \cdot x) = t^{-d} f(x)$ for all $t\in \Gm$, $x\in U$.  If $x\in C^\circ$, then 
$\displaystyle \lim_{t\rightarrow 0} t^d f(x) = f \left(\lim_{t\rightarrow \infty} t\cdot x\right)$, so if $d<0$ then $f(x) = 0$.  Thus any $f\in \mathbb{C}[U]$ of negative weight vanishes on $C^\circ$ and consequently (by density) vanishes on $C\cap U$.  It follows that $\mathbb{C}[C\cap U]$ has non-negative $\Gm$-weights; since $C\cap U$ is closed in $U$, hence affine, we conclude that the $\Gm$-action on $C\cap U$ extends to an action of the monoid $\mathbb{A}^1$ on $C\cap U$, proving the lemma.
\end{proof} 

Returning to the proof of the theorem, we claim that, if we assume fullness of \eqref{restriction-equiv}, essential surjectivity follows from Corollary \ref{cor:esssurj}.  Indeed, to prove essential surjectivity, it suffices to replace $\resol$ by any $\Gm$-stable open subset of $\resol$ that contains $C$.  Thus, choose a collection $\{\resol_i\}$ of $\Gm$-stable affine open subsets of $\resol$ whose union contains $C$.  Then, by Lemma \ref{lem:limiting-affine},  $C\cap \resol_i \subset \resol_i$ is a closed subset satisfying the hypotheses of the theorem, and so the fullness assertion holds for restriction from $\resol_i$ to $U_i = \resol_i\smallsetminus C$.  Corollary \ref{cor:esssurj} thus implies that for every coherent $\sA_{U_i}$-module $\ms{M}_i$, there is a coherent $\sA_{\resol_i}$-module $\overline{\ms{M}}_i$ and an isomorphism $\overline{\ms{M}}_i|_{U_i} \cong \ms{M}_i$.  A standard gluing argument then shows that every coherent $\sA_U$-module extends to a coherent $\sA$-module, proving essential surjectivity.

Thus, we return to the proof of fullness of \eqref{restriction-equiv}.  We note that taking a covering of ${\resol}$ by affine open $\Cs$-stable sets, the sheaf property implies that the fullness statement is local. Therefore we may assume that ${\resol}$ is affine.
 Shrinking ${\resol}$ if necessary, we may assume that $C = Z(f_1, \ds, f_k)$ is a complete intersection in ${\resol}$ of codimension $k$, where each $f_i$ is homogeneous with respect to $\Cs$.  As in Lemma \ref{lem:limiting-affine}, if $f \in \mc{O}({\resol})$ is homogeneous of negative weight with respect to $\Cs$, then $f \in I(C)$. The fact that ${\resol}$ is affine implies that we can (and will) fix an identification $\mc{A} = \mc{O}_{\resol}[\![\h]\!]$ of pro-rational sheaves of $\C[\![\h]\!]$-modules. Notice that for any $\Cs$-stable affine open subset $V$ of ${\resol}$, the identification gives a canonical identification $\mc{A}(V) = \mc{O}(V)[\![\h]\!]$. Given $f \in \mc{O}({\resol})$, let $\nu(f)$ denote the corresponding section of $\mc{A}({\resol})$ under this identification.

 Let $W_{\rat}$ denote the $\C[\h,\h^{-1}]$-subalgebra of $\Gm$-rational sections in $\Gamma({\resol},\cW_{\resol})$. Given a $\Cs$-equivariant $W$-module $M$, let $M_{\rat}$ denote the $W_\rat$-submodule of all rational sections. We say that a $W_\rat$-module $M$ is supported on $C$ if, for each section $m \in M$, there exists $N \gg 0$ such that $\nu(f_i)^N \cdot m = 0$ for all $i$. 

\begin{lem}\label{lem:ratC}
Suppose $\resol$ is affine.
\begin{enumerate}
\item The functor $\ms{M} \mapsto \Gamma({\resol},\ms{M})_{\rat}$ is an equivalence of categories $R : \Good{\cW_{\resol}} \stackrel{\sim}{\longrightarrow} W_{\rat}\text{-}\mathsf{mod}$.
\item Under the equivalence of (1),  $\ms{M}$ is supported on $C$ (in the usual sense) if and only if $\Gamma({\resol},\ms{M})_{\rat}$ is supported on $C$ (in the above sense). 
\end{enumerate}
\end{lem}

\begin{proof}
Part (1) follows from Propositions \ref{prop:equiv-coheretaffinity} and \ref{prop:quot-cat} by a standard argument (see the proof of Proposition \ref{prop:twistedequiv}). 

(2) If $\ms{M}_0$ is a coherent $\mc{A}$-submodule of $\ms{M}$ such that $\ms{M} = \ms{M}_0[\h^{-1}]$, then $M= M_0[\h^{-1}]$, where $M = \Gamma({\resol},\ms{M})_\rat$ and $M_0 = \Gamma(X,\ms{M}_0)_{\rat}$. Certainly, if $m \in M_0$ and $\nu(f_i)^N \cdot m = 0$ then $f_i^N \cdot \overline{m} = 0$ in $M_0 / \h M_0$. Hence $\ms{M}$ is supported on $C$ in the usual sense. 

We need to check the converse. So assume that $\ms{M}$ is supported on $C$ in the usual sense. Our assumptions on $C$ imply that $\mc{O}({\resol}) = \C[f_1, \ds, f_k, x_1, \ds, x_l] / I$, where $\deg x_i \ge 0$ and $I$ is a homogeneous ideal. Hence $\C[C]$, a quotient of the algebra $\C[f_1, \ds, f_k, x_1, \ds, x_l]$, is non-negatively graded. This implies that the finitely generated $\C[\resol]$-module $M_0 / \h M_0$ has its grading bounded from below. Since $\h$ has positive weight, the same applies to $M_0$. Let $m \in M_0$ be a homogeneous section. If $f_i^N \cdot \overline{m} = 0$ in $M_0 / \h M_0$, then $\nu(f_i)^N \cdot m \in \h M_0$ and hence $\nu(f_i)^{rN} \cdot m \in \h^r M_0$. On the other hand, $\deg f_i \le 0$ and hence $\deg(\nu(f_i)^{rN} \cdot m) = r N \deg \nu(f_i) + \deg m \le \deg (m)$. This implies that $\nu(f_i)^{rN} \cdot m = 0$ for $r,N \gg 0$, since the weights of all homogeneous elements in $\h^r M_0$ will be greater than $\deg m$ for $r \gg 0$. 
\end{proof}   

Write $U_{\alpha} = {\resol} \smallsetminus Z(f_{\alpha})$ and $U_{\alpha_0, \ds, \alpha_i} = U_{\alpha_0} \cap \cdots \cap U_{\alpha_i}$. Given $\ms{M} \in \Good{\cW_U}$, define the complex 
$$
\cech^i(\ms{M}) = \prod_{\alpha_0 < \cdots < \alpha_i} \Gamma(U_{\alpha_0, \ds, \alpha_i},\ms{M}),
$$
with the usual differential 
$$
d^i : \cech^i(\ms{M}) \rightarrow \cech^{i+1}(\ms{M}), \quad d^i(f_{\alpha_0, \ds, \alpha_i})_{\alpha_0, \ds,\beta,\ds, \alpha_i} = (-1) f_{\alpha_0, \ds, \alpha_i} |_{U_{\alpha_0, \ds,\beta,\ds, \alpha_i}}.
$$
 For $\ms{M} \in \Good{\cW_U}$, we define $F(\ms{M}) = \cech^{\bullet}(\ms{M})_{\rat}$. There is a canonical transformation $R \rightarrow F( - |_U)$, where we identify $W_\rat\text{-}\mathsf{Mod}$ with complexes concentrated in degree zero. 

\begin{lem}\label{lem:Coneexact}
$\Cone(R \rightarrow F( - |_U))$ defines an exact functor from $\Good{\cW_U}$ to complexes with terms in $W_\rat\text{-}\mathsf{Mod}$.
\end{lem}

\begin{proof}
The cone $\Cone(R \rightarrow F( - |_U))$ is exact if and only if both $R$ and $F( - |_U)$ are exact. The functor $R$ is exact by Lemma \ref{lem:ratC}. Therefore it suffices to show that the functor $F$ defines an exact functor from $\Good{\cW_U}$ to complexes with terms in $W_\rat\text{-}\mathsf{Mod}$. The exactness of $\ms{M} \mapsto \cech^{\bullet}(\ms{M})$ can be checked term by term. But it is clear that $\ms{M} \mapsto \cech^i(\ms{M})$ is exact because the open set $U_{\alpha_0, \ds,\alpha_i}$ is affine. Therefore to show that $F$ is exact, it suffices to show that the functor $\Good{\cW_{U_{\alpha_0, \ds,\alpha_i}}} \rightarrow W_{\rat}\text{-}\mathsf{Mod}$, $\ms{M} \mapsto \Gamma(U_{\alpha_0, \ds,\alpha_i},\ms{M})_{\rat}$ is an exact functor. Since $U_{\alpha_0, \ds,\alpha_i}$ is affine, this follows from Lemma \ref{lem:ratC}(1). 
\end{proof}

\begin{lem}\label{lem:belongWtoQcoh}
The cohomology of the cone $D^{\bullet} = \mathrm{Cone}(W_\rat \rightarrow F(\cW_U))$ is zero outside degree $k$. The group $H^k(D^{\bullet})$ is supported on $C$. 
\end{lem}

\begin{proof}
Notice that the differentials are $\C[\![\h]\!]$-linear and preserve the lattice defined by $\mc{A}$. Therefore $D^{\bullet} = D^{\bullet}_0[\h^{-1}]$ and flat base change implies that $H^i(D^{\bullet}) = H^i (D^{\bullet}_0)[\h^{-1}]$. Thus, it suffices to check that the corresponding statements hold for $D^{\bullet}_0$. Write $D^{\bullet}_{com}$ for $\mathrm{Cone}(\mc{O}(X) \rightarrow \cech^{\bullet}(\mc{O}_U))$. Since we have an identification $\mc{A}  = \mc{O}_X[\![\h]\!]$ of prorational sheaves, we have an identification of complexes $D^{\bullet}_0 = D^{\bullet}_{com}[\![\h]\!]$ and hence $H^i(D^{\bullet}_0) = H^i(D^{\bullet}_{com})[\![\h]\!]$, as prorational $\C[\![\h]\!]$-modules. Since $H^i(D^{\bullet}_{com}) = 0$ for $i \neq 0$, the first claim follows. 

\begin{claim}\label{claim:cohomoweight}
The space $H^k(D^{\bullet}_{com})$ is non-negatively graded as a $\Cs$-module. 
\end{claim}
\begin{proof}[Proof of Claim \ref{claim:cohomoweight}]
As in the proof of Lemma \ref{lem:ratC}, $\C[\resol] =[f_1, \ds, f_k, x_1, \ds, x_l] / I$ for some homogeneous ideal $I$. Then 
$$
H^k(D^{\bullet}_{com}) = \C[\resol]_{f_1 \cdots f_k} / \sum_i \C[\resol]_{f_1 \cdots \widehat{f}_i \cdots f_k}
$$ 
is a quotient of $\C[f_1, \ds, f_k, x_1, \ds, x_l]_{f_1 \cdots f_k} / \sum_i \C[f_1, \ds, f_k, x_1, \ds, x_l]_{f_1 \cdots \widehat{f}_i \cdots f_k}$. The latter clearly has the desired properties.  This proves the claim. 
\end{proof}
We return to the proof of Lemma \ref{lem:belongWtoQcoh}.
To prove the second assertion of the lemma, we may assume that $m$ belongs to $H^k(D^{\bullet}_0)$ and is rational. The image of $m$ in $H^k(D^{\bullet}_{com}) = H^k(D^{\bullet}_0) / \h H^k(D^{\bullet}_0)$ is torsion with respect to the $f_i$. Therefore there exists $N$ such that $\nu(f_i)^N \cdot m \in \h H^k(D^{\bullet}_0)$ and hence $\nu(f_i)^{rN} \cdot m \in \h^r H^k(D^{\bullet}_0)$. Claim \ref{claim:cohomoweight} implies that the weight of every homogeneous element in $H^k(D^{\bullet}_0)$ is non-negative. Thus, every homogeneous element in $\h^r H^k(D^{\bullet}_0)$ has degree at least $r$. On the other hand $\deg (f_i) \le 0$ and hence $\deg(\nu(f_i)^N \cdot m) = N \deg(\nu(f_i)) + \deg(m) \le \deg (m)$. This implies that $\nu(f_i)^N \cdot m = 0$ for $N \gg 0$. 
\end{proof}

\begin{prop}\label{prop:Csupportrat}
The cohomology of $\Cone(R(\ms{M}) \rightarrow F( \ms{M} |_U))$ is supported on $C$ for any $\ms{M} \in \Good{\cW_{\resol}}$. 
\end{prop}

\begin{proof}
Since $\Good{\cW_{\resol}}$ has finite homological dimension, we prove the claim by induction on projective dimension. Certainly Lemma \ref{lem:belongWtoQcoh} implies that the claim holds for every summand of $\cW_{\resol}^N$ (for $N$ finite).  By Lemma \ref{lem:Coneexact}, the functor $\Cone(R \rightarrow F( - |_U))$ is exact. Therefore the long exact sequence in cohomology implies that if it holds for all modules of projective dimension $i$, then it also holds for all modules of projective dimension $i+1$. 
\end{proof}

Finally, we show that the faithful functor $\Good{\cW_{\resol}} / \Good{\cW_{\resol}}_C \rightarrow \Good{\cW_U}$ is also full. This will complete the proof of Theorem \ref{keyprop}. Thus, suppose $\ms{M}, \ms{N}\in \Good{\cW_{\resol}}$ and $\phi \in \Hom_{\Good{\cW_U}}(\ms{M} |_U,\ms{N} |_U)$.  Applying the functors $R$ and $F$ from above, we get a diagram
\begin{displaymath}
\xymatrix{R(\ms{M}) \ar[d] & & R(\ms{N})\ar[d] \\ H^0\big(F(\ms{M}|_U)\big)\ar[rr]^{H^0\big(F(\phi)\big)} & & H^0\big(F(\ms{N}|_U)\big).}
\end{displaymath}
By Proposition \ref{prop:Csupportrat}, the vertical arrows have kernel and cokernel supported on $C$.  Letting $N'$ be any finitely generated $W_{\rat}$-submodule of $H^0\big(F(\ms{N}|_U)\big)$ that contains the images of $R(\ms{M})$ and $R(\ms{N})$, we get a diagram 
$$
R(\ms{M})\xrightarrow {\widetilde{\phi}'} N' \leftarrow R(\ms{N})
$$
where $R(\ms{N})\rightarrow N'$ has kernel and cokernel supported on $C$.  Applying $R^{-1}$, we get a diagram of good $\cW_{\resol}$-modules  $\ms{M}\xrightarrow{\phi'} \ms{N}' \leftarrow \ms{N}$, where the support assertion guarantees that $\ms{N}\rightarrow \ms{N}'$ becomes an isomorphism in the quotient category.  Thus $\phi'$ defines a morphism $\ms{M}\rightarrow \ms{N}$ in the quotient category, whose restriction to $U$ is, by construction, identified with $\phi$.  This proves Theorem \ref{keyprop}. \hfill\qedsymbol

\subsection{Holonomic Modules}

By Gabber's theorem, the support of any good $\cW$-module has dimension at least $\dim \resol / 2$. A good $\cW$-module is said to be holonomic if the dimension of its support is exactly $\dim \resol / 2$. The category of holonomic $\cW$-modules is denoted $\Hol{\cW}$. The theory of characteristic cycles implies:

\begin{lem}
Let $\ms{M}$ be a holonomic $\cW$-module. Then, $\ms{M}$ has finite length. 
\end{lem}

A $\cW$-module $\ms{M}$ is said to be \textit{regular holonomic} if there exists a lattice $\ms{M}'$ of $\ms{M}$ such that the support of $\ms{M}' / \h \ms{M}'$ is reduced. The category of regular holonomic $\cW$-modules is denoted $\Reghol{\cW}$. 

\subsection{Equi-dimensionality of Supports}

In this subsection we note that the analogue of the Gabber-Kashiwara equi-dimensionality theorem holds for $\cW$-algebras. First, given a coherent $\sA$-module $\ms{M}$, the sheaf $\bigoplus_{n \ge 0} \h^n \ms{M} / \h^{n+1} \ms{M}$ is a coherent $\mc{O}_{ \resol}[\h]$-module. Therefore, its support is a closed subvariety of $\resol \times \mathbb{A}^1$, which we will denote $\widetilde{\Supp} \ \ms{M}$. 

\begin{lem}\label{lem:tildesupp}
Let $p :  \resol \times \mathbb{A}^1 \rightarrow \resol$ be the projection map. Then, $\Supp \ \ms{M} = p \left( \widetilde{\Supp} \ \ms{M} \right)$. 
\end{lem}

\begin{proof}
Since $\ms{M}$ is coherent and support is a local property, we may assume that $ \resol$ is affine and set $M = \Gamma( \resol,\ms{M})$, an $A = \Gamma( \resol,\sA)$-module. Both $\Supp$ and $\widetilde{\Supp}$ are additive on short exact sequences, therefore the sequence $0 \rightarrow M_{\tor} \rightarrow M \rightarrow M_{t.f.} \rightarrow 0$ implies that we may assume that $M$ is either $\h$-torsion or $\h$-torsion free. First, if $M$ is $\h$-torsion free, then $M \simeq (M / \h M) \widehat{\o} \C[\![\h]\!]$ as a $\C[\![\h]\!]$-module. This implies that $\widetilde{\Supp} \ms{M} = \Supp \ms{M} \times \mathbb{A}^1$. On the other hand, if $M$ is $\h$-torsion, then $\gr M := \bigoplus_{n \ge 0} \h^n M / \h^{n+1} M$ equals $\bigoplus_{n = 0}^N \h^n M / \h^{n+1} M$ for some $N$. If $I \cdot (M / \h M ) = 0$ for some $I \lhd \C[ \resol]$, then $I \cdot \gr M = 0$. Since $\h^{N+1} \gr M = 0$ too, this implies that $\widetilde{\Supp} \ \ms{M} = \Supp \ms{M} \times \{ 0 \}$.
\end{proof}

Let $\dim  \resol = 2 m$. The analogue of the Gabber-Kashiwara theorem reads

\begin{thm}\label{thm:equisupport}
Let $\ms{M}$ be a good $\cW$-module. Then, there exists a unique filtration 
$$
0 = D_{m-1}(\ms{M}) \subset D_m(\ms{M}) \subset D_{m+1}(\ms{M}) \subset \cdots \subset D_{2m} (\ms{M}) = \ms{M}
$$
such that $\Supp D_i(\ms{M}) / D_{i-1}(\ms{M})$ is pure $i$-dimensional. 
\end{thm}

\begin{proof}
We fix a lattice $\ms{M}'$ of $\ms{M}$, it is a coherent torsion-free $\sA$-module. We will show that the analogue of the above statement holds for $\ms{M}'$, then set $D_i(\ms{M}) = D_i(\ms{M}')$ and check that $D_i(\ms{M})$ is independent of the choice of lattice. The uniqueness property actually implies that the statement will hold globally on $ \resol$ if it holds locally, therefore we may as well assume that $ \resol$ is affine and set $M' = \Gamma( \resol,\ms{M}')$. Since $M$ is $\h$-torsion free, the proof of Lemma \ref{lem:tildesupp} shows that $\widetilde{\Supp} \ms{M} = \Supp \ms{M} \times \C$. Then, noting this fact, the required result is \cite[Theorem V8, page 342]{CharCycles}. 

Therefore, it suffices to show that $D_i(\ms{M})$ is independent of the choice of lattice. Let $M''$ be another choice of lattice and $N \gg 0$ such that $\h^N M'' \subset M'$.  As explained in \cite[Section 1]{CharCycles}, $D_i(M') := \{ m \in M' \ | \ \dim \gr (A \cdot m) \le i + 1 \}$. Therefore, if $m \in D_i(M'')$ then $\h^N m \in D_i(M')$, which implies that $\h^N D_i(M'') \subset D_i(M')$. By symmetry, we have $\h^{N'} D_i(M') \subset D_i(M'')$ and hence $D_i(M'')[\h^{-1}] = D_i(M')[\h^{-1}]$. 
\end{proof}

As a corollary of the theorem, we can strengthen our extension result. Again, let $\sT$ be a torus. 

\begin{corollary}
Let $U$ be a $\sT$-stable open subset of $\resol$ whose complement is a union of coisotropic cells and $\ms{M}$ a holonomic, $\sT$-equivariant $\cW |_U$-module. Then, there exists a holonomic, $\sT$-equivariant $\cW$-module $\ms{M}'$ such that $\ms{M}' |_U \simeq \ms{M}$. Moreover, if $\ms{M}$ is simple then there exists a unique simple extension $\ms{M}'$. 
\end{corollary}

\begin{proof}
Let $\ms{N}$ be a lattice of $\ms{M}$. By Theorem \ref{keyprop}, there exists a coherent $\sT$-equivariant $\sA$-module $\ms{N}'$ such that $\ms{N}' |_U = \ms{N}$. Replacing $\ms{N}'$ by its torsion free quotient, we may assume that $\ms{N}'$ is torsion free and set $\ms{M}'' = \ms{N}'[\h^{-1}]$. Then, let $\ms{M}' = D_m(\ms{M}'')$, a holonomic submodule of $\ms{M}''$. Since $\ms{M}$ is holonomic, $D_m(\ms{M}) = \ms{M}$. The uniqueness of $D_m( - )$ implies that $\ms{M}' |_U = D_m(\ms{M}'') |_U = D_m(\ms{M}' |_U) = \ms{M}$. 

If $\ms{M}$ is assumed to be simple then there is some simple subquotient of $\ms{M}'$ whose restriction to $U$ is isomorphic to $\ms{M}$. In order to show uniqueness of the extension, let $\ms{M}^1$ and $\ms{M}^2$ be two simple extensions of $\ms{M}$. Denote by $j$ the open embedding $U \hookrightarrow \resol$ and by $j_*$ the right adjoint to $j^*$ whose existence is established in Corollary \ref{cor:qcohfullesssurj}. Then for each $i=1,2$ the canonical adjunction $\ms{M}^i \rightarrow j_* \ms{M}$ is an embedding because it is an isomorphism over $U$ (and hence non-zero) and $\ms{M}^i$ is assumed to be simple. Therefore $\ms{M}^1 \cap \ms{M}^2$ is a $\sA$-submodule of $\ms{M}^i$ whose restriction to $U$ is $\ms{M}$. Thus, $\ms{M}^1 = \ms{M}^2$. 
\end{proof}

\begin{remark}
Let $\ms{M}$ be a simple $\cW$-module or a primitive quotient of $\cW$. Then Theorem \ref{thm:equisupport} implies that the support of $\ms{M}$ is equi-dimensional. A proof of the analogous result in the setting of localization via $\Z$-algebras was recently given in the paper by Gordon and Stafford  \cite{GS}. 
\end{remark}

\section{Quantum Coisotropic reduction}\label{sec:QCR}

In this section we consider the process of quantum coisotropic reduction. Our main result is that quantum coisotropic reduction can be used to prove an analogue of Kashiwara's equivalence for DQ-modules supported on a coisotropic stratum. At the end of the section we consider $\cW$-algebras on a symplectic manifold $\resol$ with an elliptic $\Cs$-action such that $Y = \resol^{\Cs}$ is connected. We show that $\Qcoh (\cW)$ is equivalent to the category of quasi-coherent sheaves for a sheaf of filtered $\mc{O}_Y$-algebras quantizing $\resol$. These filtered $\mc{O}_Y$-algebras behave much like the sheaf of differential operators on $Y$. 

\begin{notation} Throughout the remainder of the paper, $\Good{\cW}$ and $\Qcoh (\cW)$ will denote the category of good, $\Cs$-\textit{equivariant} $\cW$-modules and the Ind-category of good, $\Cs$-\textit{equivariant} $\cW$-modules respectively. Moreover, all $\cW$-modules that we consider will be assumed to be $\Cs$-equivariant.
\end{notation}

\subsection{Quantum Coisotropic Reduction (Local Case)}\label{sec:QCRlocal} We maintain the notation and conventions of Sections \ref{sec:symplecticvarietiessection} and \ref{sec:DQmod}. 

 Thus, let $R$ be a regular affine $\C$-algebra equipped with a Poisson structure $\{-,-\}$ making $\mathfrak{X} = \on{Spec}(R)$ into a smooth affine symplectic variety.  We assume in addition that $R$ comes equipped with a $\Gm$-action for which the Poisson structure has weight $-l$.
 
 Let $A$ be a deformation-quantization of $R$, equipped with an action of $\Cs$ that preserves the central subalgebra $\C[\![\hbar]\!]\subset A$ and has $\h$ as a weight vector. The quotient map $\rho : A \rightarrow R$ is equivariant. The fact that the Poisson bracket $\{ - , - \}$ on $R$ is graded of degree $-l$ implies that $\h$ has weight $l$. 
 Let $J$ be the left ideal generated by all homogeneous elements in $A$ of negative degree. 
 
 As in Section \ref{sec:localalg}  we have  $I = \rho(J)$, the ideal in $R$ generated by all homogeneous elements of strictly negative degree.  
  We write  $C  = \on{Spec}(R/I)$, a closed coisotropic subvariety of $\mathfrak{X}$.  We write $Y  = C^{\Gm}$, the $\Gm$-fixed locus of the coisotropic subset $C$.  As in Section \ref{sec:localalg}, we also assume that $Y$ has been shrunk suitably so that the affine bundle $\rho:C\rightarrow Y$ of \cite{BBFix} splits equivariantly as
  \bd
  \phi: C \stackrel{\sim}{\longrightarrow} Y\times V \times Z,
  \ed
  where $V$ is isomorphic to a vector space with $\Gm$-weights lying in $[-l, -1]$ and $Z$ is isomorphic to a vector space with $\Gm$-weights  less than
  $-l$.  
 \begin{notation}
 Let $\widehat{R}$ denote the completion of $R$ with respect to the ideal $I$, and
 $\widehat{A}$ the completion of $A$ with respect to the two-sided ideal $K := \rho^{-1}(I)$.  
 \end{notation}

\begin{lem}\label{lem:flat}
Let $\widehat{A}$ be as above. 
\begin{enumerate}
\item The algebra $\widehat{A}$ is flat over $A$ and Noetherian. 
\item If $M$ is a finitely generated $A$-module then $\displaystyle\lim_{\longleftarrow}\, (M / K^n \cdot M) \simeq \widehat{A} \o_A M$.
\item The algebra $\widehat{A}$ is $\h$-adically free and $\widehat{A} / \h \widehat{A} \simeq \widehat{R}$. 
\item The algebra $\wh{A}$ is an equivariant quantization of $\wh{R}$ with the Poisson structure on $\wh{R}$ induced from the Poisson structure on $R$. 
\end{enumerate}
\end{lem}
Parts (1) and (2) of the lemma have also been shown by Losev \cite{FiniteDimRepsLosev} using a different argument. 
\begin{proof}
Take $u_1 = \h$ and let $u_2, \dots, $ be arbitrary lifts in $A$ of a set of generators of $I$; these form a normalizing sequence of generators as defined in 
 \cite[Theorem 4.2.7]{MR}, and hence by that theorem the ideal $K$ satisfies the Artin-Rees property in $A$.  Then the proofs of \cite[Lemma~7.15 and Theorem 7.2b]{Eisenbud} apply also in the non-commutative case to imply that $\widehat{A}$ is flat over $A$.  
 The fact that $\widehat{A}$ is Noetherian follows from part (3) which implies that $\gr \widehat{A} \simeq \widehat{R}[\h]$, a Noetherian ring. 

Therefore we need to establish that $\widehat{A} / \h \widehat{A} \simeq \widehat{R}$ and that $\widehat{A}$ is a complete, flat $\h$-module. By Lemma \ref{lem:hcomp1}, and using the fact that inverse limits commute, we have
\begin{equation}\label{eq:swaplim}
\widehat{A} = \lim_{\infty \leftarrow n} A / K^n = \lim_{\infty \leftarrow n} \left( \lim_{\infty \leftarrow s} (A / K^n) / \h^s ( A / K^n) \right) = \lim_{\infty \leftarrow s} \left( \lim_{\infty \leftarrow n} (A / K^n) / \h^s ( A / K^n) \right) ,
\end{equation}
which implies that $\wh{A}$ is $\h$-adically complete. By part (1), $\widehat{A}$ is $A$-flat, hence {\em a fortiori} it is $\h$-flat. Consider the short exact sequence 
$$
0 \rightarrow \left\{ \frac{\h A + K^n}{K^n} \right\}_n \rightarrow \{ A / K^n \}_n \rightarrow \left\{ \frac{A}{\h A + K^n } \right\}_n \rightarrow 0
$$
of inverse systems. Since $\frac{\h A + K^n}{K^n} \rightarrow \frac{\h A + K^{n-1}}{K^{n-1}}$ is surjective, the inverse system $\left\{ \frac{\h A + K^n}{K^n} \right\}_n$ satisfies the Mittag-Leffler condition. Therefore, we have 
\begin{align*}
\left( \lim_{\infty \leftarrow n} A/K^n \right) \Big/ \left( \lim_{\infty \leftarrow n} \h (A / K^n) \right) & = \left( \lim_{\infty \leftarrow n} A/K^n \right) \Big/ \left( \lim_{\infty \leftarrow n} (\h A  + K^n ) / K^n \right) \\
 & = \lim_{\infty \leftarrow n} A / (K^n + \h A) = \lim_{\infty \leftarrow n} R / I^n,
\end{align*}
where the final equality follows from the fact that $\h \in K$ and hence $(K^n + \h A) / \h A$ equals $I^n$. 

The only thing left to show, in order to conclude that $\wh{A}$ is a quantization of $\wh{R}$, is that the Poisson bracket on $\wh{R}$ coming from $\wh{A}$ equals the Poisson bracket on $\wh{R}$ coming from the fact that it is a completion of $R$.  To see this, we note the following well-known properties of algebras:
\begin{sublemma}
\mbox{}
\begin{enumerate}
\item Suppose that $A\rightarrow B$ is a filtered homomorphism of almost-commutative filtered algebras.  Then the induced map $\on{gr}(A)\rightarrow\on{gr}(B)$ is a Poisson homomorphism.
\item Suppose that $\phi: R\rightarrow S$ is a continuous homomorphism of topological algebras and $\phi(R)$ is dense in $S$.  Suppose $R$ is equipped with a Poisson structure.  Then there is at most one continuous Poisson structure on $S$ making $\phi$ a Poisson homomorphism.
\end{enumerate}
\end{sublemma}
Applying part (1) of the sub-lemma to $A\rightarrow\widehat{A}$ shows that the Poisson structure on $\widehat{R}$ induced from $\widehat{A}$ makes $R\rightarrow \widehat{R}$ a Poisson homomorphism.  Part (2) then implies that this Poisson structure must agree with the Poisson structure on $\widehat{R}$ induced from $R$.  Lemma \ref{lem:flat}(4) follows.
\end{proof}

\subsection{Quantizations of the formal neighborhood of $C$} The total space of the normal bundle $\NN_{\mathfrak{X}/C}$ has a canonical symplectic structure of weight $l$. Choosing homogeneous bases $\mbf{z}$ and $\mbf{w}$ as in Theorem \ref{thm:canonicalform}, the symplectic form on $\NN_{\mathfrak{X}/C}$ is given by $\omega_S + \sum_{i = 1}^m d z_i \wedge d w_i$. We denote by $\mc{D}$ the Moyal-Weyl quantization of $T^* V$ and by $\wh{\mc{D}}$ the Moyal-Weyl quantization of the ring of functions $\wh{F}$ on the formal neighborhood $\widehat{T^*V}$ of $V$ in $T^* V$.

We write $\mf{C}$ for the formal neighborhood of $C$ in $\NN_{\resol/C}$; recall that,
as in Section \ref{sec:localalg},  this is equivariantly isomorphic (though non-canonically) to the formal neighborhood of $C$ in $\resol$.  Let $\displaystyle \Omega_{\mf{C}}^{1,\mathrm{cts}} = \lim_{\longleftarrow} \Omega_{\mf{C}_n}^1$ denote the sheaf of continuous one-forms on $\mf{C}$, where $\mf{C}_n$ is the $n$th infinitesimal neighborhood of $C$. Similarly, let $\Theta_{\mf{C}}$ and $\Theta_{\mf{C}}^{\mathrm{cts}}$ denote the sheaf of vector fields and continuous vector fields on $\mf{C}$. We denote by $\Pi$ the bivector that defines the Poisson bracket on $\mf{C}$. Then, $d = [ \Pi, - ]$ defines a differential on $\wedge^{\idot} \Theta_{\mf{C}}$, where $[ - , - ]$ is the Schouten bracket on polyvector fields. The cohomology of $\wedge^{\idot} \Theta_{\mf{C}}$ is the Poisson cohomology $H_{\Pi}^{\idot}(\mf{C})$ of $\mf{C}$. The algebraic de Rham complex of $\mf{C}$ is denoted $\Omega_{\mf{C}}^{\idot,\mathrm{cts}}$. 

\begin{lem}\label{lem:ctsforms}
Let $i : \mf{C} \rightarrow \NN_{\resol/C}$ be the canonical morphism. 
\begin{enumerate}
\item Every derivation of $\wh{R}$ is continuous, \ie $\Theta_{\mf{C}} = \Theta_{\mf{C}}^{\mathrm{cts}}$.
\item $\Omega_{\mf{C}}^{1,\mathrm{cts}} \simeq \Theta_{\mf{C}}^{*} \simeq i^* \Omega_{\NN_{\resol/C}}$. 
\item The Poisson structure defines an isomorphism of complexes $\wedge^{\idot} \Theta_{\mf{C}} \simeq \Omega_{\mf{C}}^{\idot,\mathrm{cts}}$.
\end{enumerate}  
\end{lem}

\begin{proof}
Let $\delta \in \mathrm{Der} (\wh{R})$. From the definition of a derivation, $\wh{I}^{n+1} \subset \delta^{-1}(\wh{I}^n)$.   Since the translates of the powers of $\wh{I}$ are a base of the topology, it follows that that the preimage of $\wh{I}^n$ is open in $\wh{R}$.  Hence $\delta$ is continuous and $\Theta_{\mf{C}} = \Theta_{\mf{C}}^{\mathrm{cts}}$. By \cite[Proposition 20.7.15]{EGAIII}, the module $\Omega_{\mf{C}}^{1,\mathrm{cts}}$ is coherent. This implies that the dual of $\Omega_{\mf{C}}^{1,\mathrm{cts}}$ is the same as the continuous dual of $\Omega_{\mf{C}}^{1,\mathrm{cts}}$. Hence, the isomorphism $\Omega_{\mf{C}}^{1,\mathrm{cts}} \simeq \Theta_{\mf{C}}^{*}$ follows from equation (20.7.14.4) of \textit{loc. cit.}

Let $\mc{I}$ be the ideal defining the zero section in $\NN_{\resol/C}$ and $i_n : \mf{C}_n \rightarrow \NN_{\resol/C}$ the canonical morphism. For each $n$, there is a short exact sequence 
\begin{equation}\label{eq:ses1iso}
\mc{I}^{n} / \mc{I}^{2n} \stackrel{d_n}{\longrightarrow} i_n^* \Omega^1_{\NN_{\resol/C}} \rightarrow \Omega_{\mf{C}_n}^1 \rightarrow 0,
\end{equation}
where $d_n ( \overline{f}) = 1 \o d f$. Let $\ms{N}_n$ denote the image of $d_n$ and notice that $\mc{I} \cdot \ms{N}_n = 0$ for all $n$. This implies that $\mc{I} \cdot \mc{N} = 0$, where $\mc{N} = \varprojlim_n \mc{N}_n$. But $\ms{N}$ is a submodule of the free $\mc{O}_{\mf{C}}$-module $i^* \Omega^1_{\NN_{\resol/C}}$, implying that $\ms{N} = 0$. Similarly, the map $\ms{N}_{n+1} \rightarrow \ms{N}_n$ is zero for all $n$ because $\ms{N}_{n+1}$ is a submodule of $\mathrm{ann}_{\mc{I}} (i^*_{n+1} \Omega^1_{\NN_{\resol/C}} )$, which is mapped to zero under the map $i^*_{n+1} \Omega^1_{\NN_{\resol/C}} \rightarrow i^*_{n} \Omega^1_{\NN_{\resol/C}}$. Thus, $\{ \ms{N}_n \}_{n \in \mathbb{N}}$ satisfies the Mittag-Leffler condition and hence $\varprojlim_n^{(1)} \mc{N}_n=0$.
Therefore, (\ref{eq:ses1iso}) induces an isomorphism $i^* \Omega^1_{\NN_{\resol/C}} \rightarrow \Omega^{1,\mathrm{cts}}_{\mf{C}}$. Similarly, we have $\Theta_{\mf{C}} = i^* \Theta_{\NN_{\resol/C}}$. Thus the non-degenerate Possion structure on $\NN_{\resol/C}$ defines an isomorphism
$$
\Omega_{\mf{C}}^{1,\mathrm{cts}} \simeq i^* \Omega^1_{\NN_{\resol/C}} \stackrel{\sim}{\longrightarrow} i^* \Theta_{\NN_{\resol/C}} \simeq \Theta_{\mf{C}}. 
$$
The differential on the complex $\Omega_{\mf{C}}^{\idot,\mathrm{cts}}$ is defined as in section 7 of \cite[Chapter I]{HartshorneAlgDeRham}. Thus, the fact that we have an isomorphism of complexes $\wedge^{\idot} \Theta_{\mf{C}} \simeq \Omega_{\mf{C}}^{\idot,\mathrm{cts}}$ follows from the corresponding isomorphism $\wedge^{\idot} \Theta_{\NN_{\resol/C}} \simeq \Omega_{\NN_{\resol/C}}^{\idot}$ for $\NN_{\resol/C}$ due to Lichnerowicz (cf. \cite[Theorem 2.1.4]{DufourZung}).  
\end{proof}

\begin{remark}
Unlike for vector fields, we have $\Omega_{\mf{C}}^{1,\mathrm{cts}} \neq \Omega_{\mf{C}}^{1}$: indeed, the latter is not coherent over $\mf{C}$. 
\end{remark}

As in Section \ref{sec:localalg} we let $C\rightarrow S$ denote the coisotropic reduction of $C$ (though in our present, affine, setting the symplectic quotient is denoted $S'$ in that section); as in {\em loc. cit.}, this projection has a section $\nu: S\hookrightarrow C$.  
Let $z : C \hookrightarrow \mf{C}$ be the embedding of the zero section
and write $j = z\circ\nu : S \hookrightarrow \mf{C}$ for the composite embedding. This is a closed immersion of formal Poisson schemes. Hence, restriction defines a morphism $j^{-1}\big(\wedge^{\idot} \Theta_{\mf{C}}\big)\rightarrow \wedge^{\idot} \Theta_S$. Similarly, functoriality of the de Rham complex implies that we have a morphism $j^{-1} \Omega_{\mf{C}}^{\idot,\mathrm{cts}} \rightarrow \Omega_S^{\idot}$. These form a commutative diagram
\begin{equation}\label{eq:complexderham}
\xymatrix{
j^{-1} \Omega_{\mf{C}}^{\idot,\mathrm{cts}} \ar[d] \ar[r]^{\sim} & j^{-1} \wedge^{\idot} \Theta_{\mf{C}} \ar[d] \\
\Omega_S^{\idot} \ar[r]^{\sim} & \wedge^{\idot} \Theta_S.
}
\end{equation}

\begin{lem}\label{lem:quasi-iso}
The morphism of complexes $j^{-1} \Omega_{\mf{C}}^{\idot,\mathrm{cts}} \rightarrow \Omega_S^{\idot}$ is a quasi-isomorphism. Hence, the de Rham cohomology groups $H^2_{DR}(\mf{C})$ and $H^2_{DR}(S)$ are isomorphic. 
\end{lem}

\begin{proof}
We factor $j^{-1} \Omega_{\mf{C}}^{\idot,\mathrm{cts}} \rightarrow \Omega_S^{\idot}$ as 
$$
\nu^{-1} \left(z^{-1} \Omega_{\mf{C}}^{\idot,\mathrm{cts}} \right) \longrightarrow \nu^{-1} \Omega_{C}^{\idot} \longrightarrow \Omega_{S}^{\idot}. 
$$
Thus, it suffices to show that each of $z^{-1} \Omega_{\mf{C}}^{\idot,\mathrm{cts}} \rightarrow \Omega_{C}^{\idot}$ and $\nu^{-1}  \Omega_{C}^{\idot} \rightarrow \Omega_{S}^{\idot}$ is a quasi-isomorphism. That the first is a quasi-isomorphism is Proposition 1.1 in Chapter II of \cite{HartshorneAlgDeRham}. Since $C = S \times V$ and $V \simeq \mathbb{A}^m$ is contractible, the second morphism is a quasi-isomorphism. 
\end{proof}

Now we would like to use the above results to relate quantizations
of $\mathfrak{C}$ to quantizations of $S$. To accomplish this, we
shall use some results of Bezrukavnikov and Kaledin \cite{BK}
on period maps for quantizations. Their results are stated only for algebraic varieties,
but they apply without essential change to smooth formal schemes as
well. Since the results of this section are by now quite standard,
and have also been summarized very well in \cite{LosevIso}, where the details of compatibility with $\Cs$-actions are also examined, we shall content ourselves with a very terse recollection.

To describe the results of \cite{BK}, we first recall the notion of a Harish-Chandra
torsor on $\resol$ (see also \cite[page 1227]{LosevIso} for more details): suppose that $G$ is a (pro-algebraic) group
with Lie algebra $\mathfrak{g}$ and that $\mathfrak{h}$ is a Lie
algebra such that $\mathfrak{g}\subset\mathfrak{h}$. Suppose further
that $\mathfrak{h}$ is equipped with an action of $G$ whose differential
agrees with the adjoint action of $\mathfrak{g}$ on $\mathfrak{h}$.
Then the pair $(G,\mathfrak{h})$ is known as a Harish-Chandra pair. A Harish-Chandra torsor for $(G,\mathfrak{h})$ is a pair $(M,\theta)$,
where $M$ is a $G$-torsor on $\mathfrak{X}$ and $\theta$ is an
$\mathfrak{h}$-valued flat connection on $M$ (the notions of torsor
and flat connection are defined for (formal) schemes exactly as they
are in usual differential geometry). 

A symplectic variety comes equipped with a canonical Harish-Chandra
torsor, defined as follows: let $\mathfrak{A}$ denote the algebra
of functions on a symplectic formal disc. Then the group of symplectomorphisms
$\mbox{Aut}(\mathfrak{A})$ of $\mathfrak{A}$ is naturally a pro-algebraic
group. Furthermore, the Lie algebra of Hamiltonian derivations of $\mathfrak{A}$, denoted $\mathfrak{H}$, is a pro-Lie algebra and $(\mbox{Aut}(\mathfrak{A}),\mathfrak{H})$ is a Harish-Chandra pair. Then the Harish-Chandra torsor $\mathcal{M}_{symp}(\mathfrak{X})$ is defined to be the pro-scheme parametrizing all maps $\mbox{Spec}(\mathfrak{A})\to\mathfrak{X}$ which preserve the symplectic form. 

Now, let $D$ denote the (unique) quantization of a $2n$-dimensional
formal disc over $\mathbb{C}$. Then the group of automorphisms $\mbox{Aut}(D)$
comes with a natural pro-algebraic group structure; similarly, the
Lie algebra of derivations $\mbox{Der}(D)$ is naturally a pro-Lie
algebra making $\big(\mbox{Aut}(D),\mbox{Der}(D)\big)$ a Harish-Chandra pair;
we note that the map {}``reduction mod $\h$'' gives a morphism of
Harish-Chandra pairs
 $\big(\mbox{Aut}(D),\mbox{Der}(D)\big)\to\big(\mbox{Aut}(\mathfrak{A}),\mathfrak{H}\big)$.

Define 
$$
H_{\mathcal{M}_{\mathrm{symp}}}^{1}\big(\mbox{Aut}(D),\mbox{Der}(D)\big)
$$
to be the set of all isomorphism classes of $\big(\mbox{Aut}(D),\mbox{Der}(D)\big)$-torsors
on $\resol$ which are liftings of $\mathcal{M}_{\mathrm{symp}}$; \ie
those torsors equipped with a reduction of structure group to $\big(\mbox{Aut}(\mathfrak{A}),\mathfrak{H}\big)$
such that the resulting $\big(\mbox{Aut}(\mathfrak{A}),\mathfrak{H}\big)$-torsor
is isomorphic to $\mathcal{M}_{\mathrm{symp}}$. 

It is shown in \cite[\S 3]{BK} that there is a natural bijection 
$$
\mathsf{Loc}\colon H_{\mathcal {M}_{\mathrm{symp}}}^{1}\big(\mbox{Aut}(D),\mbox{Der}(D)\big) \stackrel{\sim}{\rightarrow} Q(\mathfrak{X}),
$$
where the right-hand side denotes the set of all isomorphism classes of quantizations of $\mathfrak{X}$. This bijection respects the $\Cs$-action on both sides, and hence it can be checked that it induces a bijection between equivariant quantizations and  $H_{\mathcal{M}_{symp}}^{1}\big(\mbox{Aut}(D),\mbox{Der}(D)\big)^{\Cs}$.

Now the non-Abelian cohomology group admits a natural ``period map'' 
\[
\mathsf{Per} \colon H^{1}_{\mathcal M_{\mathrm{symp}}}\big(\mathrm{Aut}(D),\mathrm{Der}(D)\big) \to H^2_{DR}(\mathfrak{X})[\![\h]\!],
\]
which moreover restricts to give a map $\mathsf{Per}\colon H^{1}_{\mathcal M_{symp}}\big(\mathrm{Aut}(D),\mathrm{Der}(D)\big)^{\Cs} \to H^2_{DR}(\mathfrak X)$. In good situations, such as when  $H^i(\mathfrak X,\mathcal O_{\mathfrak X}) =0$ for $i=1,2$, this equivariant period map is an isomorphism. We thus have the following classification:

\begin{thm}
Let $\mathfrak{X}$ be a smooth symplectic affine algebraic variety
or formal scheme, with an elliptic action of $\Cs$ (assumed
to be pro-rational if $\mathfrak{X}$ is formal). We let $Q^{\Cs}(\mathfrak{X})$
denote the set of isomorphism classes of $\Cs$-equivariant quantizations of $\mathfrak{X}$.
Then there is a natural bijection 
\[
Q^{\Cs}(\mathfrak{X})\to H_{DR}^{2}(\mathfrak{X}).
\]
\end{thm}

To relate quantizations of $\mathfrak C$ to quantizations of $S$ we may thus relate the corresponding Harish-Chandra torsors: given a Harish-Chandra torsor on $S$
and a Harish-Chandra torsor on $\widehat{T^{*}V}$, the external product
defines a new Harish-Chandra torsor on $S\times\widehat{T^{*}V} \simeq \mathfrak{C}$.
If we suppose that these Harish-Chandra torsors are liftings of $\mathcal{M}_{\mathrm{symp}}(S)$
and $\mathcal{M}_{\mathrm{symp}}(\widehat{T^{*}V})$, then under the map $Loc$,
this external product of torsors corresponds to taking a quantization
$B$ of $S$ and associating to it the quantization $B\otimes\widehat{\cD}$
on $\mathfrak{C}$. On the other hand, the map $\mathsf{Per}$ is constructed
by associating, to a Harish-Chandra torsor which is a lifting of $\mathcal{M}_{\mathrm{symp}}$,
a certain canonical cohomology class associated to the lifting. From
this is follows that the map $Q(S)^{\Cs}\to Q(\mathfrak{C})^{\Cs}$
(given by $B\to B\otimes\widehat{\cD}$) corresponds, under $\mathsf{Per}\circ \mathsf{Loc}^{-1}$,
to the map $H_{DR}^{2}(S)\to H_{DR}^{2}(\mathfrak{C})$ associated
to the inclusion $S\to\mathfrak{C}$. Since this has been shown to
be an isomorphism above, we have shown:

\begin{prop}\label{prop:completeiso}
Let $\wh{A}$ be a $\Cs$-equivariant quantization of $\mf{C}$. Then there exist an equivariant quantization $B$ of $S$ and an equivariant isomorphism $\psi : \wh{A} \stackrel{\sim}{\longrightarrow} B \wh{\o} \wh{\mc{D}}$ of deformation-quantization algebras. 
\end{prop} 

A quantization $A$ of $\resol$ induces a quantization $\wh{A}$ of $\mf{C}$. In the following subsections we use the isomorphism $\psi$ given by the above Proposition to show that the quantum Hamiltonian reduction of $A$ is isomorphic to $B$.

\subsection{The completion of the ideal $J$} Recall the ideal $J\subset A$ from Section \ref{sec:QCRlocal}.  Let $\wh{J} = \wh{A} \o_A J$. Lemma \ref{lem:flat} implies that the natural map $\wh{J} \rightarrow \wh{A}$ is an embedding. Let $\mf{u}$ be the $m$-dimensional subspace of $A_{< 0}$ spanned by a collection of homogeneous lifts of the $w_i \in I$
and let $J' = A \mf{u}$. Since the isomorphism $\psi$ of Proposition \ref{prop:completeiso} induces the identity on $\wh{R}$, the image $\psi(\mf{u})$ consists of lifts of the $w_i$
to $\wh{\mc{D}}$. We denote this space by $\mf{u}$ as well. 

\begin{lem}\label{lem:JeqJprime}
We have $J = J'$ and hence $\psi(\wh{J}) = B \wh{\o} (\wh{\mc{D}} \mf{u})$. Thus $(A/J)/\h \stackrel{\sim}{\rightarrow} R/I$. 
\end{lem}

\begin{proof}
Clearly we have $J'\subseteq J$. In order to get the opposite inclusion, let $a\in J$ be any element; since $J$ is generated by homogeneous elements, it suffices to assume $a$ is homogeneous.  Recall that $\rho: A\rightarrow R$ is the projection.
Since 
$\rho(J) = I = R\mathfrak{u}$, 
we may write
\[
a'_{1}:=a-\sum a_{i}y_{i}\in \h A
\]
 for some homogenous elements $\{a_{i}\}$, and we may select these
elements so that $\mbox{deg}(\sum a_{i}y_{i})=\mbox{deg}(a)<0$. But
since $\deg (\h)=l$, we have that $a_{1}'=\h a_{1}$ for some $a_{1}$
of strictly smaller degree than $a$; evidently $a_{1}$ is again in $J \smallsetminus J'$. We may
repeat this process to find $a_{2}'=\h a_{2}$, etc. But now if we look
at these equations in the finitely generated left $A$-module $N=J/J'$,
we see that we have a sequence $a= \h a_{1}=\h^{2} a_{2}=\dots$ . But
this implies that ${\displaystyle \bigcap_{n\geq0}\h^{n} N}$ is
nonzero; which contradicts the fact that $N$ is $h$-complete and hence separated.  

To see the last statement, note that the first statement shows that the image of the natural map $J/ \h \to A/ \h=R$ is precisely $I$. Then the statement follows from applying the functor $M\to M/ \h$ to the short exact sequence $J \to A \to A/J$. 
\end{proof}

\subsection{Identication of the formal quantum coisotropic reduction} 
We require the following result, which is Lemmata 1.2.6 and 1.2.7 of \cite{KSDQ}. 

\begin{lem}\label{lem:KSlemma}
Let $L$ be a finitely generated, free $A$-module and $N$ a submodule. Then, $\bigcap_{i = 1}^{\infty} (N + \h^i L) = N$. 
\end{lem}

\begin{prop}\label{prop:AJflat}
The natural map $A / J \rightarrow \widehat{A/J}$ is an embedding. Hence, $A/J$ is $\h$-flat. 
\end{prop}

\begin{proof}
Note that Lemma \ref{lem:flat} implies that the sequence $0 \rightarrow \widehat{J} \rightarrow \widehat{A} \rightarrow \widehat{A/J} \rightarrow 0$ is exact, hence it suffices to show that $\widehat{J} \cap A = J$. Since the image of $J$ in $A / K^n$ equals $(J + K^n) / K^n$, we have 
$$
\widehat{J} \cap A = \bigcap_{n = 1}^{\infty} (J + K^n). 
$$

\begin{claim}\label{claim:subinter}
We have $\bigcap_{n = 1}^{\infty} (J + K^n) \subset J$.
\end{claim}

\begin{proof}[Proof of Claim]
Notice that $K = JA + \h A = J + \h A$ because $K / \h A = (J + \h A) / \h A = I$. Consider the expansion of $(J + \h A)^n$. Since $\h$ is central, a term of this expansion containing $i$ copies of $\h A$ equals $\h^i J^{n-i} A$. Multiplying $JA + \h A = J + \h A$ on the left by $\h^i J^{n-i-1}$ implies that 
$$
\h^i J^{n-i} A + \h^{i+1} J^{n-i-1} A = \h^i J^{n-i} + \h^{i+1} J^{n-i-1} A.
$$
Thus, 
$$
(J + \h A)^n = \h^n A + \sum_{i = 0}^{n-1} \h^i J^{n-i} \subset \h^n A + J. 
$$
Since $J  + \bigcap_{n= 1}^{\infty} \h^n A \subset J + \h^n A$ for all $n$, we have $J  + \bigcap_{n = 1}^{\infty} \h^n A \subset \bigcap_{n = 1}^{\infty} (J + \h^n A)$. On the other hand, Lemma \ref{lem:KSlemma} says that $\bigcap_{n = 1}^{\infty} (J + \h^n A) = J$. This completes the proof of Claim \ref{claim:subinter}. 
\end{proof}
Returning to the proof of Proposition \ref{prop:AJflat}, the second assertion will now follow from the fact that $\widehat{A} / \widehat{J}$ is $\h$-flat. By Lemma \ref{lem:JeqJprime}, 
\begin{equation}\label{eq:flatiso}
\widehat{A} / \widehat{J}  = B \wh{\o} \left( \wh{\mc{D}} / \wh{\mc{D}} \mf{u} \right).
\end{equation}
The Poincar\'e-Birkhoff-Witt property of $\wh{\mc{D}}$ implies that the right hand side of (\ref{eq:flatiso}) is $\h$-flat. 
\end{proof}

Recall again that we have constructed a coisotropic reduction $\pi : C \rightarrow S$. Let $B$ be the quantization of the Poisson algebra $T = \C[S]$ given by Propositin \ref{prop:completeiso}. We have an identification $T = (R / I)^{ \{ I, - \} }$, such that the embedding $(R / I)^{ \{ I, - \} } \hookrightarrow R/I$ is just the comorphism $\pi^*$.

\begin{prop}\label{lem:topcoisored}
\mbox{}
\begin{enumerate}
\item
The $\C[\![\h]\!]$-algebra $\End_{\wh{\mc{D}}} ( \wh{\mc{D}} / \wh{\mc{D}} \mf{u})$ is isomorphic to $\C[\![\h]\!]$ and $\Ext^i_{\wh{\mc{D}}}(\wh{\mc{D}} / \wh{\mc{D}} \mf{u},\wh{\mc{D}} / \wh{\mc{D}} \mf{u})$ is a torsion $\C[\![\h]\!]$-module for $i > 0$. 
\item We have an isomorphism of $\C[\![\h]\!]$-algebras 
$$
\End_{\widehat{A}}(\widehat{A} / \widehat{J})^{\mathrm{opp}} \stackrel{\sim}{\longrightarrow} B
$$
and $\Ext^i_{\widehat{A}}(\widehat{A} / \widehat{J},\widehat{A} / \widehat{J})$ is $\h$-torsion for all $i > 0$. 
\end{enumerate}
\end{prop}

\begin{proof}
(1) Recall that, as above, $\wh{\D}$ is the Moyal-Weyl quantization of the algebra $\wh{F}$ of functions on $\wh{T^*V}$, the formal neighborhood of the zero section;  $\mf{u}$ is a space spanned by homogeneous lifts of generators of $I/I^2$ to $\wh{\D}$.  Under the identification of $\wh{\D}$ with $\wh{F}[\![\hbar]\!]$, we may assume that the elements $w_1, \dots, w_m$ of $\mf{u}$ are coordinate functions on $V^*$.  Write
$L:= \wh{\D}/\wh{\D}\mf{u}.$

It is standard for a cyclic left module $S/P$ over a ring $S$ that
$\End_S(S/P) \cong \{q\in S/P \; | Pq \subseteq P\}$; it is also standard that $\{ f \in \mc{D} / \mc{D} \mf{u} \ | \ \mf{u} \cdot f = 0 \} \cong \C[[\h]]$ (and in any case this can be computed, by hand, inductively using the Moyal-Weyl product).

Letting $z_i$ denote the dual coordinates on $V$ as in Theorem \ref{thm:canonicalform}, we get an identification
\begin{equation}\label{eq:D}
\wh{\D} \cong \C[z_1,\dots, z_m][\![w_1,\dots, w_m,\hbar]\!],
\end{equation}
 with Moyal-Weyl product $\ast$ satisfying 
\bd
w_i\ast w_j = w_iw_j, \hspace{1em} z_i\ast z_j = z_iz_j, \hspace{1em}
w_i\ast f(z,w) = w_i\cdot f(z,w) + \frac{\hbar}{2} \frac{\partial f}{\partial z_i}, \hspace{1em}
f(z,w)\ast w_i = w_i\cdot f(z,w) -\frac{\hbar}{2} \frac{\partial f}{\partial z_i}.
\ed
It follows from Lemma \ref{lem:JeqJprime} that the natural composite $\C[z_1,\dots, z_m][\![\hbar]\!]\hookrightarrow \wh{D}\rightarrow L$ is an isomorphism of vector spaces: via the vector space isomorphism of \eqref{eq:D} and the formulas above, we can write any element of $\wh{D}$ as $\sum_{I,j} f_{I,j}(z)\ast w^I\hbar^j$, and then those terms with nonconstant $w^I$ vanish in $L$.   
Under this identification, for $f(z)\in L$, we have 
\begin{equation}\label{eq:ast-action}
w_i\ast f(z) = w_i\ast f(z) - f(z)\ast w_i = \hbar\frac{\partial f}{\partial z_i}.
 \end{equation}

Let $K(\mf{u}) \cong \big(\wedge^\bullet(\mf{u})\otimes \Sym(\mf{u}), d\big)$ denote the Koszul complex of $\C$ as a $\on{Sym}(\mf{u})$-module, and let $\displaystyle K(\wh{\D},\mf{u}) = \wh{\D}\underset{\Sym(\mf{u})}{\otimes}K(\mf{u})$; so $K(\wh{\D},\mf{u})$ is a finite free resolution of $L$ over $\wh{\D}$.  By adjunction,
\bd
\Ext^i_{\wh{\mc{D}}}(L,L) 
\cong H^i\big(\Hom_{\on{Sym}(\mf{u})}(K(\mf{u}), L)\big) \cong H^i\big(\wedge^\bullet(\mf{u}^*)\otimes \wh{\D}/\wh{\D}\mf{u}\big).
\ed
It is then evident from \eqref{eq:ast-action} that the identification $\C[z_1,\dots, z_m][\![\hbar]\!]\rightarrow L$ intertwines the Koszul differential and $\hbar$ times the de Rham differential, thus yielding, when $\hbar$ is inverted,
\bd
H^i\big(\wedge^\bullet(\mf{u}^*)\otimes \wh{\D}/\wh{\D}\mf{u}\big)[\hbar\inv] \cong H^i_{DR}(\on{Spec}\C[z_1,\dots, z_m])(\!(\hbar)\!),
\ed
which proves the $i>0$ part of (1).

(2) Again, we have
\bd
\Ext^i_{B \wh{\o} \wh{\D}}(B \wh{\o} L, B \wh{\o} L) \cong H^i\big(\Hom_{B\wh{\o} \wh{\D}}(B\wh{\o} K(\wh{D},\mf{u}), B\wh{\o} L)\big) 
\cong H^i\big(\wedge^\bullet(\mf{u}^*)\otimes B\wh{\o} L\big)
\ed
where the last isomorphism follows by adjunction as before.   Let $d_B$ denote the Koszul differential on this completed tensor product and $d$ the Koszul differential on $\wedge^\bullet(\mf{u}^*)\otimes L$.  
The $\mf{u}$-action commutes with all elements of $B$ and $B\cong \C[S][\![\hbar]\!]$ as a free $\C[\![\hbar]\!]$-module.  Thus, letting $\{s_i\}$ denote a vector space basis of $\C[S]$, for any element $\sum s_i\otimes l_i$ of $B\wh{\o} L$ we get $d_B(\sum s_i\otimes l_i) = \sum s_i d(l_i)$, and it follows that 
\bd
\on{ker}(d_B) = B\wh{\o} \on{ker}(d), \hspace{2em} \on{Im}(d_B) = B\wh{\o} \on{Im}(d).
\ed
Thus,
\bd
\Ext^i_{B \wh{\o} \wh{\D}}(B \wh{\o} L, B \wh{\o} L) \cong B\wh{\o}_{\C[\![\hbar]\!]} H^i\big(\wedge^\bullet(\mf{u}^*)\otimes L\big)
\cong
B\wh{\o}_{\C[\![\hbar]\!]}\Ext^i_{\wh{\mc{D}}}(L,L),
\ed
reducing the assertions of part (2) to part (1).
\end{proof}

\subsection{Identication of the quantum coisotropic reduction} For any $s > 0$, let $A_s = A / \h^s A$ and $\widehat{A}_J = \limn A / J^n$. Even though $J$ is only a left ideal of $A$, we can form the Rees algebra $\Rees_J(A) = \bigoplus_{n \ge 0} J^n$, with the obvious multiplication. We shall abuse notation and denote by $J$ the left ideal generated by the image of $J$ in $A_s$.

\begin{lem}\label{lem:completeness-properties}
\mbox{}
\begin{enumerate}
\item The inclusion $J^n \subset K^n$ induces an isomorphism of complete topological algebras $\widehat{A}_J \stackrel{\sim}{\rightarrow} \widehat{A}$. 
\item The Rees algebra $\Rees_J(A_s) $ is (both left and right) Noetherian. 
\item Let $M$ be a finitely generated $A$-module. Then, $\widehat{A}_J  \o_A M \simeq \limn M / J^n M$. 
\end{enumerate}
\end{lem}

\begin{proof}
Part 1). Since $[A,A] \subset \h A$, we have $K^n \subseteq J^n + \h J^n + \cdots + \h^s J^{n-s}$ for all $s, n > 0$. Therefore the filtrations $\{ K^n \}_n$ and $\{ J^n \}_n$ of $A_s$ are comparable and the canonical morphism $\displaystyle\lim_{\longleftarrow} A_s / J^n \rightarrow \lim_{\longleftarrow} A_s / K^n$ is an isomorphism. Thus, commutativity of limits implies that 
$$
\limn A / J^n = \lims \left( \limn A_s / J^n \right) \longrightarrow \lims \left( \limn A_s / K^n \right) = \limn A / K^n
$$
is an isomorphism. 

Part 2). Since $A_s$ is finitely generated and nilpotent, we may choose a finite dimensional vector subspace $\mathfrak{n}$ of $A_s$ which is bracket closed and generates $A_s$ as an algebra; enlarging $\mathfrak{n}$ if necessary, we may suppose $J \cap \mathfrak{n} =\mathfrak{n}_1$ is a Lie subalgebra which generates $J$ as an ideal. Then $\mathfrak{n}$ is a nilpotent Lie algebra, and we have a surjection $U(\mathfrak{n}) \to A_s$. Further, we have the subalgebra $U(\mathfrak{n}_1)$; the image of its augmentation ideal in $A_s$ is $J$. Thus the claim is reduced to showing the following: let $\mathfrak{n}_1 \subset \mathfrak{n}$ be nilpotent Lie algebras, and let $\mathfrak{a}$ be the left ideal of $U(\mathfrak{n})$ generated by $\mathfrak{n}_1$; then $\Rees_{\mathfrak{a}}(U(\mathfrak{n}))$ is Noetherian. But this is a standard argument, see for instance \cite{Skryabin}. 

Part 3). This is a non-commutative analogue of \cite[Theorem 7.2]{Eisenbud}. If $M'$ is a submodule of $M$, then the argument given in the proof of Lemma 7.15 of \textit{loc. cit.} shows that the claim reduces to showing that the morphism
$$
{\displaystyle\lim_{\substack{\longleftarrow\\s,n}}\,}  M' / (J^n M' + \h^s M') \longrightarrow {\displaystyle\lim_{\substack{\longleftarrow\\s,n}}\,}  M' / ((J^n M) \cap M' + \h^s M')  
$$
is an isomorphism. This will be an isomorphism if, for each $s,n \ge 1$, there exist $N(s,n),S(s,n) \gg 0$ such that 
$$
(J^N M) \cap M' + \h^S M' \subset J^n  M' + \h^s M'. 
$$
By Lemma \ref{lem:hcomp1} (2), the Rees algebra $\Rees_{\h A} (A)$ is Noetherian. Therefore, there exists some $s_0$ such that $\h^{i+s_0} M \cap M' \subset \h^i M'$ for all $i \ge 1$. The $A_{s+s_0}$-module $M' / (\h^{s+s_0} M \cap M')$ is a submodule of $M / \h^{s + s_0} M$. Since we have shown in Part 2) that the Rees algebra $\Rees_{J}(A_{s+s_0})$ is Noetherian, the usual Artin-Rees argument shows that there exists some $N \gg 0$ such that 
\begin{equation}\label{eq:inclusion1}
(J^N \cdot M) \cap M' + (\h^{s+ s_0} M \cap M') \subset  J^n \cdot M' + (\h^{s+ s_0} M \cap M'). 
\end{equation}
Since $\h^{s+ s_0} M \cap M' \subset \h^s M'$, the inclusion (\ref{eq:inclusion1}) implies that 
$$
(J^N \cdot M) \cap M' + \h^s M' \subset J^n \cdot M' + \h^s M',
$$
as required. 
\end{proof}

\begin{remark}
One can check that the ring $\Rees_J(A)$ is \textit{not} in general Noetherian. 
\end{remark}

\begin{thm}\label{thm:bigone}
\mbox{}
\begin{enumerate}
\item The $\C[\![\h]\!]$-algebras $\End_A(A /J)^{\opp}$ and $B$ are isomorphic. Hence $\End_A(A /J)^{\opp}$ is a deformation-quantization of $S$. 
\item The ext groups $\Ext^i_A(A/J,A/J)$ are $\h$-torsion for all $i > 0$. 
\item $A/J$ is a faithfully flat $B$-module.
\end{enumerate}
\end{thm}

\begin{proof}
Parts (1) and (2): Since $\widehat{A} / \widehat{J} = \widehat{A} \o_A (A/J)$ and $\wh{A}$ is flat over $A$, adjunction says that we have 
$$
\Ext^i_{\widehat{A}} ( \widehat{A} / \widehat{J}, \widehat{A} / \widehat{J}) \simeq \Ext^i_A( A/ J, \widehat{A} / \widehat{J}), \quad \forall \ i \ge 0. 
$$
Since $A/J$ is a finitely generated $A$-module, Lemmata \ref{lem:flat} and \ref{lem:completeness-properties} imply that $\widehat{A} / \widehat{J} = \widehat{A} \o_A (A/J)$ is isomorphic to $\widehat{A}_J \o_A (A/J)$. By Lemma \ref{lem:completeness-properties}, we have 
$$
\widehat{A}_J \o_A (A/J) = \limn \frac{(A/J)}{J^n \cdot (A/J)} = A/J. 
$$
Therefore, claims (1) and (2) of the theorem follows from Lemma \ref{lem:topcoisored}. 

(3) By Proposition \ref{prop:AJflat}, $A/J$ is $\h$-flat, or equivalently $\h$-torsion free. Since it is finitely generated over $A$, it is also $\h$-complete. Therefore, \cite[Corollary 1.5.7]{KSDQ} says that it is cohomologically complete. By \ref{lem:JeqJprime}, $(A / J) / \h (A / J) \simeq  R/I$ and hence is a free $T$-module. Thus, \cite[Theorem 1.6.6]{KSDQ} implies that $A/J$ is a faithfully flat $B$-module.  
\end{proof}

Let $\W = \A[\h^{-1}]$, a $\cW$-algebra. By base change, Theorem \ref{thm:bigone} implies: 

\begin{corollary}\label{cor:BDQ}
The algebra 
$$
\End_{\W}(\W / J[\h^{-1}], \W / J[\h^{-1}] )^{\opp} \stackrel{\sim}{\longrightarrow} B[\h^{-1}] =: \W_S
$$
is a $\cW$-algebra, and $\Ext^i_{\W}(\W / J[\h^{-1}], \W / J[\h^{-1}] ) = 0$, for all $i > 0$.
\end{corollary}

\subsection{Equivariant Modules}\label{ss:affine-eq}
We maintain the notation and assumptions of the prior subsections of Section \ref{sec:QCR}.  In particular, $\resol = \on{Spec}(R)$ is a smooth affine symplectic variety with $\Gm$-action and with coisotropic subvariety $C= \on{Spec}(R/I)$, where $I$ is generated by all homogeneous elements of negative degree.  Moreover, $A$ is a deformation-quantization of $R$, and $W = A[\hbar\inv]$.  We have a symplectic quotient $C\rightarrow S$, also assumed affine, and $B$ is a deformation-quantization of $\C[S]$.  
\begin{defn}
The full abelian subcategory of $\Lmod{(A,\Cs)}$ consisting of all modules supported on $C$ is denoted $\Lmod{(A,\Cs)}_C$. 
The full abelian subcategory of $\Good{(\W,\Cs)}$ consisting of good $\cW$-modules supported on $C$ is denoted by $\Good{(\W,\Cs)}_C$. 
\end{defn}
We define a filtration $\mathcal{H}_i M$ on a finitely generated, equivariant $A$-module by letting $\mc{H}_i (M / \h^n M)$ be the direct sum $\bigoplus_{j \ge i} (M / \h^n M)_j$ and 
$\mc{H}_i M  = \limn   \mc{H}_i (M / \h^n M). $
Then $\h \mc{H}_i M \subset \mc{H}_{i+1} M$. The filtration $\mc{H}_i M$ need not be exhaustive. 

\begin{lem}\label{lem:Hfilt}
Let $M \in \Lmod{(A,\Cs)}$. Then $\Supp \ M \subset C$ if and only if $\mc{H}_N M = M$ for $N \ll 0$. 
\end{lem}

\begin{proof}
If $\mc{H}_N M = M$, then $\mc{H}_N (M / \h M) = M / \h M$. This means that $M / \h M$ is a graded $R$-module, with $(M / \h M)_i = 0$ for $i \ll 0$. This implies that $\Supp \ (M / \h M)$ is contained in $C$. Therefore, by Lemma \ref{lem:suppsupp}, the support of $M$ is contained in $C$. Conversely, if $\Supp \ M \subset C$ then clearly the support of $M / \h M$ is contained in $C$ too. This implies that $\mc{H}_N (M / \h M) = M / \h M$ for some $N \ll 0$. By induction on $n$, the exact sequence
$$
M / \h M \stackrel{\cdot \h^n}{\longrightarrow} M / \h^{n+1} M \rightarrow M / \h^n M \rightarrow 0
$$
implies that $\mc{H}_N (M / \h^n M) = M / \h^n M$ and hence $\mc{H}_N (M) = M$. 
\end{proof}

\subsection{Quantum Coisotropic Reduction: Affine Case}\label{sec:QCRaffine}
We maintain the notation of Section \ref{ss:affine-eq}.
For a module $M[\h^{-1}] \in \Good{(\W,\Cs)}_C$, we denote by $M$ a choice of lattice in $\Lmod{(A,\Cs)}_C$. Recall that $\W = \A[\h^{-1}]$ and $\W_S = B[\h^{-1}]$.  

By Theorem \ref{thm:bigone}, we can define an adjoint pair $(\Hamp, \Ham)$ of functors
\bd
\xymatrix{
\Lmod{(A,\Cs)}_C 
  \ar@/^1pc/[r]^{\Ham}  & 
\Lmod{(B,\Cs)} \ar@/^1pc/[l]_{\Hamp} 
}
\ed
by
$$
\Ham(M) = \Hom_A(A/J, M), \quad \textrm{and} \quad  \Hamp (N)= A /J \o_{B} N.
$$
The functors $\Ham, \Hamp$ clearly preserve the subcategories of $\hbar$-torsion modules, and in particular by Proposition \ref{prop:quot-cat} they thus induce a well-defined adjoint pair of functors
\begin{equation}\label{eq:Ham-adj-pair}
\xymatrix{
\Good{(\W,\Cs)}_C   \ar@/^1pc/[rr]^{\Ham}  & & \ar@/^1pc/[ll]^{\Hamp}  \Good{(\W_S,\Cs)} 
}
\end{equation}
for which
$$
\Ham(M[\h^{-1}] ) = \Hom_A(A/J, M)[\h^{-1}], \quad \textrm{and} \quad  \Hamp (N[\h^{-1}] )= (A /J \o_{B} N)[\h^{-1}].
$$
\begin{thm}\label{thm:mainlocalequiv}
The functors $\Ham$, $\Hamp$  of \eqref{eq:Ham-adj-pair} are exact, mutually quasi-inverse equivalences of abelian categories.
\end{thm}

\begin{proof}
Suppose $M\in\Lmod{(A,\Cs)}_C$.  Let $M^{\rat} = \bigoplus_{i} M_i$ be the subspace of $\Cs$-locally finite vectors. By Lemma \ref{lem:Csfinitedim}, this space is non-zero if $M$ is. Lemma \ref{lem:Hfilt} implies that there exists $N$ such that $M_N \neq 0$ but $M_i = 0$ for all $i < N$. Now $\mf{u} \cdot M_N = 0$ and hence:
\begin{center}
$\Ham$ is left-exact on $\Lmod{(A,\Gm)}_C$ and $\Good{(W,\Gm)}_C$, and conservative on $\Lmod{(A,\Gm)}_C$. 
\end{center}
\begin{remark}\label{tf-ness}
By Theorem \ref{thm:bigone}(3), $A/J$ is a faithfully flat $B$-module; thus $\Hamp (N)$ is $\h$-torsion-free if $N$ is $\h$-torsion-free. Similarly $\Ham(N)$ is $\hbar$-torsion free if $N$ is, by inspection.
\end{remark}
Theorem \ref{thm:bigone}(3) also implies:
\begin{center}
 $\Hamp$ is exact and conservative on $\Lmod{(B,\Cs)}$ and $\Good{(\W_S,\Cs)}$.
 \end{center}
  
Next, we show:
\begin{claim}
The adjunction 
$$
\id \rightarrow \Ham \circ \Hamp
$$
is an isomorphism of functors of $\Good{(\W_S,\Cs)}$.
\end{claim}
To prove the claim, we observe that the global dimensions of $B$ and $\W_S$ are finite, and therefore we prove, by induction on the projective dimension of $N \in \Lmod{(B,\Cs)}$, that:\begin{enumerate}
\item[(a)] $N[\hbar\inv]\rightarrow \Ham(\Hamp(N[\hbar\inv]))$ is an isomorphism.
\item[(b)] $\Ext^i_A\big(A/J, \Hamp(N)\big)[\hbar\inv] = 0$ for all $i>0$.
\end{enumerate}
When $N$ is a finitely generated projective $B$-module, the assertions are immediate from Corollary \ref{cor:BDQ}. So we may assume that assertions (a) and (b) hold for all modules $F$ with projective dimension less than $\mathrm{p.d.}_{B} \ N$. Fix a presentation $0 \rightarrow F \rightarrow B^k \rightarrow N \rightarrow 0$, so that the projective dimension of $F$ is less than the projective dimension of $N$. Since $\Hamp$ is exact we get an exact sequence 
\begin{equation}\label{eq:ham}
0 \rightarrow \Ham \circ \Hamp (F) \rightarrow  \Ham \circ \Hamp (B^k) \rightarrow  \Ham \circ \Hamp (N) \rightarrow \Ext^1_A(A/J, \Hamp(F))\rightarrow\dots
\end{equation}
Inverting $\hbar$ and using the inductive hypothesis that assertion (b) holds for $F$, we get a short exact sequence
\bd
0 \rightarrow \Ham \circ \Hamp (F)[\hbar\inv] \rightarrow  \Ham \circ \Hamp (B^k)[\hbar\inv] \rightarrow  \Ham \circ \Hamp (N)[\hbar\inv] \rightarrow 0.
\ed
Assertion (a) for $F$ and $B^k$ then implies assertion (a) for $N$.   Similarly, it follows, by continuing the exact sequence \eqref{eq:ham}, from assertion (b) for $B^k$ and $F$ that 
$\Ext^i_A(A/J, \Hamp(N))[\hbar\inv] = 0$ for $i\geq 1$, \ie assertion (b) holds for $N$ as well.  This proves the inductive step, hence the claim.

Finally, we need to show: 
\begin{claim}
The adjunction 
$$
\Hamp \circ \Ham \rightarrow \id
$$
is an isomorphism of functors on $\Good{(\W_S,\Cs)}$. 
\end{claim}
Suppose $M\subset M[\hbar\inv]$ is a lattice. Taking kernels and cokernels gives
$$
0 \rightarrow E \rightarrow \Hamp \circ \Ham (M) \rightarrow M \rightarrow E' \rightarrow 0.
$$
Applying $\Ham$ and localizing gives
$$
0 \rightarrow \Ham(E)[\h^{-1}] \rightarrow \Ham \circ \Hamp \circ \Ham (M)[\h^{-1}] \stackrel{\id}{\longrightarrow} \Ham(M)[\h^{-1}] \rightarrow \Ham(E')[\h^{-1}] \rightarrow \cdots
$$
This implies that $\Ham(E)[\h^{-1}] = 0$ and hence $\Ham(E)$ is $\h$-torsion.  But by Remark \ref{tf-ness}, $\Hamp(\Ham(M))$ is $\hbar$-torsion-free, hence so is $E$, hence again by the remark so is $\Ham(E)$; this implies that $\Ham(E) = 0$. But $\Ham$ is conservative, so $E = 0$. Thus, we have 
$$
0 \rightarrow \Hamp \circ \Ham (M) \rightarrow M \rightarrow E' \rightarrow 0.
$$
Again, applying $\Ham$, localizing, and using assertion (b) above to obtain that the extension group $\Ext^1_A(A/J, \Hamp \circ \Ham (M))[\hbar\inv]$ is zero, we get an exact sequence
$$
0 \rightarrow \Ham(M)[\h^{-1}] \stackrel{\id}{\longrightarrow} \Ham(M)[\h^{-1}] \rightarrow \Ham(E')[\h^{-1}] \rightarrow 0
$$
This implies that $\Ham(E') = \Hom_A(A/J,M)$ is $\h$-torsion.  Let $E'_{tf}$ denote the quotient of $E'$ by its $\hbar$-torsion submodule.  Using the exact sequence 
\bd
0\rightarrow \Hom_A(A/J, E'_{\mathrm{tor}})\rightarrow \Hom_A(A/J,E') \rightarrow \Hom_A(A/J, E'_{tf})\rightarrow \Ext^1_A(A/J, E'_{\mathrm{tor}})
\ed
and that the left-hand and right-hand terms are $\hbar$-torsion, together with Remark \ref{tf-ness}, 
we conclude that $\Ham(E'_{tf})\subset \Ham(E'_{tf})[\hbar\inv]  = \Ham(E')[\hbar\inv] = 0$.  Since $\Ham$ is conservative, $E'_{tf} = 0$, and thus $E'$ is a torsion module. The claim follows.  This completes the proof of the theorem.
\end{proof}

\subsection{Quantum Coisotropic Reduction: Global Case}\label{sec:QCRglobal}

In this section, we fix a connected component $Y$ of the fixed point set of $\resol$ and let $C \subset \resol$ denote the set of points converging to $Y$ under the elliptic $\Cs$-action. We assume that $C$ is \textit{closed} in $\resol$. Let $\Good{\cW_{\resol}}_C$ denote the category of $\Cs$-equivariant, good $\cW_{\resol}$-modules supported on $C$.

\begin{lem}\label{lem:nicecover}
There exists an affine $\Cs$-stable open covering $\{ U_i \}_{i \in I}$ of $\resol$ such that 
$$
C \cap  U_i = \{ x \in U_i \ | \ \lim_{t \rightarrow \infty} t \cdot x \in Y \cap U_i \}
$$
for all $i$.
\end{lem}

\begin{proof}  
First choose a $\Cs$-stable affine open covering $\{ V_i \}$ of $\resol \smallsetminus C$. Replacing $\resol$ by $\resol \smallsetminus D$, where $D = \resol^{\Cs} \smallsetminus Y$, we may assume that if $\lim_{t \rightarrow \infty} t \cdot x$ exists then it belongs to $Y$. Now take any collection of affine $\Cs$-stable open subsets $V_j'$ of $\resol$ such that a) $V_j' \cap Y \neq \emptyset$ for all $j$, and b) $\bigcup_j U_j' \cap Y = Y$. Then Lemma \ref{lem:limiting-affine} implies that $\{ V_i \} \cup \{ V_j'\}$ is the desired covering. 
\end{proof}

Let $\mc{I}\subset\theo_{\resol}$ denote the ideal of $C$.
\begin{lem}
There exists an $\h$-flat quantization $\mc{J}$ of $\mc{I}$. 
\end{lem}

\begin{proof}
Let $\{ U_i \}_{i \in I}$ be a $\Gm$-stable  open cover of $\resol$ satisfying the properties of Lemma \ref{lem:nicecover}.  It suffices to construct a sheaf $\mc{J}_i$ on each $U_i$, together with a natural identification on over-laps $U_i \cap U_j$. If $U_i \cap C = \emptyset$ then we set $\mc{J}_i = \sA$. If $U_i \cap C \neq \emptyset$, then $\mc{J}_i$ is defined to be the coherent sheaf associated to the left ideal of $\Gamma(U_i, \sA)$ generated by all homogeneous sections of negative degree. In each case, $\mc{J}_i$ is a subsheaf of $\sA |_{U_i}$. Therefore, it suffices to show that $\mc{J}_i |_{U_i \cap U_j} = \mc{J}_j |_{U_i \cap U_j}$ as subsheaves of $\sA |_{U_i \cap U_j}$. 

If $U_i \cap C = U_j \cap C = \emptyset$ there is nothing to prove. Therefore, we assume that $U_i \cap C \neq \emptyset$. If $U_i \cap U_j \cap C = \emptyset$, let $D(f)\subset U_i\cap U_j$ be any $\Gm$-stable affine open subset that is the complement of the vanishing set of $f\in \Gamma(U_i, \theo_{U_i})$.  Then $C\cap D(f) = \emptyset$, \ie if $I\subset\Gamma(U_i,\theo_{U_i})$ is the ideal of $C\cap U_i$ then $I[f\inv] = \Gamma(U_i,\theo_{U_i})[f\inv]$.  It follows
that $Q^\mu_f(\mc{J}) = Q^\mu_f(\sA(U_i))=\sA(D(f))$.  The argument being symmetric in $i$ and $j$ and applying to all $\Gm$-stable affines in an open cover of $U_i\cap U_j$, we get $\mc{J}_i|_{U_i\cap U_j} = \sA|_{U_i\cap U_j} = \mc{J}_j|_{U_i\cap U_j}$.

Finally, assume $U_i \cap U_j \cap C \neq \emptyset$. Since $U_i \cap U_j$ is affine, it suffices to show that $\mc{J}_i |_{U_i \cap U_j} = \mc{J}_{i,j}$, where $\mc{J}_{i,j}$ is the left ideal in $\mc{A}_{i,j} := \mc{A} |_{U_i \cap U_j}$ generated by negative sections. Noting that $\mc{J}_i |_{U_i \cap U_j}$ is clearly contained in $\mc{J}_{i,j}$, we have 
$$
0 \rightarrow \mc{K} \rightarrow \mc{A}_{i,j}  / (\mc{J}_i |_{U_i \cap U_j}) \rightarrow \mc{A}_{i,j} / \mc{J}_{i,j} \rightarrow 0.
$$
By Proposition \ref{prop:AJflat}, $\mc{A}_{i,j} / \mc{J}_{i,j}$ is $\h$-flat, therefore tensoring by $\C[\![\h]\!]/ (\h)$, and applying Lemma \ref{lem:reducedI} we get 
$$
0 \rightarrow \mc{K}_0 \rightarrow \mc{O}_{U_i \cap U_j} / (\mc{I}_i |_{U_i \cap U_j}) \rightarrow \mc{O}_{U_i \cap U_j} / \mc{I}_{i,j} \rightarrow 0.
$$
But this is just the sequence $0 \rightarrow \mc{K}_0 \rightarrow \mc{O}_{C_{i,j}} \rightarrow \mc{O}_{C_{i,j}} \rightarrow 0$, where $C_{i,j} =U_i \cap U_j \cap C$. By Nakayama’s Lemma, this implies that $\mc{K} = 0$. 
\end{proof}

Let us recall that $\pi:C \to S$ denotes the morphism of symplectic reduction. Since the sheaf $\sA/\mc{J}$ is supported on $C$, we shall denote the sheaf restriction $i^{-1}(\sA/\mc{J})$ simply by $\sA/\mc{J}$. Under this convention, we have:

\begin{prop}\label{prop:calB}
$\mc{B} = (\pi_{\bullet}\mc{E}nd_{\sA}(\sA / \mc{J}))^{\mathrm{opp}}$ is a deformation-quantization of $S$. 
\end{prop}

\begin{proof}
This is a local statement. Thus, the proposition follows from Corollary \ref{cor:BDQ}. 
\end{proof}

Let $\cW_S = \mc{B}[\h^{-1}]$; by Proposition \ref{prop:calB}, $\cW_S$ is a $\cW$-algebra on $S$.  
As we did in the affine setting in Section \ref{sec:QCRaffine}, 
define an adjoint pair $(\Hamp,\Ham)$ of functors of DQ-modules by
\begin{equation}
\Ham(\ms{M}') = \pi_{\bullet} \mc{H}om_{\sA}(\sA / \mc{J} , \ms{M}') \quad \textrm{and} \quad  \Hamp (\ms{N}') = \pi^{-1}(\sA / \mc{J}) \o_{\pi^{-1}\mc{B}} \ms{N}'. 
\end{equation}
As in the affine setting, these functors preserve $\hbar$-torsion modules and thus descend to an adjoint pair on the $\cW$-module categories defined by 
\begin{equation}\label{eq:defs-of-ham-p}
\Ham(\ms{M}) = \pi_{\bullet} \mc{H}om_{\sA}(\sA / \mc{J} , \ms{M}')[\hbar\inv] \quad \textrm{and} \quad  \Hamp (\ms{N}) = (\pi^{-1}(\sA / \mc{J}) \o_{\pi^{-1}\mc{B}} \ms{N}')[\hbar\inv], 
\end{equation}
where $\ms{M}'$ and $\ms{N}'$ are choices of lattice in $\ms{M}$ and $\ms{N}$, respectively.
\begin{thm}\label{thm:equivmain1}
The adjoint functors $(\Hamp,\Ham)$ defined by \eqref{eq:defs-of-ham-p} form a pair of exact, mutually quasi-inverse equivalences of abelian categories 
$$
\xymatrix{
\Good{\cW_{\resol}}_C \ar@/^1pc/[rr]^{\Ham}  & & \ar@/^1pc/[ll]^{\Hamp}  \Good{\cW_S}. 
}
$$
\end{thm}

\begin{proof}
It suffices to check locally that the canonical adjunctions $\id \rightarrow \Ham \circ \Hamp$ and $\Hamp \circ \Ham \rightarrow \mathrm{id}$ are exact isomorphisms. Therefore the theorem follows from Theorem \ref{thm:mainlocalequiv}.
\end{proof}

\subsection{Support} The support of modules is well behaved under the functor of quantum coisotropic reduction. 

\begin{prop}\label{prop:support}
Let $\ms{M} \in \Good{\cW_{\resol}}_C$ and $\ms{N} \in \Good{\cW_S}$. Then
$$
\Supp \ \Ham(\ms{M}) = \pi \left( \Supp \ \ms{M} \right), \quad \Supp \ \Hamp (N) = \pi^{-1} \left( \Supp \ \ms{N} \right).
$$
\end{prop}

\begin{proof}
Since support is a local property, we may assume that $S \hookrightarrow C = S \times V \twoheadrightarrow S$. Let $N$ be the global sections of a lattice for $\ms{N}$. Since $\Ham$ is an equivalence it suffices to show that $\Supp \ \Hamp (N) = \pi^{-1} \left( \Supp \ N \right)$ and $\pi \left( \pi^{-1} \left( \Supp \ N \right) \right) = \Supp \ N$. As noted in \cite[Proposition 1.4.3]{KSDQ},
$$
\gr_{\h} ( A/ J \o_B N) = (\gr_{\h} A / J) \o^L_{B_0} (\gr_{\h} N).
$$
Hence, using the fact that $A/J$ is $\h$-flat and $H^0(\gr_{\h} A / J)$ is free over $B_0$, 
$$
\Supp \ \Hamp (N) = \Supp \ \gr_{\h} ( A/ J \o_B N) = V \times \Supp \ \gr_{\h} (N) = V \times \Supp \ N.
$$
From this, both claims are clear. 
\end{proof}

We recall that a holonomic $\cW_{\resol}$-module $\ms{M}$  is said to be \textit{regular} if it admits a lattice $\ms{M}'$ such that the support of the $\mc{O}_{\resol}$-module $\ms{M}' / \hbar\ms{M}'$ is reduced. Proposition \ref{prop:support} implies that the functors $\Ham$ and $\Hamp$ preserve both holonomicity and regular holonomicity: 
\begin{corollary}\label{cor:preshol}
The functor $\Ham$ restricts to equivalences 
$$
\Ham : \Hol{\cW_{\resol}}_C \stackrel{\sim}{\longrightarrow} \Hol{\cW_S}, \quad \Ham : \Reghol{\cW_{\resol}}_C \stackrel{\sim}{\longrightarrow} \Reghol{\cW_S}.
$$
\end{corollary}

\begin{proof}
The first claim follows directly from Proposition \ref{prop:support}. 

For the second claim, we may assume given a regular lattice.  Then it suffices to check that  applying either functor of DQ-modules yields again a regular lattice; moreover, this can be checked locally.  We thus revert to the affine setting of Section \ref{sec:QCRaffine}.  Suppose first that $M$ is an $\hbar$-torsion-free $A$-module for which $M/\hbar M$ has reduced support.  We use the following variant of Lemma 7.13 of \cite{KTDerived}, whose proof is identical.
\begin{lemma}\label{base change of Hom}
Let $\cR$ be a Noetherian, flat $\C[t]$-algebra (in particular, $\C[t]$ is central in $\cR$).  Suppose that $M$ is an $\cR$-module of finite type and $N^\bullet$ is a complex of $\C[t]$-flat $\cR$-modules.  Then for any $a\in \C$,
\bd
\C[t]/(t-a)\otimes_{\C[t]}\Hom_{\cR}(M, N^\bullet) \cong \Hom_{\cR/(t-a)\cR}(M/(t-a)M, N^\bullet/(t-a)N^\bullet).
\ed
\end{lemma}

\begin{comment}
LEAVE IN FILE!
\begin{proof}
Note that a module is $\C[t]$-flat if and only if it is $\C[t]$-torsion free.   
Suppose $N_i$ is torsion free over $\C[t]$.  Then for $\phi \in \Hom(M, N_i)$ and
$f\in \C[t]$, $f\cdot \phi = 0$ iff $0 = f\phi(m)$ for all $m\in M$ iff $\phi(m) = 0$ for all $m\in M$; so $\Hom(M,N_i)$ is torsion free, hence it and all its submodules are $\C[t]$-flat.  

Now take a finite free presentation $P_1\rightarrow P_0 \rightarrow M\rightarrow 0$ of $M$.  We get an exact sequence 
\bd
0\rightarrow \Hom_{\cR}(M, N^\bullet) \rightarrow  \Hom_{\cR}(P_0, N^\bullet) \rightarrow  \Hom_{\cR}(P_1, N^\bullet).
\ed
Tensoring with $\C[t]/(t-a)$ and using the conclusion of the previous paragraph, we get an exact sequence
\bd
0\rightarrow \C[t]/(t-a)\otimes \Hom_{\cR}(M, N^\bullet) \rightarrow  \C[t]/(t-a)\otimes\Hom_{\cR}(P_0, N^\bullet) \rightarrow  \C[t]/(t-a)\otimes\Hom_{\cR}(P_1, N^\bullet).
\ed
As in \cite[Proposition~2.10]{Eisenbud}, there is a canonical map from that complex to 
\bd
\begin{split}
0\longrightarrow \Hom(M/(t-a)M, N^\bullet/(t-a)N^\bullet) \longrightarrow  \Hom(P_0/(t-a)P_0, N^\bullet/(t-a)N^\bullet) \\ \longrightarrow \Hom(P_1/(t-a)P_1, N^\bullet/(t-a)N^\bullet).
\end{split}
\ed
Arguing as in \cite[Proposition~2.10]{Eisenbud}, the maps for $P_0$ and $P_1$ are isomorphisms, implying that the map for $M$ is an isomorphism, as claimed.
\end{proof}
\end{comment}
By the lemma we, $\Hom_A(A/J,M)\otimes_{\C[\![\hbar]\!]} \C[\![\hbar]\!]/(\hbar) \cong \Hom_R(R/I, M/\hbar M).$
The support of the last module is the scheme-theoretic intersection of $C$ with $\on{Supp}(M/\hbar M)$, but since the latter is set-theoretically contained in $C$ and is assumed to be reduced, this intersection is reduced.  Since $\C[S]\subset R/I$, the annihilator of 
$\Hom_R(R/I, M/\hbar M)$ in $\C[S]$ is thus also a radical ideal, as required.

Suppose, on the other hand, that $N$ is a $B$-module with $N/\hbar N$ reduced; we must show that $A / J \o N$ is also regular. For this we note the equivalence 
\[
(A/J \otimes_{B} N)/ \h \stackrel{\sim}{\longrightarrow} (A/J)/ \h \otimes_{B/ \h} (N/ \h) \stackrel{\sim}{\longrightarrow} (R/I) \otimes_{\C[S]} (N/ \h)
\]
which follows from standard base-change identities. If we view the right hand side as a module over $R/I$, then it follows that 
\[
\ann_{R/I} ((R/I) \otimes_{\C[S]} (N/\h)) = R/I \cdot \ann_{\C[S]}(N/\h)
\]
and since the map $T \to R/I$ is (locally) an inclusion of polynomial rings, we have that $\ann_{\C[S]}(N/\h)$ is reduced implies $\ann_{R/I} ((R/I) \otimes_{\C[S]} (N/\h))$ is reduced. Finally, we have that the annihilator of $(R/I) \otimes_{\C[S]} (N/\h)$ as an $R$-module is the pre-image  under $R \to R/I$ of the annihilator as an $R/I$-module. The result follows because $I\subset R$ is a generated by a regular sequence. 
\end{proof}

\subsection{Filtered Quantizations}\label{sec:TDO}

Using Theorem \ref{thm:equivmain1} we can reduce the study of equivariant modules supported on a smooth closed coisotropic subvariety $C\subseteq\resol$ to the situation where the set of fixed points $Y$ of our symplectic manifold $\resol$ with elliptic $\Gm$-action is connected. In this section we study deformation-quantization algebras on such manifolds. Such algebras are equivalent to filtered quantizations, which we now define.

Suppose $\resol$ is a smooth symplectic variety with elliptic $\Gm$-action and connected fixed locus $Y = \resol^{\Gm}$.  Let $\rho : \resol \rightarrow Y$ be the projection. Recall that, as in Theorem \ref{cor:sympafffactor}, the group scheme $T^* Y$ acts on $\resol$ so that the quotient $\SF = \resol / T^* Y$ is an elliptic symplectic fibration. In this section we write $\mc{O}_{\resol}$ for the sheaf $\rho_{\bullet} \mc{O}_{\resol}$ of algebras on $Y$. 

\begin{defn}
A {\em filtered quantization} of $\resol$ is a sheaf of quasi-coherent $\mc{O}_Y$-algebras $\dd_{\resol}$ equipped with an algebra filtration $\dd_{\resol} = \bigcup_{i \ge 0} \mc{F}_i \dd_{\resol}$ by coherent $\mc{O}_Y$-submodules and
an isomorphism 
\bd
\alpha: \gr_{\mc{F}} \dd_{\resol} \stackrel{\sim}{\rightarrow} \mc{O}_{\resol} \hspace{1em} \text{of $\mc{O}_Y$-algebras},
\ed
such that, for $D \in \mc{F}_i \dd_{\resol}$ and $D' \in \mc{F}_j \dd_{\resol}$ , $[D,D'] \in \mc{F}_{i + j - l} \dd_{\resol}$ defines a Poisson bracket on $\gr_{\mc{F}} \dd_{\resol}$, making $\alpha$ an isomorphism of Poisson algebras. 
\end{defn}

\begin{prop}\label{prop:twistedequiv}
The sheaf $\dd_{\resol} := (\rho_{\bullet} \cW)^{\Cs}$ is a filtered quantization of $\resol$. The functor $\ms{M} \mapsto (\rho_{\bullet} \ms{M})^{\Cs}$ defines an equivalence $\Good{\cW} \stackrel{\sim}{\rightarrow}  \Lmod{\dd_{\resol}}$.  
\end{prop}

\begin{proof}
Let $\sA$ be the DQ-algebra on $\resol$ such that $\cW = \mc{A}[\h^{-1}]$. We note that since $\Cs$ acts on $\sA$ with positive weights, the sheaf of algebras $\widetilde{\D}_{\resol}:= (\rho_{\bullet} \mc{A})^{\rat}$ is a polynomial quantization of ${\resol}$; \ie $\widetilde{\D}_{\resol}$ is a flat sheaf of $\mathbb{C}[\h]$-algebras satisfying $\widetilde{\D}_{\resol} / \h \widetilde{\D}_{\resol} \simeq \mc{O}_{\resol}$; see \cite{LosevQaunt} for a proof of this fact. From this it follows easily that $\widetilde{\D}_{\resol} /(\h-1) \simeq (\widetilde{\D}_{\resol}[h^{-1}])^{\Cs} = \D_{\resol}$. Furthermore, via this isomorphism $\D_{\resol}$ inherits a filtration from the grading on $\widetilde{\D}_{\resol}$, inducing an isomorphism $\widetilde{\D}_{\resol}  \simeq \Rees{(\D_{\resol})}$; see \cite{LosevQaunt}, section 3.2 for a detailed discussion. Summing up, we see that $\mc{A}$ can be recovered as the $\h$-completion of the algebra $\Rees({\D_{\resol}})$. 

Now, for any finitely generated ${\D_{\resol}}$-module $\mathcal{N}$, we may choose a good filtration on $\mc{N}$ and obtain the $\widetilde{\D}_{\resol}$-module $\Rees(\mc{N})$. Completing at $\h$ gives a $\sA$-module $\widehat{\Rees(\mc{N})}$, and it is easy to see that the $\cW_{\resol}$-module $\widehat{\Rees(\mc{N})}[h^{-1}]$ doesn't depend on the choice of good filtration on $\mc{N}$. Now one checks directly that this is a well-defined functor which is quasi-inverse to the one above.
\end{proof}

When $Y$ is a single point, $\resol$ is isomorphic to $\mathbb{A}^{2n}$ and a filtered quantization of $\resol$ is isomorphic to the Weyl algebra on $\mathbb{A}^n$, equipped with a filtration whose pieces are all finite dimensional. In the other extreme, when $2 \dim Y = \dim \resol$, one gets a sheaf of twisted differential operators (we refer the reader to \cite{BB} for basics on twisted differential operators, Lie and Picard algebroids). If 
$0\longrightarrow \theo \longrightarrow \mc{P}\longrightarrow \Theta_Y \longrightarrow 0$
 is a Picard algebroid on $Y$, then we denote by $U(\mc{P})$ the enveloping algebroid of the Picard algebroid \cite{BB}. 
 Since $\mc{P}$ is locally free as an $\mc{O}_Y$-module, the algebra $U(\mc{P})$ is equipped with a canonical filtration such that $\gr U(\mc{P}) \simeq \Sym \Theta_Y$.  

\begin{lem}\label{lem:TDO}
When $\resol = T^* Y$, $\dd_{\resol}$ is a sheaf of twisted differential operators on $Y$. 
\end{lem}

\begin{proof}
We construct a Picard algebroid $\mc{P}$ on $Y$ and show that $\dd_{\resol}$ is isomorphic, as a filtered algebra, to the sheaf of twisted differential operators $U(\mc{P})$. The assumption $\resol = T^* Y$ implies that $\mc{F}_i \dd_Y = \mc{O}_Y$ for $0 \le i < l$ and we have a short exact sequence 
$$
0 \rightarrow \mc{O}_Y \rightarrow \mc{P} \stackrel{\sigma}{\longrightarrow} \Theta_{Y} \rightarrow 0
$$
where $\mc{P} = \mc{F}_{l} \dd_\resol$. Here $\sigma(D)(f) = [D,f]$ and $\mc{P}$ is closed under the commutator bracket on $\dd_\resol$. Therefore we get a filtered morphism $U(\mc{P}) \rightarrow \dd_{\resol}$ whose associated graded morphism is just the identity on $\Sym \Theta_{Y}$.  
\end{proof}
\noindent
In fact, one can show that there is an equivalence of categories between twisted differential operators on $Y$ and $\Cs$-equivariant deformation-quantizations of $T^* Y$.

\subsection{Refining Kashiwara's Equivalence}\label{sec:TDO2}

We continue to assume that $\resol$ is a smooth symplectic variety with elliptic $\Gm$-action and that the set of $\Cs$-fixed points $Y$ in $\resol$ is connected. Let $Y' \subset Y$ be a closed, smooth subvariety, $i : Y' \hookrightarrow Y$ the embedding and $\mc{I}$ the sheaf of ideals in $\mc{O}_{Y}$ defining $Y'$. If $C' = \rho^{-1}(Y')$, then Corollary \ref{cor:bigcoiso} says that $C'$ is coisotropic and admits a reduction $C' \rightarrow S'$.  

\begin{prop}\label{prop:qhr2}
The sheaf $\mc{E}nd_{\dd_{\resol}}(\dd_{\resol} / \dd_{\resol} \mc{I})^{\mathrm{opp}}$ is a filtered quantization of $S'$. 
\end{prop}

In order to establish the above proposition we first consider the case where $\resol = T^* Y$. Let $\dd_{\resol}$ be a filtered quantization of $\resol$. By Lemma \ref{lem:TDO}, $\dd_{\resol} \simeq U(\mc{P})$ is a sheaf of twisted differential operators on $Y$. The following is the analogue of Kashiwara's equivalence for twisted differential operators.  

\begin{lem}\label{lem:TDO2}
Let $\mc{P}$ be a Picard algebroid on $Y$ and $\dd_Y = U(\mc{P})$ the associated sheaf of twisted differential operators. 
\begin{enumerate}
\item The sheaf $\mc{P}'$ associated to the presheaf $\{ s \in \mc{P} / \mc{P} \mc{I} \ | \ [\mc{I}, s] = 0 \}$ is a Picard algebroid on $Y'$. 
\item The sheaf of twisted differential operators $U(\mc{P}')$ associated to the Picard algebroid  $\mc{P}'$ is isomorphic to $\mc{E}nd_{\dd_{Y}}( \dd_Y / \dd_Y \mc{I})^{\opp}$. 
\item The functor of Hamiltonian reduction 
$$
\ms{M} \mapsto \mc{H}om_{\dd_{Y}}( \dd_Y / \dd_Y \mc{I}, \ms{M})
$$
defines an equivalence between the categories of quasi-coherent $\dd_Y$-modules supported on $Y'$ and quasi-coherent $U(\mc{P}')$-modules. 
\end{enumerate}
\end{lem}

\begin{proof}
The proof is analogous to the untwisted case; therefore we only sketch the proof. The main difference is that one must work in a formal neighborhood of each point of $Y'$ since Picard algebroids do not trivialize in the \'etale topology; see \cite{BB}. One can check that $\mc{P}'$ is a sheaf of $\mc{O}_{Y'}$-modules, which is coherent since $\mc{P}$ is a coherent $\mc{O}_Y$-module. The anchor map $\sigma : \mc{P} \rightarrow \Theta_Y$ descends to a map $\sigma' : \mc{P}' \rightarrow \Theta_{Y'}$ such that $\mc{O}_{Y'}$ is in the kernel of this map. To show that the sequence 
$$
0 \rightarrow \mc{O}_{Y'} \rightarrow \mc{P}' \stackrel{\sigma}{\longrightarrow} \Theta_{Y'} \rightarrow 0
$$
is exact, it suffices to consider the sequence as a sequence of $\mc{O}_Y$-modules and check exactness in a formal neighborhood of each point of $Y' \subset Y$. But then $\mc{P}$ trivializes and the claim is clear. The other statements are analogously proved by reducing to the untwisted case in the formal neighborhood of each point of $Y'$. 
\end{proof}

\begin{proof}[Proof of Proposition \ref{prop:qhr2}] Since each graded piece $(\mc{O}_{\resol})_i$ is a locally free, finite rank $\mc{O}_Y$-module, it follows by induction on $i$ that each piece $\mc{F}_i \dd_{\resol}$ is locally free of finite rank over $\mc{O}_Y$. Therefore by vanishing of $\Tor_1^{\mc{O}_{Y}}(\mc{O}_{Y'}, - )$ the sequence 
$$
0 \rightarrow \mc{F}_{i-1} / \mc{F}_{i-1} \mc{I} \rightarrow \mc{F}_{i} / \mc{F}_{i} \mc{I} \rightarrow (\mc{O}_{\resol})_i / (\mc{O}_{\resol})_i \mc{I} \rightarrow 0
$$
is exact. Consequently, $\gr_{\mc{F}} (\dd_{\resol} / \dd_{\resol} \mc{I}) \simeq  \mc{O}_{\resol} / \langle \mc{I} \rangle$ and the fact that $\dd_{\resol}$ quantizes the Poisson bracket on $ \mc{O}_{\resol}$ implies that we have an embedding 
\begin{align*}
\gr_{\mc{F}} \mc{E}nd_{\dd_{\resol}}(\dd_{\resol} / \dd_{\resol} \mc{I})^{\mathrm{opp}} = &   \gr_{\mc{F}}( \{ s \in \dd_{\resol} / \dd_{\resol} \mc{I} \ | \ [\mc{I}, s] = 0 \})  \\ & \hookrightarrow \{ f \in \mc{O}_{\resol} / \langle \mc{I} \rangle \ | \ \{ \mc{I} , f \} = 0 \} = \mc{O}_{S'}, 
\end{align*}
where the final identification is (\ref{eq:coiso2}). Therefore it suffices to show that the embedding is an isomorphism. 

We will do this by \'etale base change. Let $\phi : U \rightarrow Y$ be an \'etale map and let $\resol' = U \times_Y \resol$. Assume now that $U$ is affine, and replacing $Y$ by the image of $U$ we will assume that $Y$ is too. Let $\sA_{\resol}$ be the sheaf of DQ-algebras on $\resol$ corresponding to $\D_{\resol}$ via Proposition \ref{prop:twistedequiv}. Then we have a $\C[\![\h]\!]$-module isomorphism $\sA_{\resol} \simeq \prod_{n \ge 0} \mc{O}_{\resol} \h^n$. Since the multiplication in $\sA_{\resol}$ is given by polydifferential operators, it uniquely extends to a multiplication structure on $\prod_{n \ge 0} \mc{O}_{\resol'} \h^n $, which by abuse of notation we write as $\phi^* \sA_{\resol}$. This shows that filtered quantizations behave well under \'etale base change. Shrinking $U$ if necessary, Proposition \ref{thm:hardfactor}, together with the Bezrukavnikov-Kaledin  classification, implies that there is an equivariant isomorphism $\phi^* \sA_{\resol} \simeq \sA_{T^* U} \ \widehat{\boxtimes} \ \mc{D}$ for some $\Cs$-equivariant quantization $\sA_{T^* U}$ of $T^* U$. Inverting $\h$ and taking $\Cs$-invariance as in Proposition \ref{prop:twistedequiv}, we get an isomorphism of filtered algebras $\phi^* \dd_{\resol} \simeq \dd_{T^* U} \boxtimes \D(\mathbb{A}^{2n})$ where $\D(\mathbb{A}^{2n})$ is the usual Weyl algebra (but equipped with a particular filtration) and $\dd_{T^* U}$ is a filtered quantization of $T^* U$. Let $\mc{I}_U$ be the ideal in $\mc{O}_{U}$ defining $\phi^{-1}(Y')$ in $U$. Then, since $\phi$ is flat, 
$$
\phi^* ( \dd_{\resol} / \dd_{\resol} \mc{I}) \simeq \left( \dd_{T^* U} / \dd_{T^* U} \ \mc{I}_U \right) \boxtimes \D(\mathbb{A}^{n-j}).
$$
Thus, 
$$
\phi^* \mc{E}nd_{\dd_{\resol}}(\dd_{\resol} / \dd_{\resol} \mc{I})^{\mathrm{opp}} \simeq \mc{E}nd_{\dd_{T^* U}} \left( \dd_{T^* U} / \dd_{T^* U} \ \mc{I}_U \right)^{\mathrm{opp}} \boxtimes \D(\mathbb{A}^{n-j}) .
$$
Lemma \ref{lem:TDO} says that $\dd_{T^* U}$ is a sheaf of twisted differential operators on $U$. Hence, by Lemma \ref{lem:TDO2}, $\mc{E}nd_{\dd_{T^* U}} \left( \dd_{T^* U} / \dd_{T^* U} \ \mc{I}_U \right)^{\mathrm{opp}}$ is a sheaf of twisted differential operators on $\phi^{-1}(Y')$. This completes the proof of Proposition \ref{prop:qhr2}.
\end{proof}

Now we may consider the category of coherent $\dd_{\resol}$-modules supported on $Y'$, or equivalently the category of good, $\Cs$-equivariant $\cW_{\resol}$-modules whose support is contained in $C'$. Applying Lemma \ref{lem:TDO2} (3) and the \'etale local arguments of the proof of Proposition \ref{prop:qhr2}, one gets that
$$
\Ham : \Lmod{\dd_{\resol}}_{Y'} \rightarrow \Lmod{\dd_{S'}}, \quad \ms{M} \mapsto \mc{H}om_{\dd_{\resol}}(\dd_{\resol} / \dd_{\resol} \mc{I}, \ms{M}), 
$$
is an equivalence. The arguments involved are analogous to the proof of Theorem \ref{thm:equivmain1} and are omitted.

\subsection{Generalizing Kashiwara's Equivalence} Combining Theorem \ref{thm:equivmain1} with the above equivalence gives a direct generalization of Kashiwara's equivalence. Let $\resol$ be an arbitrary symplectic manifold with elliptic $\Cs$-action, equipped with a $\Cs$-equivariant DQ-algebra $\sA_{\resol}$. Fix a smooth, closed coisotropic attracting locus $\rho : C \rightarrow Y$ and let $\pi : C \rightarrow S$ be the coisotropic reduction of $C$. Let $Y' \subset Y$ be a smooth, closed subvariety and set $C' = \rho^{-1}(Y')$. Proposition \ref{prop:support} implies that the equivalence $\Ham$ of Theorem \ref{thm:equivmain1} restricts to an equivalence between the category $\Good{\cW_{\resol}}_{C'}$ of good $\cW_{\resol}$-modules supported on $C'$ and the category $\Good{\cW_{S}}_{\pi(C')}$ of good $\cW_{S}$-modules supported on $\pi(C')$. By the $\Cs$-equivariance of $\pi$, we have $\pi(C') = \widetilde{\rho}^{-1}(Y')$, where $\widetilde{\rho} : S \rightarrow Y$ is the bundle map. By Corollary \ref{cor:bigcoiso}, there exists a coisotropic reduction $\pi' : \widetilde{\rho}^{-1}(Y') \rightarrow S'$. As we have argued above in terms of filtered quantizations, the category $\Good{\cW_{S}}_{\pi(C')}$ is equivalent to the category $\Good{\cW_{S'}}$ of good $\cW_{S'}$-modules. Thus, we have shown:   

\begin{thm}
The category $\Good{\cW_{\resol}}_{C'}$ of $\Gm$-equivariant good $\cW_{\resol}$-modules supported on $C'$ is equivalent to the category of $\Gm$-equivariant good $\cW_{S'}$-modules. 
\end{thm}

\section{Categorical Cell Decomposition and Applications}\label{sec:apps}

We use the generalized Kashiwara equivalence to show that $\WQcoh$ admits a categorical cell decomposition. As a consequence we are able to calculate additive invariants of this category. 

\begin{center}
{\em In this section, open subsets $U\subset \resol$ are always assumed to be $\Gm$-stable.}
\end{center}

\subsection{Categorical Cell Decomposition}\label{categoricalcell}

Recall from Lemma \ref{lem:partial order} that the closure relation on coisotropic strata is a partial ordering. This provides a topology on the set $\cS$ of indices of strata, so that a subset $K\subseteq \cS$ is closed if and only if $i\in K$ implies that $j\in K$ for all $j\leq i$.  

Given a subset $K\subseteq \cS$, we let $C_K := \cup_{i\in K} C_i$. When $K$ is closed, we let $\WQcoh_K$ denote the full subcategory of $\Cs$-equivariant objects whose support is contained in the closed set $C_K \subset \resol$. The open inclusion $\mf{X}\smallsetminus C_K \hookrightarrow \mf{X}$ is denoted $j_K$. 

The closed embedding $C_K \hookrightarrow \resol$ is denoted $i_K$. For $K \subset L \subset \cS$ closed, the inclusion functor $\WQcoh_K \hookrightarrow \WQcoh_L$ is denoted $i_{K,L,*}$. We have:

\begin{prop}\label{filtration}
\mbox{}
\begin{enumerate}
\item The functors $i_{K,L,*}$ have right adjoints $i_{K,L}^! = \Gamma_{K,L}$ of ``submodule with support'' such that the adjunction $\mathrm{id} \longrightarrow i_{K,L}^! \circ i_{K,L,*}$ is an isomorphism.   
\item The categories $\WQcoh_K$ provide a filtration of $\WQcoh$, indexed by the collection of closed subsets $K$ of $\cS$, by localizing subcategories.
\item The quotient $\WQcoh_L / \WQcoh_K$ is equivalent (via the canonical functor) to \newline $\mathsf{Qcoh}(\cW_{\resol\smallsetminus C_K})_{L \smallsetminus K}$.  
\end{enumerate}
\end{prop}

\begin{proof}
Each module $\ms{M}$ in $\cW\operatorname{-}\mathsf{good}_L$ has a unique maximal submodule $\ms{M}_K$ supported on $C_K$. Then, $i_{K,L}^!(\ms{M}) = \ms{M}_K$ defines a right adjoint to $i_{K,L,*}$ such that the adjunction $\mathrm{id} \longrightarrow i_{K,L}^! \circ i_{K,L,*}$ is an isomorphism.  Both functors are continuous and hence the functors extend to  $\WQcoh$.  

Let $U = \resol\smallsetminus K$. It is clear that the subcategory $\cW\operatorname{-}\mathsf{good}_K$ of $\cW\operatorname{-}\mathsf{good}_L$ is a localizing subcategory, from which part (2) follows. 

Part (3) follows from Corollary \ref{cor:qcohfullesssurj}, though this is not immediate since $C_K$ is a union of cells, and not a single cell. The proof is an easy induction on $|K|$. If $i \in K$ is maximal, then let $K' = K \smallsetminus \{ i \}$. By induction $\WQcoh_L / \WQcoh_{K'}$ is equivalent to $\mathsf{Qcoh}\left(\cW_{\resol\smallsetminus C_{K'}} \right)_{L \smallsetminus K'}$. Under this equivalence, the full subcategory $\WQcoh_{K}  / \WQcoh_{K'}$ is sent to $\mathsf{Qcoh}\left(\cW_{\resol\smallsetminus C_{K'}} \right)_{C_i}$. Corollary \ref{cor:qcohfullesssurj} now implies that the quotient of $\mathsf{Qcoh}\left(\cW_{\resol\smallsetminus C_{K'}} \right)_{L \smallsetminus K'}$ by $\mathsf{Qcoh}\left(\cW_{\resol\smallsetminus C_{K'}} \right)_{C_i}$ is canonically isomorphic to $\mathsf{Qcoh}(\cW_{\resol\smallsetminus C_K})_{L \smallsetminus K}$. 
\end{proof}

\begin{corollary} \label{cor:filter}
The category $\WQcoh$ has a filtration by full, localizing subcategories whose subquotients are of the form $\mathsf{Qcoh}(\cW_{S_i})$, for various $i \in \cS$. 
\end{corollary}

\begin{proof}
Fix $i \in \cS$. Replacing $\cS$ by $\{ j \ | \ j \ge i \}$, we may assume that $C_i$ is closed in $\resol$. Then the corollary follows from Proposition \ref{filtration} and Theorem \ref{thm:main1b}. 
\end{proof}

Since the category $\Qcoh(\cW)_K$ is a Grothendieck category, it contains enough injectives. For $K \subset L \subset \cS$ closed subsets, we let $D_K(\Qcoh(\cW)_L)$ denote the full subcategory of the unbounded derived category  $D(\Qcoh(\cW)_L)$ consisting of those objects whose cohomology sheaves lie in $\Qcoh(\cW)_K$. It is a consequence of Proposition \ref{filtration} and \cite[Lemma 4.7]{KS} that $j_{K,L}$ defines an equivalence 
\begin{equation}\label{eq:quoQcoh}
j_{K,L}^* : D(\Qcoh(\cW)_L) / D_K(\Qcoh(\cW)_L) \simeq D(\Qcoh(\cW)_{L \smallsetminus K}).
\end{equation}
Since $\Qcoh(\cW)_L$ has enough injectives the left exact functor $i^!_{K,L}$ can be derived to an exact functor $\mathbb{R} i^!_{K,L} : D(\Qcoh(\cW)_L) \rightarrow D(\Qcoh(\cW)_K)$ such that the adjunction $\mathrm{id} \rightarrow \mathbb{R} i^!_{K,L} \circ i_{K,L,*}$ is an isomorphism. We have the 

\begin{lemma}\label{lem:DKDK}
The functor $i_{K,L,*}:D(\Qcoh(\cW)_K) \rightarrow D_K(\Qcoh(\cW)_L)$ is an equivalence. 
\end{lemma}

\begin{proof}
The quasi-inverse to this functor is given by $\mathbb{R} i^!_{K,L}$. The already noted isomorphism $\mathrm{id} \rightarrow \mathbb{R} i^!_{K,L} \circ i_{K,L,*}$ gives that $i_{K,L,*}$ is fully faithful. For the other direction, we note:

\begin{claim}
The adjunction $i_{K,L,*} \circ \mathbb{R} i^!_{K,L} (\ms{M}) \to \ms{M}$ is an isomorphism for any $\ms{M} \in D^{b}_K(\Qcoh(\cW)_L)$. 
\end{claim}

\begin{proof}
Let $F = i_{K,L,*} \circ \mathbb{R} i^!_{K,L}$. The proof of the claim is essentially identical to the proof of \cite[Corollary 1.6.2]{HTT}, but we provide details for the reader's convenience. As usual, let $l(\ms{M})$ denote the cohomological length of $\ms{M} \in  D^{b}_K(\Qcoh(\cW)_L)$ \ie it is the difference $\max \{ i \ | \ H^i(\ms{M}) \neq 0 \} - \min \{ j \ | \ H^j(\ms{M}) \neq 0 \}$. The claim will follow by induction. If $l(\ms{M}) = 0$, then we may assume without loss of generality that $\ms{M} \in \Qcoh(\cW)_K$. In this case the claim follows from Theorem \ref{thm:equivmain1}.  

In general, we choose $k \in \Z$ such that $l(\tau^{\le k} \ms{M})$ and $l(\tau^{> k} \ms{M})$ are strictly less than $l(\ms{M})$. Applying $F$ to the triangle $\tau^{\le k} \ms{M} \rightarrow \ms{M} \rightarrow \tau^{> k} \ms{M} \stackrel{[1]}{\longrightarrow} $ gives a commutative diagram
$$
\xymatrix{
 \tau^{\le k} \ms{M} \ar[r] \ar[d]^{\alpha} & \ms{M} \ar[r] \ar[d]^{\beta} & \tau^{> k} \ms{M} \ar[r]^{[1]} \ar[d]^{\gamma} & \\
F(\tau^{\le k} \ms{M}) \ar[r] & F(\ms{M}) \ar[r] &  F(\tau^{> k} \ms{M}) \ar[r]^(0.65){[1]} & 
}
$$
Since $\alpha$ and $\gamma$ are isomorphisms by induction, so too is $\beta$. 
\end{proof}

Since $i_{K,L,*}$ is fully faithful and $t$-exact with respect to the standard $t$-structure, it follows that $\mathbb{R} i^!_{K,L}$ has finite cohomological dimension on $D_K(\Qcoh(\cW)_L)$; and thus that $i_{K,L,*}: D^{b}(\Qcoh(\cW)_K) \rightarrow D^{b}_K(\Qcoh(\cW)_L) $  is an equivalence, with quasi-inverse $\mathbb{R} i^!_{K,L}$. Now the unbounded case follows from the bounded by noting that both functors are continuous.
\end{proof}
\noindent

Let $\mc{T}^c$ be the full subcategory of compact objects in a triangulated or dg category $\mc{T}$.

The full subcategory of $D(\Qcoh(\cW))$ consisting of all objects locally  isomorphic to a bounded complex of projective objects inside $\Good{\cW}$ (\ie the perfect objects) is denoted $\mathsf{perf}(\cW)$. To see that this is well behaved, we first note that $\mathsf{perf}(\cW)$ is contained in $D^{b}_{\Good{\cW}}(\Qcoh(\cW))$; this follows from the fact that any object in $\Qcoh(\cW)$, being a limit of good modules, is in $\Good{\cW}$ if and only if it is locally in $\Good{\cW}$. Next, we recall from Lemma 2.6 of \cite{HMS}, that in fact $D^{b}(\Good{\cW}) \stackrel{\sim}{\rightarrow} D^{b}_{\Good{\cW}}(\Qcoh(\cW))$. So we in fact have $\mathsf{perf} ( \cW) \subseteq D^{b}(\Good{\cW})$. 

\begin{lem}\label{lem:affineperfcoh}
We have 
$$
\mathsf{perf}(\cW) = D^b(\Good{\cW}).
$$
Moreover, when $\resol$ is affine,
$$
\mathsf{perf}(\cW) = D^b(\Good{\cW}) = D(\Qcoh(\cW))^c.
$$
\end{lem}

\begin{proof}
Since both $\mathsf{perf}(\cW)$ and $D^b(\Good{\cW})$ are locally defined full subcategories of $D(\Qcoh(\cW))$, the first statement follows from the second. So we suppose $\resol$ is affine. In this case, Proposition \ref{prop:coheretaffinity} says that the category $\Lmod{\sA}$ is equivalent to $\Lmod{A}$. Under this equivalence the full subcategory of $\h$-torsion sheaves is sent to the subcategory of $\h$-torsion $A$-modules. The quotient of $\Lmod{A}$ by this subcategory is equivalent to $\Lmod{W}$. Hence, since the global section functor commutes with colimits, $\Qcoh(\cW)$ is equivalent to $\LMod{(W,\Cs)}$. Hence $D(\Qcoh(\cW)) \simeq D(\LMod{(W,\Cs)})$. As usual, the projective good objects in $\LMod{(W,\Cs)}$ are precisely the summands of a finite graded free $\cW$ modules. Since the category $\LMod{(W,\Cs)}$ has finite global dimension and $W$ is Noetherian, the claim is now standard.  
\end{proof}

\begin{lemma}
\label{lem:perfectext}
Let $U$ be an open subset of $\resol$ whose complement is a union of coisotropic cells. Then, any perfect complex on $U$ admits a perfect extension to $\resol$.  
\end{lemma}

\begin{proof}
Write $j : U \hookrightarrow \resol$. As noted above, $\mathsf{perf} ( \cW) = D^b(\Good{\cW})$. Therefore it suffices to show that any bounded complex $\ms{M}^{\bullet}$ of good $\cW_U$-modules can be extended to a bounded complex of good $\cW$-modules. We shall construct this as a subcomplex of $j_* \ms{M}^{\bullet}$, where $j_*$ is the right adjoint to $j^*$ given by Corollary \ref{cor:qcohfullesssurj}. Assume that $\ms{M}^0$ is the first non-zero term. By Theorem \ref{keyprop}, there exists a good extension $\ms{N}^0$, which by adjunction maps naturally to $j_* \ms{M}^0$. Since its image is again a coherent extension of $\ms M^0$ we may assume $\ms{N}^0$ is a submodule of $j_*\ms{M}^0$. Similarly we can find some good submodule $\ms{E}^1$ of $j_*\ms M^1$ that extends $\ms{M}^1$. Let $\ms{N}^1$ be the sum inside $j_* \ms{M}^1$ of $\ms E^1$ and the image of $\ms N^0$ under the differential of $j_*\ms M^\bullet$. It is a good $\cW$-module extending $\ms{M}^1$. Continuing in this fashion, the lemma is clear. 
\end{proof}

\begin{lem}
Let $U \subset \resol$ be open. Then the perfect complexes in $D(\Qcoh(\cW_{U}))$ are compact. 
\end{lem}

\begin{proof}
By Lemma \ref{lem:affineperfcoh}, we already know this when $U$ is affine. In general, the result follows from the argument given in example 1.13 of \cite{Neeman}. Namely, given a perfect complex $\ms{P}$, we consider the map of sheaves of $\C(\!(\h)\!)$-modules 
$$
\oplus_i \mc{R}\mc{H}om(\ms{P},\ms{M}_i) \rightarrow \mc{R}\mc{H}om(\ms{P}, \oplus_i \ms{M}_i).
$$
This is an isomorphism since Lemma \ref{lem:affineperfcoh} implies that its restriction to every $\Cs$-stable affine open subset of $U$ is an isomorphism. Then, since $U$ is a Noetherian topological space, Proposition III.29 of \cite{H} implies
$$
\oplus_i \Hom (\ms{P},\ms{M}_i) = H^0(\resol, \oplus_i \mc{R}\mc{H}om(\ms{P},\ms{M}_i)) \stackrel{\sim}{\longrightarrow} H^0(\resol,\mc{R}\mc{H}om(\ms{P}, \oplus_i \ms{M}_i)) = \Hom (\ms{P}, \oplus_i \ms{M}_i). 
$$
\end{proof} 

In the following, we let $\mathcal{S}_i \subset \mathcal{S}$ be a collection of subsets such that $C_{\mathcal{S}_i} = U_i$ is open in $\resol$, $\mathcal{S}_i \subset \mathcal{S}_{i+1}$, and $C_{\mathcal{S}_{i+1}} \smallsetminus C_{\mathcal{S}_i}$ is a union of strata of the same dimension. We have that $C_{\mathcal{S}_0} = U_0 $ is the open stratum, and that $\mathcal{S}_n=\mathcal{S}$ for $n \gg 0$, so that $U_n = \resol$.

\begin{prop} \label{prop:compactgen}
The triangulated category $D(\Qcoh(\cW_{U_i}))$ is compactly generated for all $i$. 
\end{prop}

\begin{proof}
Since every perfect complex is compact, it suffices to show that $D(\Qcoh(\cW_{U_i}))$ is generated by its perfect complexes. The proof is by induction on $i$. When $i = 0$, $U_0$ is a single stratum and Proposition \ref{prop:twistedequiv} implies that $D(\Qcoh(\cW)) \simeq D(\Qcoh(\dd_{\resol}))$. Then, we have the restriction and induction functors
\bd
\xymatrix{
 D(\Qcoh(\dd_{\resol})) \ar[r] &   \ar@<-1ex>[l]^{} D(\Qcoh(\resol^{\Cs})) 
}
\ed
we can argue in exactly the same way as in the case of $\dd$-modules--see \cite[example 1.14]{Neeman}--to see that this category is compactly (indeed, perfectly)  generated. 

Next, we assume that the theorem is known for $U_i$, and we set $C = U_{i+1} \smallsetminus U_i$, a union of closed strata in $\resol$. Then, Theorem \ref{thm:equivmain1} together with Proposition \ref{prop:twistedequiv} implies that $D(\Qcoh(\cW_{U_{i+1}})_C)$ is equivalent to the direct sum of the $D(\Qcoh(\dd_{S_k}))$, where the $S_k$ are the coisotropic reductions of the strata in $C$. As above, these categories are generated by their subcategories of perfect objects. Since the natural functors $D(\Qcoh(\dd_{S_k})) \to D(\Qcoh(\cW_{U_{i+1}})_C)$ take good modules to good modules, we see that $D(\Qcoh(\cW_{U_{i+1}})_C)$ is generated by perfect (and hence compact) objects in $D(\Qcoh(\cW_{U_{i+1}}))$.

Now we wish to show that $D(\Qcoh(\cW)_{U_{i+1}})$ is perfectly generated, given that both of the categories $D(\Qcoh(\cW_{U_{i+1}}))_C$ and $D(\Qcoh(\cW_{U_i}))$ are generated by their perfect objects. Assume that $\ms{M} \in D(\Qcoh(\cW_{U_{i+1}}))$ is such that $\Hom(\ms{P}, \ms{M} ) = 0$ for all perfect objects $\ms{P}$ in $D(\Qcoh(\cW_{U_{i+1}}))$. In particular, $\Hom(\ms{P}, \ms{M} ) = 0$ for all perfect objects $\ms{P}$ in $D(\Qcoh(\cW_{U_{i+1}})_C)$. Thus, \cite[Lemma 1.7]{NeemanSmashing} says that $\ms{M} \simeq j_{i,*} j^*_{i} \ms{M}$. Since $D(\Qcoh(\cW_{U_i}))$ is perfectly generated, $j^*_{i} \ms{M} \neq 0$ implies that there is some perfect object $\ms{Q}$ in $D(\Qcoh(\cW_{U_i}))$ such that $\Hom(\ms{Q}, j^*_{i} \ms{M}) \neq 0$. Take $0 \neq \phi$ in $\Hom(\ms{Q}, j^*_{i} \ms{M}) $. By Lemma \ref{lem:perfectext}, there exists some perfect complex $\ms{Q}'$ on $U_{i+1}$ whose restriction to $U_{i}$ equals $\ms{Q}$. Then the composite $\ms{Q}' \rightarrow j_{i,*} \ms{Q} \rightarrow j_{i,*} j^*_{i} \ms{M}$ is non-zero since its restriction to $U$ equals $\phi$, and we have a contradiction. Thus, $D(\Qcoh(\cW)_{U_{i+1}})$ is perfectly, and hence compactly, generated as claimed.
\end{proof}

\begin{corollary}\label{cor:compactgen}
Let $K \subset L \subset \cS$ be closed subsets.
\begin{enumerate}
\item The subcategory $D_K(\Qcoh(\cW))$ of $D(\Qcoh(\cW))$  is generated by 
$$
D_K(\Qcoh(\cW)) \cap D(\Qcoh(\cW))^c.
$$
\item Let $U = \resol \smallsetminus C_K$. Then the exact functor 
$$
j_{K,L}^* : D_L(\Qcoh(\cW)) \rightarrow  D_{L \smallsetminus K}(\Qcoh(\cW_U))
$$
admits a right adjoint 
$$
j_{K,L,*} : D_{L \smallsetminus K}(\Qcoh(\cW_U)) \rightarrow D_L(\Qcoh(\cW)).
$$ 
\end{enumerate}
\end{corollary}

\begin{proof}
We first prove 1), by induction on $| K |$. We assume that $\cS$ is totally ordered with $K = \{ i \le k \}$. The case $| K | = 1$ has been done in Proposition \ref{prop:compactgen}. We define $\mathsf{perf}_K(\cW)$ to be the full subcategory of $\mathsf{perf}(\cW)$ consisting of all complexes whose cohomology is supported on $C_K$. Since $D_K(\Qcoh(\cW))$ is a full subcategory of $D(\Qcoh(\cW))$, the objects in $\mathsf{perf}_K(\cW)$ are compact in $D_K(\Qcoh(\cW))$. Let $K' = \{ i \le k-1 \}$ so that, by induction, $D_{K'}(\Qcoh(\cW))$ is generated by $\mathsf{perf}_{K'}(\cW)$. Since the objects in $\mathsf{perf}_{K'}(\cW)$ are compact in $D(\Qcoh(\cW))$, construction 1.6 of \cite{NeemanSmashing} says that there exist a right adjoint $j_{K',*}$ to $j_{K'}^*$. Let $\ms{M} \in D_K(\Qcoh(\cW))$ such that $\Hom(\ms{P},\ms{M}) = 0$ for all $\ms{P} \in \mathsf{perf}_{K}(\cW)$. In particular, $\Hom(\ms{P},\ms{M}) = 0$ for all $\ms{P} \in \mathsf{perf}_{K'}(\cW)$ and hence $\ms{M} = j_{K,*} j^*_K \ms{M}$.  Assume that there exists some $\ms{Q} \in \mathsf{perf}_{K \smallsetminus K'}(\cW_U)$ and non-zero morphism $\phi : \ms{Q} \rightarrow j_{K'}^* \ms{M}$. Just as in the proof of Proposition \ref{prop:compactgen}, this implies that there is some $\ms{Q}' \in \mathsf{perf}(\cW)$ and non-zero $\phi' : \ms{Q}' \rightarrow \ms{M}$ extending $\phi$. However, the fact that $\resol = U \sqcup C_{K'}$ and the cohomology of $\ms{Q}$ was assumed to be contained in $C_{K} \smallsetminus C_{K'} \subset U$ implies that the cohomology of $\ms{Q}'$ is supported on $C_K$ \ie $\ms{Q}'$ belongs to $\mathsf{perf}_{K}(\cW)$. Thus, we conclude that $\Hom(\ms{Q}, j^*_{K'} \ms{M}) = 0$ for all $\ms{Q} \in \mathsf{perf}_{K \smallsetminus K'}(\cW_U)$. Since $C_{K} \smallsetminus C_{K'}$ is a single closed stratum in $U$, Proposition \ref{prop:compactgen} implies that $j^*_{K'} \ms{M} = 0$ and hence $\mathsf{perf}_K(\cW)$ generates $D_K(\Qcoh(\cW))$. 

Now we deduce part 2). By construction 1.6 of \cite{NeemanSmashing}, a right adjoint $j_{K,*}$ to $j_{K}^*$ exists. The image of $D_{L \smallsetminus K}(\Qcoh(\cW_U))$ under $j_{K,*}$ is contained in $D_L(\Qcoh(\cW))$ since $\resol = U \sqcup C_K$. Thus, $j_{K,*}$ restricted to $D_{L \smallsetminus K}(\Qcoh(\cW_U))$ is a right adjoint to $j^*_{K,L}$.
\end{proof}

The proof of Proposition \ref{prop:compactgen} shows that $\mathsf{perf}(\cW)$ generates $D(\Qcoh(\cW))$. Since $\mathsf{perf}(\cW)$ equals $D^b(\Good{\cW})$ by Lemma \ref{lem:affineperfcoh}, it is clear that $\mathsf{perf}(\cW)$ is closed under summands; the same is true of $\mathsf{perf}_{K}(\cW)$. Therefore \cite[Theorem 2.1]{Neeman} implies:

\begin{corollary}\label{cor:perfcompact}
For any closed $K\subseteq \cS$, the compact objects in $D_K(\Qcoh(\cW))$ are precisely the perfect complexes \ie $D_K(\Qcoh(\cW))^c = \mathsf{perf}_{K}(\cW)$. 
\end{corollary}

\subsection{Consequences: $K$-Theory and Hochschild and Cyclic Cohomology}\label{sec:homology}

In this section we consider the case where $\resol$ has only finitely many $\Cs$-fixed points. The fact that $\cW \operatorname{-}\mathsf{good}$ admits an algebraic cell decomposition in this case (see definition \ref{defn:algcell}) allows us to inductively calculate $K_0$, and the additive invariants Hochschild and cyclic homology of $\mathsf{perf}(\cW)$. 

When $\resol$ has isolated fixed points, the coisotropic strata $C_i$ are affine spaces and their coisotropic reductions $S_i$ are isomorphic as symplectic manifolds to $T^* \mathbb{A}^{t_i}$, for some $t_i$. Moreover, $\Qcoh(\cW_{S_i}) \simeq \LMod{\dd (\mathbb{A}^{t_i})}$. For each $i \in \cS$ we can form the open subsets $\ge i = \{ j | j \ge i \}$ and $> i = \{ j | j > i \}$. 

\begin{defn}\label{defn:algcell}
Let $\cC$ be an abelian category with a collection of Serre subcategories $\cC_K$ indexed by closed subsets $K$ in a finite poset. We say that the $\cC_K$ form an \textit{algebraic cell decomposition} of $\cC$ if each subquotient $\cC_{\le i} / \cC_{< i}$ is equivalent to the category of modules over some Weyl algebra. 
\end{defn} 

By Corollary \ref{cor:filter}, $\Good{\cW}$ admits an algebraic cell decomposition. Let $\on{DG-cat}_{\C}$ denote the category of all small $\C$-linear dg categories and $\on{DG-vect}_{\C}$ the dg derived category of $\C$-vector spaces. Let $\mathsf{L}$ denote a $\C$-linear functor from $\on{DG-cat}_{\C}$ to $\on{DG-vect}_{\C}$.  We are interested in the case when $\mathsf{L}$ admits a {\em localization formula}.  This means that any short exact sequence of dg categories gives rise, in a natural way via $\mathsf{L}$, to an exact triangle in $\on{DG-vect}_{\C}$. 

For any quantization $\cW$, we denote by $\mathsf{Perf}(\cW)$ the dg category of perfect complexes for $\cW$, see \cite{Keller}, section 4. 

Moreover, we say that $\mathsf{L}$ is \textit{even} if $$
H^i(\mathsf{L}(\mathsf{Perf}(\dd(\mathbb{A}^n)))) = 0, 
$$
for all $n$ and all odd $i$.

\begin{prop}\label{thm:Keller}
Suppose $\mathsf{L}:\on{DG-cat}\longrightarrow \on{DG-vect}$ is an even $\C$-linear functor that admits a localization formula. Then, there is a (non-canonical) splitting:
\bd
\mathsf{L}\big( \mathsf{Perf}(\cW)\big) \cong \bigoplus_i \mathsf{L}\big( \mathsf{Perf}(\dd(\mathbb{A}^{t_i})) \big).
\ed
\end{prop}

\begin{proof}
By induction on $k := |\cS|$, we may assume that the result is true for $U = \resol \smallsetminus C_{ K}$, where $K = \{ k \}$. Lemma \ref{lem:DKDK} together with (\ref{eq:quoQcoh}) implies that we have a short exact sequence 
$$
0 \rightarrow D(\Qcoh(\cW)_K) \rightarrow D(\Qcoh(\cW)) \rightarrow D(\Qcoh(\cW_U)) \rightarrow 0.
$$
By Theorem \ref{thm:equivmain1}, we may identify $D(\Qcoh(\cW)_K)$ with $D(\Qcoh(\cW_S))$, where $S$ is the coisotropic reduction of $C_K$. This in turn can be identified with $D(\LMod{\D(\mathbb{A}^{t_k})})$. By lemma \ref{lem:perfectext}, the sequence
$$
0 \rightarrow \mathsf{perf}(\dd(\mathbb{A}^{n_k})) \rightarrow \mathsf{perf}(\cW) \rightarrow \mathsf{perf} (\cW_U) \rightarrow 0
$$
is exact. Since all functors involved lift to the dg level, we obtain an exact sequence 
\begin{equation} \label{eq:sesperf}
0 \rightarrow \mathsf{Perf}(\dd(\mathbb{A}^{n_k})) \rightarrow \mathsf{Perf}(\cW) \rightarrow \mathsf{Perf} (\cW_U) \rightarrow 0
\end{equation} 

Applying $\mathsf{L}$, we get a triangle 
\begin{equation}\label{eq:triangleL}
\mathsf{L}( \mathsf{Perf}(\dd(\mathbb{A}^{t_k}))) \rightarrow \mathsf{L}(\mathsf{Perf}(\cW)) \rightarrow \mathsf{L}(\mathsf{Perf} (\cW_U)) \rightarrow \mathsf{L}( \mathsf{Perf}(\dd(\mathbb{A}^{t_k})))[1]
\end{equation}
and hence a long exact sequence in cohomology. Therefore, the fact that $\mathsf{L}$ is even implies by induction that  $H^i(\mathsf{L}(\mathsf{Perf}(\cW)) ) = 0$ for all odd $i$, and we have short exact sequences of vector spaces
$$
0 \rightarrow H^{2i} (\mathsf{L}(\mathsf{Perf}(\dd(\mathbb{A}^{t_k})))) \rightarrow H^{2i} (\mathsf{L}(\mathsf{Perf}(\cW))) \rightarrow H^{2i}(\mathsf{L}(\mathsf{Perf} (\cW_U))) \rightarrow 0
$$
for all $i$. The claim of the proposition follows. 
\end{proof}

Next we prove Corollary \ref{cor:homocyc}. By \cite[Theorem 1.5 (c)]{Keller2}, both Hochshild and cyclic homology are localizing functors. Let $\epsilon$ be a variable, given degree two. We recall that the main theorem of \cite{Wodzicki} implies $HH_{*}(\dd(\mathbb{A}^n)) = \mathbb{C}\epsilon^{2n}$. Furthermore, we have $HC_{*}(\dd(\mathbb{A}^n)) = \epsilon^{2n}\mathbb{C}[\epsilon]$.  In particular, $HH_*$ and $HC_*$ are \textit{even} localizing functors. Proposition \ref{thm:Keller} implies that the Hochshild and cyclic homology of $\mathsf{Perf}(\cW)$ are given by 
$$
HH_{\idot}(\mathsf{Perf}(\cW)) = \bigoplus_{i = 1}^k \C \epsilon^{t_i} \quad \textrm{and} \quad HC_{\idot}(\mathsf{Perf}(\cW)) = \bigoplus_{i = 1}^k \epsilon^{t_i} \C[\epsilon]
$$
respectively. Therefore, in order to identify $HH_{\idot}(\mathsf{Perf}(\cW))$ with $H_{\idot - \dim \resol}(\resol)$ as graded vector spaces, it suffices to show that   
$$
H_{\idot}(\resol) = \epsilon^{\frac{1}{2} \dim \resol} \bigoplus_{i = 1}^k \C \epsilon^{t_i}.
$$
This follows from the BB-decomposition of $\resol$, noting that $\dim C_i = \frac{1}{2} \dim \resol + 2 t_i$. 

\subsection{The Grothendieck Group of $\mathsf{Perf}(\cW)$}\label{sec:Kgroup}

Finally, we turn to the proof of Corollary \ref{cor:Ktheory}, which states that the Grothendieck group $K_0(\mathsf{Perf}(\cW))$ is a free $\Z$-module of rank $| \cS |$. Again, the proof is by induction on $k = |\cS|$. Using the Bernstein filtration on $\mathcal{D}(\mathbb{A}^{t_i})$, Theorem 6.7 of \cite{Quillen} says that we have identifications 
$$
K_j(\dd (\mathbb{A}^{t_i})) = K_j(\mathsf{Perf} \dd(\mathbb{A}^{t_i})) \simeq K_j(\C) \quad \forall j,
$$
where $K_j(\dd(\mathbb{A}^{t_i})) $ is the $j$th $K$-group of the exact category of finitely generated projective $\dd(\mathbb{A}^{t_i})$-modules. The higher $K$-groups $K_j(\C)$ for $j \ge 1$, of $\C$ have been calculated by Suslin and are known to be divisible, see \cite[Corollary 1.5]{WeibelKtheory}. Quillen's Localization Theorem, \cite[Theorem 5.5]{Quillen}, says that the short exact sequence (\ref{eq:sesperf}) induces long exact sequences  
\begin{equation}\label{eq:lesK}
\cdots \longrightarrow K_1(\mathsf{Perf} (\cW_U)) \longrightarrow K_0(\mathsf{Perf}(\dd(\mathbb{A}^{t_k})) \longrightarrow K_0(\mathsf{Perf}(\cW)) \longrightarrow K_0(\mathsf{Perf} (\cW_U)) \longrightarrow 0.
\end{equation}
Since a divisible group is an injective $\Z$-module and the quotient of a divisible group is divisible, the subsequence 
$$
0  \longrightarrow K_0(\mathsf{Perf}(\dd(\mathbb{A}^{t_k})) \longrightarrow K_0(\mathsf{Perf}(\cW)) \longrightarrow K_0(\mathsf{Perf} (\cW_U)) \longrightarrow 0
$$
is exact and it follows by induction that $K_0(\mathsf{Perf}(\cW))$ is free of rank $k$.

\subsection{Hochschild Cohomology}\label{sec:HH}
In many instances one can apply Van den Bergh's duality theorem to calculate the Hochschild cohomology of the category $\cW \operatorname{-}\mathsf{good}$. In this subsection, we assume that $\resol$ is a symplectic resolution $f\colon \resol \to X$ where $X$ is an affine cone, \ie $X$ has a $\mathbb G_m$-action with a single attracting fixed point.  Moreover, we assume that $\resol$ arises as the GIT quotient of a $G$-representation $W$, where $G$ is a reductive algebraic group. Thus, $\resol = \mu^{-1}(0) /\!\!/_{\chi} G$ for an appropriate character $\chi$ of $G$, and $X = \text{Spec}(\C[\mu^{-1}(0)])$. Note that in this case, Lemma 3.7 of \cite{KTDerived} shows that Assumptions 3.2 of that paper apply once we assume (as we shall) that the moment map is flat. Let $U$ denote the algebra of $\Cs$-invariant global sections of $\cW_{\resol}$. 

\begin{assumptions}
\label{DH iso assumption}
Suppose that $f\colon \resol \to X$ is a symplectic resolution obtained via Hamiltonian reduction as above. Then if $\mathfrak g = \text{Lie}(G)$ and $\mathfrak{z} = (\mathfrak{g}/[\mathfrak g,\mathfrak g])^*$, we have a natural ''Duistermaat-Heckman'' map $\mathfrak{z}\to H^2(\resol)$. We assume this map is surjective.
\end{assumptions}

\begin{remark}
If $\resol$ is a Nakajima quiver variety, then main theorem of the recent preprint \cite{McNnew} implies that the canonical map $\mathfrak{z} \rightarrow H^2(\resol)$ is surjective.
\end{remark}

Recall that for any $c \in \mathfrak z$ we may define the quantum Hamiltonian reduction of the Weyl algebra associated to $W$. We denote this algebra by $U_c$.

\begin{lemma}
Let $\resol$ be a conic symplectic resolution as above with a $\mathbb G_m$-equivariant quantization $\cW_\resol$. Suppose that Assumption \ref{DH iso assumption} holds for $\resol$. Then the filtered quantization $U$ of $X$ associated to $\cW$ is of the form $U_c$ for some $c \in \mathfrak{z}$.
\end{lemma}
\begin{proof}
In \cite{LosevIso} it is shown that, via the Bezrukavnikov-Kaledin noncommutative period map, graded quantizations of $\resol$ are parametrized by $H^2(\resol,\C)$. Moreover, one can show that, provided Assumption \ref{DH iso assumption} holds, this period map for quantum Hamiltonian reductions $U_c$ for $c \in \mathfrak{z}$ is realized by the Duistermaat-Heckman map, up to a shift corresponding to the canonical quantization. Since the Rees construction gives an equivalence between filtered and graded quantizations, it follows from this that any filtered quantization is of the form $U_c$ for some $c \in \mathfrak{z}$. As the $\Cs$-invariant global sections $U$ of our quantization of $\resol$ gives such a filtered quantization, it is thus of the form $U_c$ for some $c \in \mathfrak{z}$.
\end{proof}

For the rest of this section we fix $c \in \mathfrak z$ such that $U\cong U_c$.

\begin{lemma}
\label{unstableproduct}
Let $W$ be a $G$-representation as above, and $\chi\colon G \to \Cs$. Then $W\oplus W$ is a $G\times G$-representation, and the $(\chi,\chi)$-unstable locus contains $W^{\chi-\text{un}}\times W^{\chi-\text{un}}$.
\end{lemma}
\begin{proof}
Let $W\times \C_\chi$ be the $G$-representation given by $g(w,z)=(g(w),\chi(g) \cdot z)$. 
By the Hilbert-Mumford criterion, a point $x \in W$ is unstable if there is a one-parameter subgroup $\lambda \colon \Cs \to G$ such that $
\lim_{t \to 0} \lambda(t)(x,1) =(0,0)$, for $(x,1)\in W\times \C_\chi$, and similarly for $(x,y) \in W\oplus W$. 
Thus if $x$ and $y$ are both in $W^{\chi-\text{un}}$, destabilized by one-parameter subgroups $\lambda_1,\lambda_2$ respectively, it is clear that $(\lambda_1,\lambda_2)$ gives a one-parameter subgroup of $G\times G$ which destabilizes $(x,y)$, and we are done.

\end{proof}

\begin{lemma}
\label{U_csmooth}
Let $c \colon \mf{g} \rightarrow \C$ be a character of $\mf{g}$ for which the algebra $U= U_c$ has finite cohomological dimension. Then $U_c$ is smooth, that is, $U^e$ has finite global dimension.
\end{lemma}
\begin{proof}
First note that since $U_c$ is (left or right) Noetherian, its global dimension is its Tor-dimension, and hence it has finite global dimension if and only if $U_c^{\text{op}}$ has finite global dimension. Moreover, it follows from the construction of the noncommutative period map in \cite{BK} and the Duistermatt-Heckman theorem (see for example \cite{LosevIso}) that the algebra $U_c^{\text{op}}$ is isomorphic to $U_{\rho-c}$ where $\frac{1}{2}\rho$ is the character corresponding to the canonical quantum moment map, \ie the quantum moment map which yields the canonical quantization of $\resol$. 

To see that $U^e = U_c \otimes U_c^{\text{op}} \cong U_c \otimes U_{\rho-c}$ has finite global dimension, first note that 
\[
U_c\otimes U_{\rho-c} \cong \Gamma(\resol\times \resol, \mathcal E_{\resol,c}\boxtimes \mathcal E_{\resol,\rho-c})
\]
or, in the notation of \cite{KTDerived}, $(M_c\boxtimes M_{\rho-c})^{G\times G}$. Now, since $\resol$ is smooth, Corollary 7.6 of \cite{KTDerived} shows that $U_c\otimes U_{c-\rho}$ has finite global dimension if and only if the pull-back functor $\mathbb Lf^*$ is cohomologically bounded. Explicitly, $\mathbb Lf^*$ is the functor from $D(U_c\otimes U_{c-\rho}-\text{mod})$ to $D_{G,(c,\rho-c)}(A\otimes A-\text{mod})$ given by
\[
N \mapsto \pi((M_c\boxtimes M_{\rho-c})\otimes^{\mathbb L}_{U_c\otimes U_{\rho-c}}N), \quad N \in D(U_c\boxtimes U_{\rho-c}-\text{mod}),
\]
where $\pi$ is the quotient functor given by the $(\chi,\chi)$-unstable locus.  

Choose free $U_c$ and $U_{\rho-c}$ resolutions $P^\bullet$ of $M_c$ and $Q^\bullet$ of $M_{\rho-c}$, respectively. By Corollary 7.6 of \cite{KTDerived}, each of $\pi(P^\bullet)$ and $\pi(Q^\bullet)$ has only finitely many cohomologies, i.e for each of $P^\bullet$, $Q^\bullet$ all but finitely many cohomologies have associated graded supported in $W^{\chi-\text{un}}$.  Hence the resolution $\operatorname{Tot}(P^\bullet\boxtimes Q^\bullet)$ of $M_c\boxtimes M_{\rho-c}$ has all but finitely many cohomologies with associated graded supported in $W^{\chi-\text{un}}\times W^{\chi-\text{un}}$. Thus using Lemma \ref{unstableproduct} it follows the cohomologies are $(\chi,\chi)$-unstably supported after finitely many terms also. Thus $\mathbb Lf^*$ is bounded and hence $U^e$ has finite global dimension as required.
\end{proof}

\begin{remark}
In a number of examples, such as the Hilbert scheme of points in $\mathbb C^2$ or the minimal resolution of $\mathbb C^2/\Gamma_l$ the cyclic quotient singularity, it is known explicitly when the algebra $U_c$ has finite global dimension, moreover it is shown in \cite{BLPW2}, building on work of Kaledin, that the algebra $U_c$ has finite global dimension for sufficiently generic $c$.
\end{remark}

Recall that an algebra $U$ is said to have finite Hochschild dimension (or is \textit{smooth}) if $U$ has a finite resolution when considered as a $U^e = U \o U^{\mathrm{opp}}$-module. 

\begin{prop}\label{prop:chomology}
Assume that $U_c$ has finite global dimension. Then
$$
HH^{\idot}(\mathsf{Perf}(\cW)) = HH^{\idot}(U_c) = H^{\idot}(\resol, \C). 
$$
\end{prop}

\begin{proof}
By \cite [Lemma 3.14]{KTDerived}, the algebra $U_c$ is Auslander Gorenstein with rigid Auslander dualizing complex $\mathbb{D}^{\idot} = U_c$.
By Lemma \ref{U_csmooth}, the enveloping algebra $U^e$ of $U_c$ has finite global dimension, and hence $U_c$ has finite Hochschild dimension, thus we are able to apply Van den Bergh's duality result \cite[Theorem 1]{VdBduality} to conclude that $HH^{\idot}(U_c) = HH_{\dim \resol - \idot}(U_c)$.  Since $U_c$ has finite global dimension, \cite[Theorem 1.1]{KTDerived} and Corollary \ref{cor:homocyc} show that
$$
HH^{\idot} (\mathsf{Perf}(\cW)) = HH^{\idot}(U_c) = HH_{\dim \resol - \idot}(U_c) = H_{\dim\resol - \idot}(\resol)
$$
where in the last equality we use the fact that the degrees in Borel-Moore homology are twice those in Hochschild homology.
On the other hand, Poincar\'e duality \cite[Equation (2.6.2)]{CG} says that $H_{\dim \resol -\idot}(\resol) = H^{\idot}(\resol)$, and so the result follows. 
\end{proof}

Presumably one can also apply the results of \cite{DolVdBduality} and \cite[Section 2.5]{KSDQ} to the category $\mathsf{Perf}(\cW)$ in order to express directly the Hochschild cohomology of that category in terms of its Hochschild homology.  

More generally, the proof of Proposition \ref{prop:chomology} shows that if $\resol$ is a symplectic resolution $f\colon \resol \to X$ of an affine cone $X$, the number of $\Cs$-fixed points on $\resol$ is finite, and $U_c$ is a quantization of $\C [X]$ such that derived localization holds and the enveloping algebra $U^e$ has finite global dimension, then $HH^{\idot}(U_c) = H^{\idot}(\resol,\C)$. We conclude with a number of standard examples, arising from representation theory, where Proposition \ref{prop:chomology}, or the above more general statement, is applicable. 

\begin{example}\label{ex:RCA}
Let $\Gamma$ be a cyclic group and $\mathfrak{S}_n \wr \Gamma$ the wreath product group that acts as a symplectic reflection group on $\C^{2n}$. The corresponding symplectic reflection algebra at $t = 1$ and parameter $\mathbf{c}$ is denoted $\mathsf{H}_{\mathbf{c}}(\mathfrak{S}_n \wr \Gamma)$. Define an increasing filtration $F_{\bullet} (\C[ \mathfrak{S}_n \wr \Gamma])$ on the group algebra of $\mathfrak{S}_n \wr \Gamma$ by letting $F_i(\C[\mathfrak{S}_n \wr \Gamma])$ for $i \ge 0$, be the subspace spanned by all elements $g \in \mathfrak{S}_n \wr \Gamma$ such that $\mathrm{rk}(1 - g) \le k$, where $1 - g$ is thought of as an endomorphism of $\C^{2n}$. This is an algebra filtration and restricts to a filtration $F_{\bullet}( \mathsf{Z} \mathfrak{S}_n \wr \Gamma)$ on the center of the group algebra. The following result, which was proved for generic $\mathbf{c}$ in \cite[Theorem 1.8]{EG} for generic $\mathbf{c}$, follows easily from the results of this paper. 

\begin{prop}\label{prop:sph}
Assume that $\mathbf{c}$ is spherical. Then,   
$$
HH^{\idot}(\mathsf{H}_{\mathbf{c}}(\mathfrak{S}_n \wr \Gamma)) = HH_{2n - \idot}(\mathsf{H}_{\mathbf{c}}(\mathfrak{S}_n \wr \Gamma)) = \gr^F_{\idot}(\mathsf{Z} \mathfrak{S}_n \wr \Gamma),
$$
as graded vector spaces. 
\end{prop}

In \cite{EG}, it is shown that the identification $HH^{\idot}(\mathsf{H}_{\mathbf{c}}(\mathfrak{S}_n \wr \Gamma)) = \gr^F_{\idot}(\mathsf{Z} \mathfrak{S}_n \wr \Gamma)$ is as graded algebras. 

\end{example}

\begin{example}
Let $G$ be a connected, semisimple, complex Lie group and $\mathfrak{g}$ its Lie algebra. Fix a Cartan subalgebra $\mathfrak{h}$ of $\mathfrak{g}$ and let $W$ be the Weyl group of $G$. Let $\mathcal{N}$ denote the nilpotent cone in $\mathfrak{g}$. The Springer resolution of $\mathcal{N}$ is $\pi : T^* \mathcal{B} \rightarrow \mathcal{N}$, where $\mathcal{B}$ is the flag variety. We fix $e \in \mathcal{N}$. Associated to $e$ is a Slodowy slice $e \in S \subset \mathfrak{g}$. The intersection $S_0 := S \cap \mathcal{N}$ is a conic symplectic singularity, and the restriction of $\pi$ defines a symplectic resolution $\widetilde{S}_0 := \pi^{-1}(S_0) \rightarrow S_0$. Quantizations of $S_0$ are given by finite $W$-algebras, which we denote $A_{\lambda}(e)$ to avoid confusion with our notation for DQ-algebras. Here $\lambda \in \mathfrak{h}^*$. Notice that $\mathcal{B} \cap \widetilde{S}_0$ is the Springer fiber $\mathcal{B}_e$ of $e$. 

Let $\mathfrak{l} \subset \mathfrak{g}$ be a minimal Levi subalgebra containing $e$. Recall that the element $e$ is said to be \textit{of standard Levi type} if it is regular in $\mathfrak{l}$ (this is independent of the choice of $\mathfrak{l}$). In type $A$, every nilpotent element is of standard Levi type. 

\begin{prop}
Let $e$ be of standard Levi type and $\lambda \in \mathfrak{h}^*$ regular. Then 
$$
HH^{\idot}(A_{\lambda}(e)) \simeq H^{\idot}(\mathcal{B}_e).
$$     
\end{prop}

\begin{proof}
If $e$ is of standard Levi type then it follows from \cite[Proposition 1]{FresseReg}, that there is a one parameter subgroup $H \subset G$ acting on $\widetilde{S}_0$ such that $\widetilde{S}_0^H$ is finite. Since this action is Hamiltonian, we may assume by twisting that the elliptic action of $\Cs$ on $\widetilde{S}_0$ has only finitely many fixed points.  

Therefore it suffices to check that, when $\lambda$ is dominant regular, localization holds (\ie there is a DQ-algebra $\cW_{\lambda}$ on $\widetilde{S}_0$ such that $A_{\lambda}(e)\text{-}\mathsf{mod} \simeq \Good{\cW_{\lambda}}$) and the enveloping algebra of $A_{\lambda}(e)$ has finite global dimension. This is well-known so we give the appropriate references. The statement about localization follows from \cite[Theorem 6.5]{DKr} or \cite[Theorem 6.3.2]{VictorSlodowy} and the fact that the enveloping algebra of $A_{\lambda}(e)$ has finite global dimension is a consequence of \cite[Proposition 6.5.1]{VictorSlodowy} and the fact that the abelian category $\cW-\good$ has finite global dimension for any DQ-algebra $\cW$ on $\widetilde{S}_0 \times \widetilde{S}_0$.
\end{proof}

In particular, when $e = 0$, we have $S = \mathfrak{g}$ and hence $S \cap \mathcal{N} = \mathcal{N}$ and $\widetilde{S}_0 = T^* \mathcal{B}$. In this case $A_{\lambda}(e)$ is a primitive central quotient of the enveloping algebra $U(\mathfrak{g})$ and our result recovers a result of Soergel \cite{SoergelHoch}, who also used localization (but coupled with the Riemann-Hilbert correspondence). 
\end{example}

\begin{example}
Let $M(r,n)$ be a framed moduli space of torsion free sheaves on $\mathbb{P}^2$ with rank $r$ and second Chern character $c_2 = n$. Specifically, $M(r,n)$ parameterizes isomorphism classes of pairs $(E,\phi)$ such that
\begin{enumerate}
\item The sheaf $E$ is torsion free of rank $r$ and $\langle c_2(E), [\mathbb{P}^2] \rangle = n$. 
\item The sheaf $E$ is locally
free in a neighborhood of $\ell^{\infty}$, with fixed isomorphism $\phi : E|_{\ell^{\infty}} \stackrel{\sim}{\longrightarrow} \mathcal{O}_{\ell^{\infty}}^{\oplus r}$. 
\end{enumerate}
Here $\ell^{\infty} = \{ [0:z_1:z_2] \in \mathbb{P}^2 \}$ is the line at infinity. The space $M(r,n)$ is isomorphic to the quiver variety associated to the framed Jordan quiver, with dimension vector $(r,n)$, see \cite{NakInstanton}. Let $M_0^{\mathrm{reg}}(r,n)$ be the open subset of locally free sheaves. The space $M(r,n)$ is a symplectic resolution of $M_0(r,n)$, where the latter is the Uhlenbeck partial compactification of $M_0^{\mathrm{reg}}(r,n)$. Quantizations $\mathcal{A}_c(r,n)$, for $c \in \C$, of $M_0(r,n)$ have been studied in \cite{LosevEtingofII}. 

\begin{prop}
Assume that $c$ is not of the form $\frac{s}{m}$, where $1 < m < n$ and $-rm < s < 0$. Then 
$$
HH^{\idot}(\mathcal{A}_c(r,n)) = H^{\idot}(M(n,r),\C). 
$$
\end{prop}

\begin{proof}
This follows from Proposition \ref{prop:chomology} and \cite[Theorem 1.1]{LosevEtingofII}, once one knows that $M(r,n)$ has finitely many fixed points under $\Cs$. But this follows from \cite[Theorem 3.7]{NakInstanton}, which shows that there is a natural torus $\mathsf{T}$ acting by Hamiltonian automorphisms on $M(r,n)$ such that the set $M(r,n)^{\mathsf{T}}$ is finite. 
\end{proof}

The above proposition implies that the graded dimension of $HH^{\idot}(\mathcal{A}_c(r,n))$ has a concise expression in terms of $r$-multipartitions of $n$; see \cite[Theorem 3.8]{NakInstanton}. 
\end{example}

\begin{example}
Our final example is slices in affine Grassmannians. We follow \cite{KWWY1,KWWY2}. Let $G$ be a complex semi-simple Lie group and $\Gr = G(\!(t^{-1})\!) / G[t]$ its thick affine Grassmannian. For any given pair of dominant coweights $\lambda \ge \mu$, we have Schubert subvarieties $\Gr^{\lambda}$ and $\Gr^{\mu}$ of $\Gr$ such that $\Gr^{\mu} \subset \overline{\Gr^{\lambda}}$. The intersection $\Gr^{\lambda}_{\mu} := \overline{\Gr^{\lambda}} \cap \Gr_{\mu}$, with $\Gr_{\mu}$ an orbit for the first congruence subgroup $G_1[\![t^{-1}]\!]$ of $G[\![t^{-1}]\!]$, is called the Lusztig slice. It is a finite dimensional affine conic symplectic singularity. If $\lambda$ is a sum of miniscule coweights, then it is shown in \cite[Theorem 2.9]{KWWY1} that $\Gr^{\lambda}_{\mu}$ admits a symplectic resolution $\widetilde{\Gr}^{\lambda}_{\mu}$, given by closed convolution of Schubert varieties associated to miniscule weights. If $T \subset G$ is a maximal torus, then $T$ acts Hamiltonian on both $\widetilde{\Gr}^{\lambda}_{\mu}$ and $\Gr^{\lambda}_{\mu}$, such that the resolution morphism is equivariant. It has been shown in \cite[Lemma 4.4]{KWWY2} that $(\widetilde{\Gr}^{\lambda}_{\mu})^T$ is finite. Therefore, one can twist the elliptic action of $\Cs$ on $\widetilde{\Gr}^{\lambda}_{\mu}$ so that it has only finitely many fixed points. Thus, Proposition \ref{prop:chomology} implies that if $U$ is a quantization of $\C[\Gr^{\lambda}_{\mu}]$ such that $U^e$ has finite global dimension and derived localization hold, then 
$$
HH^{\idot}(U) \simeq H^{\idot}(\widetilde{\Gr}^{\lambda}_{\mu},\C).
$$
Conjecturally, any such quantization is given by a quotient $Y^{\mu}_{\lambda}(\mathbf{c})$ of a shifted Yangian $Y^{\mu}$; see \cite[Conjecture 4.11]{KWWY1}. 
\end{example}

\bibliographystyle{alpha}

\end{document}